\documentclass{preprint}
\usepackage[full]{textcomp}
\usepackage[osf]{newtxtext}

\usepackage{amssymb}
\usepackage{amsfonts}
\usepackage{amsmath}
\usepackage[utf8]{inputenc}
\usepackage{mathabx}
\usepackage{mathrsfs}
\usepackage{mhequ}
\usepackage{microtype} 
\usepackage{bbm}
\usepackage[utf8]{inputenc}
\usepackage{xcolor}
\usepackage{verbatim}
\usepackage{amsthm}
\usepackage{enumitem}
\usepackage{comment}

\usepackage{hyperref}
\usepackage{breakurl}

\newtheorem*{mthm}{H\"{o}rmander's theorem}
\newtheorem{thm}{Theorem}[section]

\newtheorem{ass}{Assumption}
\newtheorem{ex}{Example}
\newtheorem{lem}[thm]{Lemma}
\newtheorem{prop}[thm]{Proposition}

\newtheorem{defn}[thm]{Definition}
\newtheorem{rem}[thm]{Remark}
\numberwithin{equation}{section}

\newcommand{\ddh}{\hat{\delta}}
\newcommand{\ddc}{\check{\delta}}
\newcommand{\vv}{\vvvert}
\newcommand{\Sym}{\mathrm{Sym}}
\newcommand{\HH}{\mathcal{H}}
\newcommand{\LL}{\mathcal{L}}
\newcommand{\DD}{\mathcal{D}}
\newcommand{\CC}{\mathcal{C}}
\newcommand{\CP}{\mathcal{P}}
\newcommand{\MM}{\mathcal{M}}
\newcommand{\FF}{\mathcal{F}}
\newcommand{\KK}{\mathcal{K}}
\newcommand{\II}{\mathcal{I}}
\newcommand{\CM}{\mathcal{CM}}
\newcommand{\R}{\mathbb{R}}
\newcommand{\N}{\mathbb{N}}
\newcommand{\X}{\mathbb{X}}
\newcommand{\B}{\mathbb{B}}
\newcommand{\Prob}{\mathbb{P}}
\newcommand{\E}{\mathbb{E}}
\newcommand{\cD}{\mathscr{D}}
\newcommand{\cA}{\mathscr{A}}

\newcommand{\cM}{\mathscr{M}}
\newcommand{\cC}{\mathscr{C}}
\newcommand{\rI}{\mathrm{I}}
\newcommand{\rII}{\mathrm{I\!I}}
\newcommand{\rIII}{\mathrm{I\!I\!I}}
\newcommand{\rIV}{\mathrm{I\!V}}
\newcommand{\rV}{\mathrm{V}}
\newcommand{\rVI}{\mathrm{V\!I}}
\newcommand{\rVII}{\mathrm{V\!I\!I}}
\newcommand{\rVIII}{\mathrm{V\!I\!I\!I}}

\newcommand{\ind}{\mathbbm{1}}

\colorlet{darkblue}{blue!90!black}
\colorlet{darkred}{red!90!black}

\def\eqref#1{(\ref{#1})}
\def\Poly{\mathrm{Poly}}

\def\XX{\mathbf{X}}
\def\id{\mathrm{id}}
\def\Strat{{\mathrm{Strat}}}
\def\Ito{{\mathrm{It\hat o}}}
\def\loc{{\mathrm{loc}}}
\def\pvar{{p\text{-}\mathrm{var}}}
\def\scal#1{\langle#1\rangle}
\def\curl{\mathop{\mathrm{curl}}\nolimits}
\def\tr{\mathop{\mathrm{tr}}\nolimits}
\def\back{\leftarrow}

\def\slash{\unskip\kern0.2em/\penalty\exhyphenpenalty\kern0.2em\ignorespaces}
\def\dash{\unskip\kern0.2em--\penalty\exhyphenpenalty\kern0.2em\ignorespaces}

\begin{document}

\author{Andris Gerasimovi\v cs,  Martin Hairer}
\date{7 November 2019}
\title{H\"{o}rmander's theorem for semilinear SPDEs}

\institute{Imperial College London, UK\\
\email{a.gerasimovics17@imperial.ac.uk, m.hairer@imperial.ac.uk}}

\maketitle
\begin{abstract}
We consider a broad class of semilinear SPDEs with multiplicative noise driven by a finite-dimensional Wiener process.
We show that, provided that an infinite-dimensional analogue of H\"ormander's bracket condition holds, the
Malliavin matrix of the solution is an operator with dense range. In particular, we show that the laws of 
finite-dimensional projections of such solutions admit smooth densities with respect to Lebesgue measure.
The main idea is to develop a robust pathwise solution theory for such SPDEs using rough paths theory,
which then allows us to use a pathwise version of Norris's lemma to work directly on the Malliavin matrix,
instead of the ``reduced Malliavin matrix'' which is not available in this context.

On our way of proving this result, we develop some new tools for the theory of rough paths like
a rough Fubini theorem and a deterministic mild Itô formula for rough PDEs.\\[.5em]
\textbf{Keywords:} Rough paths, rough PDEs, rough Fubini theorem, H\"{o}rmander's condition. 
\end{abstract}

\tableofcontents

\section{Introduction}

The goal of this paper is to generalise the series of articles \cite{hairer2006ergodicity,bakhtin2007malliavin,hypoelipticity}
where the authors developed Malliavin calculus for semilinear stochastic partial differential equations (SPDEs) with 
additive degenerate noise and showed non-degeneracy of the Malliavin matrix under H\"ormander's bracket
condition. The main novelty of the present article is that we are able to extend these results to 
equations driven by multiplicative noise. 
In particular, we conclude that finite-dimensional projections of the solutions admit densities
and, provided that suitable a priori bounds are satisfied, that the corresponding Markov semigroup 
satisfies the asymptotic strong Feller property introduced in \cite{hairer2006ergodicity}. 
This can be understood as a genuinely infinite-dimensional ``smoothing property'' for the Markov 
semigroup that holds at infinite time. 

Equations considered in this article can formally be written in Stratonovich form as
\begin{equ} 
du_t = Lu_tdt +N(u_t)dt + \sum_{i=1}^d F_i(u_t)\circ dB^i_t,\quad u_0 \in \HH, \label{SPDE1}
\end{equ}
where $L$ is a negative definite selfadjoint operator on a separable Hilbert space $\HH$, $N,F_i$ are 
smooth non-linearities, and $B_t = (B^1_t,B^2_t,\ldots ,B^d_t)$ is a standard $d$-dimensional Brownian motion. We potentially
allow $N$ to lose derivatives and can consider for instance the 2D Navier-Stokes equations. Reaction-diffusion equations and the Cahn-Hilliard equation 
also fall in the category of equations that we consider. 
Setting $F_0(u) = Lu + N(u)$, we will make regularity assumptions guaranteeing that all iterated Lie brackets
of the $F_i$ can be given a canonical meaning, so that we can formulate an infinite-dimensional
version of H\"ormander's condition. Recall that Lie brackets
are formally given by $[F_i,F_j](u) = DF_j(u)F_i(u) - DF_i(u)F_j(u)$.

While the Malliavin matrix is not invertible on the whole space $\HH$, 
we will show that it is invertible on every finite-dimensional subspace. 
(See \cite{daprato_erg} for the exceptional case when Malliavin 
matrix of the linear equation is invertible on the whole space, see also \cite{eckmann,daprato_elworty,flandoli,cerrai} for 
situations where the Malliavin matrix 
is invertible on the image of the Jacobian.) 
Note that the situation considered in this article is orthogonal to the one considered in 
\cite{hairer2016strong,cannizzaro2017malliavin}.
There, the authors considered a situation in which the noise already acts in a ``full'' way 
in every direction of the state space, so that no Lie brackets need to be considered.
Instead, the problem addressed there is that solutions can be very singular, so that sophisticated solution
theories need to be considered, which do not interplay nicely with Malliavin calculus.

We now state the main result of the present article:
\begin{mthm}
	Let $T>0$ let $0\leq \delta < 2/3$ and $N: \HH \to \HH_{-\delta}$, and $F_i : \HH \to \HH$ be $C^\infty$ vector fields of polynomial type satisfying H\"{o}rmander's condition, Assumption~\ref{ass2}. Assume that $u$ is a global mild solution of the equation~(\ref{SPDE1}) such that both $\|u\|_{L^\infty({[0,T],\HH})}$ and $\|J\|_{L^\infty({[0,T],\LL(\HH,\HH)})}$ have moments of all orders. (here $J$ is Jacobian of the solution). Then for every finite rank orthogonal projection $\Pi: \HH \to \HH$, $a \in (0,1)$ and  every $p\geq 1$ there exist a constant $C_p$ such that the Malliavin matrix $\cM_T$ satisfies the following bound for every initial condition $u_0$:
	 \begin{equ}
	 	\Prob\bigg(\,\inf_{\|\Pi\varphi\| > a\|\varphi\|}\frac{\langle\cM_T\varphi, \varphi \rangle}{\|\varphi\|^2} \leq \varepsilon\bigg) \leq C_p \varepsilon^p.
	 \end{equ}
 	Moreover the law of $\Pi u_T$ has a smooth density with respect to Lebesgue measure on $\Pi(\HH)$.
\end{mthm}

The classical approach to proving a statement of this type was
initiated by Malliavin in \cite{Mal76} and further developed and refined by a number of 
authors in the eighties \cite{Bismut81,KSAMI,KSAMII,KSAMIII,norris1986simplified}. See also
\cite{Mal97SA,Nual:06,ProofShort} for surveys of a more expository nature.
The argument goes by contradiction: assume that $\langle\cM_T\varphi,\varphi\rangle$ 
is small (in a suitable probabilistic sense) and use this as the starting point for a chain of implications that eventually lead to an impossibility, resulting in the
conclusion that the probability of $\langle\cM_T\varphi,\varphi\rangle$ being small is (very) small. 
First note that the Jacobian $J_{t,s}$ of equation~(\ref{SPDE1}) solves 
\begin{equ} 
	dJ_{t,s} = LJ_{t,s} dt + DN(u_t)J_{t,s}dt + \sum_{i=1}^dDF_i(u_t)J_{t,s}\circ dB^i_t,\quad J_{s,s} = \id\;,\label{Jacob}
\end{equ}
and that Duhamel's formula yields the following expression for the Malliavin matrix: 
\begin{equ} 
	\langle\cM_T\varphi,\varphi\rangle = \sum_{i=1}^d \int^T_0 \langle J_{T,s}F_i(u_s),\varphi\rangle^2 ds. \label{mal_matrix1}
\end{equ}
Since we assumed that $\langle\cM_T\varphi,\varphi\rangle$ is small, this implies that each $\langle J_{T,s}F_i(u_s),\varphi\rangle$ is small too. Formally differentiating such an expression, we obtain
\begin{equ} 
	d\langle J_{T,s}G(u_s),\varphi\rangle = \langle J_{T,s}[F_0,G](u_s),\varphi\rangle ds + \sum_{j=1}^d\langle J_{T,s}[F_j,G](u_s),\varphi\rangle\circ dB^j_s\;, \label{formal_diff1}
\end{equ}
where we used Stratonovich integration to avoid the appearance of It\^{o}'s correction. Norris's lemma 
\cite{norris1986simplified} (see also \cite{ProofShort} for a version that is slightly easier to parse) then 
allows to conclude that if $\langle J_{T,s}G(u_s),\varphi\rangle$ is small for a ``nice enough'' 
function $G$, then all 
$\langle J_{T,s}[F_j,G](u_s),\varphi\rangle$ for $j \ge 0$ are small as well. 
We can iterate this procedure and define recursively 
$$\cA_0  = \{F_i\, :\, 1\leq i \leq d\}\;,\quad \cA_{k+1}  =\cA_k \cup \{[F_i,A]\, :\, A \in \cA_k,\, 0\leq i \leq d\}\;.$$ 
Then the above argument tells us that for all $k \in \N_0$ and $A \in \cA_k$ the quantity $\langle J_{T,s}A(u_s),\varphi\rangle$ is small. This contradicts H\"{o}rmander's condition, namely that for every $v \in \HH$ the set $\bigcup_{k \in \N_0} \{A(v): A \in \cA_k\}$ is dense in $\HH$. This is because by density we can necessarily
find $A \in \cA_k$ to make the quantity $\langle J_{T,s}A(u_s),\varphi\rangle$ of order one, thus
concluding the argument.

The problem with this argument is that it is not clear what the meaning of the stochastic integral 
appearing in~(\ref{formal_diff1}) is, since the process 
$\langle J_{T,s}[F_j,F_i](u_s),\varphi\rangle$ is not adapted, hence not It\^{o} (nor \textit{a fortiori} Stratonovich) integrable. In a similar vein, 
Norris's lemma applies only to semimartingales, which $\langle J_{T,s}[F_j,F_i](u_s),\varphi\rangle$ is not. 
In finite dimension, this problem has traditionally been 
circumvented by considering instead the \textit{reduced Malliavin matrix} $\hat{\cM}_T$:
$$\cM_T = J_{T,0}\hat{\cM}_TJ^*_{T,0}\;,\quad  \langle\hat{\cM}_T\varphi,\varphi\rangle = \int^T_0 \langle J^{-1}_{s,0}F(u_s),\varphi\rangle^2 ds\;.$$
This is possible in finite dimension because solutions typically 
generate an invertible flow whose Jacobian factorises as $J_{t,s} = J_{t,0}J^{-1}_{s,0}$. 
A similar calculation then yields an expression analogous to~\eqref{formal_diff1}, but this time
the stochastic integrand is of the form $\langle J^{-1}_{s,0}F_i(u_s),\varphi\rangle$,
which is a semimartingale. With the help of this trick, the argument sketched above can be made rigorous and entails
the non-degeneracy of the Malliavin matrix, which in turn implies that the law of the 
solution has a smooth density with respect to Lebesgue measure. 

In 
infinite dimension, the solution to SPDEs rarely produces a flow, so that this trick cannot be used
in general. (But see \cite{teichmann} for a situation where it can be used.) For the special case of additive noise with a `polynomial' nonlinearity $N$, \cite{hypoelipticity}
were able to use an alternative to Norris's lemma \dash a certain non-degeneracy bound 
on Wiener polynomials \dash but this approach seems to be of little use in the case of multiplicative noise. 
The idea implemented in the present article is to use the theory of rough paths in order to give meaning 
to~(\ref{formal_diff1}) directly and to be able to exploit the `deterministic' version of Norris's lemma 
for rough paths from \cite{hairer2013regularity}.

The theory of rough paths provides a pathwise approach to a stochastic integration and was originally developed by 
Lyons \cite{lyons98,lyons02} building on the works of Young and Chen \cite{young,Chen}.
The 
idea is that, in order to solve (finite-dimensional) equations of the type 
$dY_t = F(Y_t)dX_t$ with $X \in\CC^\gamma$ for $\gamma < 1/2$, one augments $X$ with a function $\X_{t,s}$ for 
$t\geq s$ that \textit{postulates} the values of the integrals $\int_s^t \delta X_{r,s}\,dX_r$ 
(we write $\delta X_{t,s} = X_t - X_s$) and that satisfies the bound $|\X_{t,s}| 
\lesssim |t-s|^{2\gamma}$ consistent with the regularity of $X$, as well as the algebraic identity
$\X_{t,s} - \X_{u,s} - \X_{t,u} = \delta X_{t,u}\otimes \delta X_{u,s}$. The pair $(X,\X)$ is then called a 
\textit{rough path}. Once $\X$ is given, integrals of the form $\int Y_t\,dX_t$ can be defined 
in a consistent way for a
class of integrands $Y$ that locally ``look like $X$ at small scales'', see \cite{controlled} where this
notion was introduced. 
Formally, this can be expressed as
\begin{equ}[e:GubSDE]
\delta Y_{t,s} = Y'_s\,\delta X_{t,s} + R^Y_{t,s}\,,
\end{equ}
where $Y' \in\CC^\gamma$ is the `Gubinelli derivative' of $Y$ and the remainder $R^Y$ satisfies $|R^Y_{t,s}|\lesssim |t-s|^{2\gamma}$.
One then sets
\begin{equ}
\int_0^T Y_t\,dX_t = \lim_{|\CP| \to 0} \sum_{[s,t] \in \CP} \bigl(Y_s \delta X_{t,s} + Y'_s \X_{t,s}\bigr)\;,
\end{equ}
where $\CP$ denotes a partition of $[0,T]$ by intervals and $|\CP|$ the length of its largest
element.
It turns 
out that there exists a canonical lift $X \mapsto (X,\X)$ for a wide range of stochastic processes, 
including Brownian motion (with $\X$ defined by Stratonovich integration), fractional Brownian 
motion and other Gaussian processes.

In \cite{evolution} Gubinelli and Tindel generalised theory of rough paths to solve not only SDEs but also SPDEs: 
evolution equations driven by the infinite dimensional Gaussian process. For that, they introduce 
operator-valued rough paths and use a slightly different kind of 
local (in time) expansion of the controlled processes, taking into account the solution to the linearised 
equation. This means that we no longer compare $Y_t$ to $Y_s$ at small scales, 
but instead to $e^{L(t-s)}Y_s$. More formally, we 
replace \eqref{e:GubSDE} by an expansion of the type 
\begin{equ} 
	Y_t-e^{L(t-s)}Y_s = e^{L(t-s)}Y'_s\,\delta X_{t,s} + R^Y_{t,s}. \label{RP_intro}
\end{equ}
Since in our case the driving noise is finite-dimensional, we use similar ideas 
to \cite{evolution}, but then stick closely to the classical theory of finite-dimensional 
rough paths as in \cite{friz}. The main difference and complication arises when one wants 
to show that if $Y$ satisfies an expansion like~(\ref{RP_intro}) then so does $F(Y)$ for any smooth enough 
function $F$. This requires an estimate on $F(Y_t) -e^{L(t-s)}F(Y_s)$ while only having a good bound on
$F(Y_t) -F(e^{L(t-s)}Y_s)$, thus requiring commutator bounds of the type 
$$
	\|e^{L(t-s)}F(Y_s) -F(e^{L(t-s)}Y_s)\| \lesssim |t-s|^{2\gamma},
$$ 
which is possible for instance if $Y_s$ itself has better space regularity. We therefore need
to obtain bounds on the space regularity 
of the path $Y_s$ that are better than the space regularity in which we measure the rough path norms.

One of the main technical difficulties we encounter is to prove that~(\ref{formal_diff1}) holds. An obstacle is that we cannot simply 
differentiate $\scal{J_{T,s}G(u_s),\varphi}$ because rough path theory only allows us to use a mild formulation of the solution to~(\ref{SPDE1}). 
This however turns out to be sufficient once we obtain a rough Fubini Theorem and a mild version of It\^o's formula for $G(u_s)$.
Once we obtain~(\ref{formal_diff1}), we follow closely the approach from \cite{hairer2013regularity}, making use of the rough Norris lemma. 
We try in most cases to work with general rough paths, not just the one lifted from Brownian motion, so that
part of our results carry over immediately to SPDEs driven by fractional Brownian motion for example. 
We do however show that in the Brownian case the solutions constructed here coincide with those obtained
from It\^o calculus, which connects our result with existing objects and allows us to exploit information known for 
the solutions to It\^o SPDEs like Malliavin differentiability, a priori bounds and global existence.
Such information might be much harder to obtain for more general Gaussian rough paths.
We want to emphasise again that once we translate our problem to the language of rough paths, most of the 
arguments are deterministic. We will only use probabilistic tools (and very basic ones at that) in the proof of 
H\"{o}rmander's theorem itself and in order to obtain global well-posedness of solutions.

\textit{Outline of the article:}
In Section~\ref{prelim}, we introduce a reduced increment $\ddh $ and reduced H\"{o}lder spaces as well as 
a version of the sewing lemma from \cite{evolution} for this reduced increment. Section~\ref{CRP} gives a self-contained 
introduction to the spaces of controlled rough paths with the semigroup and how composition with regular functions 
preserves these spaces. We also describe an integration in these spaces with respect to rough path which follows directly 
from the sewing lemma. Section~\ref{RPDEs} is devoted to the solution theory, continuity of the solution map and the 
properties of the solution. In particular, we show in Section~\ref{SPDEs} that the solutions obtained by viewing
\eqref{SPDE1} as an RPDE and driving it by the Stratonovich lift of Brownian motion coincide almost surely with the solutions
constructed using classical stochastic calculus as in \cite{daprato} for example. 
Section~\ref{fubini} is about the proof of rough Fubini theorem. In Section~\ref{weak&ito}  
we show equivalence of mild solutions and weak solutions. We later use this in order to show a mild It\^{o} formula. 
Section~\ref{R_BPDE} talks about the backwards equations. There we provide an equation for the adjoint of the Jacobian and also prove the differentiation statement~(\ref{formal_diff1}). Finally in Section~\ref{spectral} we recall the ``rough Norris lemma'' and
combine it with the previous results to prove in Theorem~\ref{main} the H\"{o}rmander-type theorem announced in the introduction.
We also show in Theorem~\ref{densities} how this immediately yields smooth densities for finite-dimensional marginals of the solution.

\subsection*{Acknowledgements}

{\small
MH gratefully acknowledges support by the Leverhulme Trust through a leadership award and by 
the ERC through a consolidator grant, project 615897 (Critical). 
}

\section{Preliminaries} \label{prelim}

\subsection{Semigroup theory}\label{prelim_semigroup}

Throughout this paper we consider a separable Hilbert space $\HH$ with inner product $\langle \cdot, \cdot \rangle$ and a negative definite selfadjoint operator $L$ such that there exists some constant $c< 0$ such that $\langle u, Lu \rangle \leq c \langle u, u \rangle$. We write $(S_t)_{t \geq 0}$ as well as $e^{Lt}$ for the semigroup generated by $L$. 
For $\alpha\geq 0$, the interpolation space $\HH_\alpha = \mathrm{Dom}((-L)^\alpha)$ is a Hilbert space when 
endowed with the norm $\|\cdot\|_{\HH_\alpha} = \|(-L)^\alpha\cdot\|_{\HH}$. Since $(-L)^{-\alpha}$ is  
bounded on $\HH$, $\|\cdot\|_{\HH_\alpha}$ is equivalent to the graph norm of $(-L)^\alpha$.
Similarly, $\HH_{-\alpha}$ is defined as the completion of $\HH$ with respect to the norm 
$\|\cdot\|_{\HH_{-\alpha}}  = \|(-L)^{-\alpha}\cdot \|_\HH$.

For any $\alpha,\beta \in \R$, denote the space of bounded operators from $\HH_\alpha$ to $\HH_\beta$ by 
$\LL^{\alpha,\beta} := \LL(\HH_\alpha;\HH_\beta)$ and write $\LL^\alpha := \LL^{\alpha,\alpha}$. 
We define a reduced semigroup operator $\tilde{S}_t = S_t - \id$. 

Defining $\HH_\infty = \bigcap_{\alpha} \HH_\alpha$ 
and $\HH_{-\infty} = \bigcup_{\alpha} \HH_\alpha$ the operators $S_t$ map $\HH_{-\infty}$ to $\HH_\infty$ for every 
$t > 0$ and $\HH_{\infty} \subset \HH_{-\infty}$ densely.
We will use extensively the fact that, for every $\alpha\geq \beta$ and every $\gamma \in [0,1]$, one has
\begin{equ}[semigroup]
	\|S_t u\|_{\HH_\alpha} \lesssim t^{\beta-\alpha} \|u\|_{\HH_\beta}\;,\qquad 
	\|\tilde{S}_t u\|_{\HH_{\beta-\gamma}} \lesssim t^{\gamma} \|u\|_{\HH_{\beta}}\;, 
\end{equ}
uniformly over $t \in (0,1]$ and $u \in \HH_\beta$. For an introduction to analytic semigroup theory, see for example \cite{Pazy,spde}.

\subsection{Increment spaces}

We now define spaces of time increments of functions taking values in some Banach space. 
We follow closely the definitions in \cite{evolution,deya2012non,deya2013malliavin}. 
Fix $T>0$ and, for $n\in \N$, define the $n$-simplex  
$\Delta_n=\{(t_1,\ldots ,t_n): T\geq t_1\geq t_2 \geq \ldots  \geq t_n\geq 0\}$. 
We will often omit the fact that spaces depend 
on $T$ since its precise value is not relevant.
\begin{defn}
	Given a Banach space $V$, $n \in \N$ and $T>0$ define $\CC_n(V) := \CC(\Delta_n, V)$ the
	space of continuous functions from $\Delta_n$ to $V$ and $\delta : \CC
	_{n-1}(V) \to \CC_{n}(V)$ by 
$$
	\delta f_{t_1t_2\ldots t_n}=\sum_i^d(-1)^if_{t_1\ldots \hat{t}_i\ldots t_n}\;,
$$
where $\hat{t}_i$ indicates that the corresponding argument is omitted.
\end{defn}
We are mostly going to use the two special cases 
\begin{equ}
\delta f_{t,s} = f_t-f_s\;,\qquad \delta g_{t,u,s} = g_{t,s}-g_{t,u}-g_{u,s}\;.
\end{equ}
One can check that $\delta\delta = 0$ as an operator $\CC
_{n-1}(V) \to \CC_{n+1}(V)$ and that for each $f \in \CC
_{n}(V)$ such that $\delta f = 0$ there exists $g\in \CC_{n-1}(V)$ such that $f = \delta g$. 

\begin{defn}
For $V$ either the space $\LL^{\alpha,\beta}$ or $\HH_\alpha$ for $\alpha,\beta \in \R$, the reduced increment operator $\hat{\delta} :  \mathcal{C}_{n-1}(V) \to \mathcal{C}_{n}(V)$ 
is given by $\hat{\delta}f = \delta f - \tilde{S}f$,
where $(\tilde{S}f)_{t_1\cdots t_n} = \tilde{S}_{t_1-t_2}f_{t_2\cdots t_n}$ with $\tilde S_t = S_t - \id$. 
\end{defn}

Again, the two most common cases will be
\begin{equ}
		\ddh  f_{t,s} = f_t - S_{t-s}f_s\;,\qquad \hat{\delta }g_{t,u,s} = g_{t,s}-g_{t,u} - S_{t-u}g_{u,s}\;.
\end{equ}
Whenever we talk about $\ddh $ on $\CC_n$ we will assume from now on that the underlying space $V$ is one 
of the spaces on which the action of the semigroup $S$ makes sense. 
Similarly to $\delta$, one verifies that $\ddh \ddh  = 0$ and that $\ddh  f = 0$ 
implies that $f = \ddh  g$ (see \cite{evolution}).

\subsection{H\"{o}lder type spaces}

\begin{defn}
Let $V$ be a Banach space and denote by $\|\cdot\|_V$ the corresponding norm. Then, for $\gamma,\mu > 0$, $n \geq 2$ 
and $f \in \CC_n(V)$, we set
\begin{equ}[e:defNormV]
|f|_{\gamma,V} = \sup_{t \in \Delta_n} {\|f(t)\|_V \over |t_n - t_1|^\gamma}\;.
\end{equ}
We then define the spaces and notations
\begin{equs}[2]
	\CC_n^\gamma &= \{f \in \CC_n\,:\, |f|_{\gamma,V} < \infty\}\;,\quad&
	\hat \CC_n^{\gamma,\mu} &= \{f \in \CC_n^\gamma\,:\, \hat \delta f \in  \CC_{n+1}^\mu\}\;,\\
	\CC^\gamma &= \{f \in \CC_1\,:\, \delta f \in \CC_{2}^\gamma \}\;,\quad& \hat \CC^\gamma &= \{f \in \CC_1\,:\, \hat \delta f \in \CC_{2}^\gamma\}\;.
\end{equs}
\end{defn}
Since~\eqref{e:defNormV} doesn't make any sense for $n=1$, we make an abuse of notation 
by writing $|f|_{\gamma,V}$ for $|\delta f|_{\gamma,V}$ for $f \in \CC^\gamma$. Later on, it will be clear 
from context whether we use $|\cdot|_{\gamma,V}$ as in~\eqref{e:defNormV} or as the seminorm on 
$\CC^\gamma$. Similarly, we define a seminorm on $\hat \CC^\gamma$ by $\|f\|_{\gamma,V} = |\ddh f|_{\gamma,V}$ and we endow $\CC_1$ 
with the supremum norm $\|f\|_{\infty,V} = \sup_{{0\leq s \leq T}} 
\|f_s\|_V$. Finally we equip $\CC^\gamma$ and $\hat \CC^\gamma$ with norms $\|f\|_{\CC^\gamma V} = \| f_0\| _{V}+|f|_{\gamma,V}$ and $\| f\| _{\hat{\CC}^\gamma V} = \| f_0\|_{V}+\| f\| 
_{\gamma,V}$.

In the case $V = \HH_\alpha$, we will write $\|f\|_{\gamma,\HH_{\alpha}} = \|f\|_{\gamma,\alpha}$, $\|f\|_{\infty,\HH_{\alpha+2\gamma}} = \|f\|_{\infty,\alpha+2\gamma}$, etc.
An important feature of elements $\Xi \in\hat{\CC}^{\gamma,\mu}_2V$ is that they can be ``integrated'' in the sense 
that $\Xi_{t,s}$ `almost' looks like $\ddh F_{t,s}$ for some function $F\in \hat{\CC}^\gamma V$. 
More precisely, one has the following version of the sewing lemma.
\begin{thm}[Sewing Lemma] \label{sewing}
 Let $\alpha \in \R$ and let  $0<\gamma\leq1<\mu$. Then there exist a unique continuous linear map $\II: \hat{\CC}^{\gamma,\mu}_2\HH_\alpha \to \CC_2^{\gamma}\HH_{\alpha}$ such that $\hat\delta\II\Xi = 0$ and
\begin{equ} 
\| \II\Xi_{t,s}-\Xi_{t,s}\| _{\HH_{\alpha}} \lesssim | \ddh \Xi| _{\mu,\alpha}|t-s|^{\mu}.\label{e:sewing1}
\end{equ}
If in addition $\ddh \Xi_{v,m,u} = S_{v-m}\tilde{\Xi}_{v,m,u}$ for some $\tilde{\Xi} \in \CC_3 \HH_\alpha$ such that there exists $M > 0$ with
\begin{equ} 
	\|\tilde{\Xi}_{v,m,u}\|_{\HH_\alpha} \leq M |v-m|^{\mu-1}|v-u|\;, \label{e:sewing2}
\end{equ}
then for every $\beta \in [0,\mu)$ the following inequality holds: 
\begin{equ}
\| \II\Xi_{t,s}-\Xi_{t,s}\| _{\HH_{\alpha+\beta}} \lesssim_{\mu,\beta} M |t-s|^{\mu-\beta}.  \label{e:sewing3}
\end{equ}
Finally, one has the identity
\begin{equ}[e:approximation]
\II\Xi_{t,s} = \lim_{|\CP|\to 0} \sum_{[u,v]\in \CP}S_{t-v}\Xi_{v,u}\;,
\end{equ}
where $|\CP|$ denotes the length of the largest element of a partition $\CP$ of $[s,t]$ into non-overlapping closed intervals. 
The same is true if we replace $\HH_\alpha$ by $\HH_\alpha^n$.
\end{thm}
\begin{proof}
The proof is almost identical to that of \cite[Lemma 4.2]{friz}, so we only focus on the details that differ.
We first show that the limit~(\ref{e:approximation}) exists over the dyadic partition: let $\CP_0 = \{[s,t]\}$ and recursively set 
$$
\CP_{n+1}  = \bigcup_{[u,v]\in \CP_n}\{[u,m],[m,v]\}\;,
$$
where $m = (u+v)/2$, so that $\CP_n$ contains $2^n$ intervals of length $2^{-n}|t-s|$.
We then define an approximation to $\II  \Xi$ by: 
$$\II ^{n+1}\Xi_{t,s} := \sum_{[u,v]\in \CP_{n+1}}S_{t-v}\Xi_{v,u} = \II ^{n}\Xi_{t,s} \quad -  \sum_{[u,v]\in \CP_n}S_{t-v}(\ddh \Xi)_{v,m,u}\;,
$$
with $m$ the midpoint between $u$ and $v$ as before.
We focus on~(\ref{e:sewing3}) under assumption~(\ref{e:sewing2}) since showing~(\ref{e:sewing1}) is even closer
to \cite[Lemma 4.2]{friz}. We assume $\ddh \Xi_{v,m,u} = S_{v-m}\tilde{\Xi}_{v,m,u}$ and choose $\delta \geq 0$ 
such that $\mu - 1 > \delta > \beta -1$. Using the semigroup smoothing property~(\ref{semigroup}) we then have:
\begin{equs}
\| \II ^{n+1}&\Xi_{t,s}- \II ^{n}\Xi_{t,s}\| _{\HH_{\alpha+\beta}} \le \Big\|\sum_{[u,v]\in \CP_n}S_{t-m}\tilde{\Xi}_{v,m,u}\Big\|_{\HH_{\alpha+\beta}}\\
& \lesssim M\sum_{[u,v]\in \CP_n} |t-m|^{-\beta} |v-m|^{\mu-1} |v-u| \\
&\le M\sum_{[u,v]\in \CP_n} |t-m|^{-\beta+\delta}|v-m|^{\mu-1-\delta}|v-u|\\
&\le M 2^{n(1-\mu + \delta)} |t-s|^{\mu-\delta}\sum_{[u,v]\in \CP_n} |t-m|^{-\beta+\delta}2^{-n} \\
&\le M |t-s|^{\mu-\delta-1}2^{n(1-\mu+\delta)}\int_s^t |t-r|^{\delta -\beta}\,dr\\
&\lesssim M|t-s|^{\mu-\beta}2^{n(1-\mu+\delta)}\;.
\end{equs}

Going from second to the third line we used that, by convexity of the integrand, the Riemann sum is bounded by 
the integral. In the last inequality we used that $\int_0^1r^{-\beta+\delta}dr < \infty$ since $\beta-\delta<1$.
Since $\delta$ is chosen so that $1-\mu + \delta<0$, this is summable and yields desired bound \eqref{e:sewing3}. 

It may appear a priori that we only have $\II \Xi \in \CC_2$, but a similar argument to \cite{controlled,friz} shows that
actually \eqref{e:approximation} holds, which immediately implies that
$\ddh (\II \Xi) = 0$ as desired.
The fact that $\II \Xi \in {\CC}^{\gamma}_2\HH_{\alpha}$ follows easily by taking $\beta = 0$ and noting that we then have $\| \II \Xi_{t,s}\| _{\HH_{\alpha}} \lesssim |\Xi|_{\gamma,\alpha}|t-s|^\gamma + | \ddh \Xi| _{\mu,\alpha}|t-s|^{\mu}$. The continuity of $\II $ 
follows exactly as in \cite{controlled,friz} and is left to the reader.
\end{proof}

%
%

\section{Controlled Rough Paths according to the Semigroup} \label{CRP}

We now recall the notion of rough path introduced by Lyons in the 90's (see for example \cite{lyons98} or \cite{lyons02}). To treat SPDEs in 
Hilbert spaces, we could use an operator-valued definition of rough path as in \cite{evolution,Ismael}. However, we will focus on equations
driven by finite-dimensional Brownian motion and we would like to reuse already known results like Norris's Lemma or Malliavin calculus for rough paths of finite dimensions. We will therefore pursue a compromise and use the ``classical'' definition of a rough path for our driving noise,
while slightly modifying the notion of a ``controlled rough path'' from \cite{evolution} to encode the interaction of our class of
integrands with the semigroup.
\begin{defn}[Rough Path]
We say that a pair of functions $(X,\X) \in \CC_2(\R^d) \times \CC_2(\R^{d \times d})$  satisfies 
Chen's relations if for all $s \le u \le t$:
\begin{equs} 
	\delta X_{t,u,s} &= X_{t,s}- X_{t,u}-X_{u,s}= 0\;,\\
	\delta\X_{t,u,s} &=\X_{t,s}-\X_{t,u}-\X_{u,s}=  X_{t,u}\otimes X_{u,s}\;. \label{chen}
\end{equs} 
For $\gamma \in (1/3,1/2]$ and for two such pairs $\XX = (X,\X)$, $\tilde{\XX} = (\tilde{X},\tilde{\X})$ we define the rough path 
metric $\varrho_\gamma$ as: 
\begin{equ}
	\varrho_\gamma(\XX,\tilde{\XX}) = |X-\tilde{X}|_\gamma + |\X-\tilde{\X}|_{2\gamma}.
\end{equ}
Finally for $\gamma \in (1/3,1/2]$ we define the space of rough paths $\cC^\gamma([0,T], \R^d)$ to be the completion with respect to $\varrho_\gamma$ of all smooth pairs $(X,\X) \in C^\infty(\Delta_2, \R^d)\times C^\infty(\Delta_2, \R^{d\times d})$ satisfying~\eqref{chen}.
\end{defn}
For simplicity we write $\varrho_\gamma(\XX) := \varrho_\gamma(0,\XX)$. Note that the convergence with respect to above metric is implying the pointwise convergence thus Chen's relation is true for elements of $\cC^\gamma([0,T], \R^d)$. Here instead of writing $|X|_{\gamma,\R^d}$ and $|\X|_{2\gamma,\R^{d\times d}}$ we made an abuse of notation by simply writing 
$|X|_\gamma$ and $|\X|_{2\gamma}$ and hope that no confusion will arise from this. 

The first equation in Chen's relation~\eqref{chen} actually tells us that $X$ belongs to $\CC_1$ in a sense that we can write $X_{t,s} = \delta (X_{\cdot, 0})_{t,s}$ and so $X$ is completely determined by $X_{\cdot,0} \in \CC_1$. We decide not to use $\CC_1$ in the definition of the rough path since in our analysis we only care about the increments of functions and not about their precise value. Nevertheless we might sometimes neglect this and talk about $X$ as a one time parameter function.
One should think of $\X^{i,j}_{t,s}$ as \textit{postulating} the value of the integral $\int_s^t X^i_{u,s}dX_u^j$ which may 
not be defined classically through the theory of Young's integration \cite{young} since for that we need $\gamma > 1/2$
in general. This motivates us to define a canonical lift of the smooth path to the rough paths and the definition of the geometric rough paths:

\begin{defn} 
	For every $X \in C^\infty([0,T],\R^d)$ define the canonical lift of $X$ to the space of rough paths $\XX^c(X) = (\delta X, \X^c)$,  where $\X^c_{t,s} = \int_s^t\delta X_{u,s} \otimes dX_u$ and the right hand side is a Riemann integral.
	
	For $\gamma \in (1/3,1/2]$ we 
	say that the rough path $\XX \in  \cC^\gamma([0,T], \R^d)$ is geometric if it satisfies 
	\begin{equ}[geometric]
		2\, \Sym(\X_{t,s})= X_{t,s} \otimes X_{t,s}\; ,
	\end{equ}
	which always holds for the canonical lift of a smooth path. We write $\cC^\gamma_g([0,T], \R^d) \subset \cC^\gamma([0,T], \R^d)$ for the subspace of geometric rough paths.
\end{defn}
One can show that the space of geometric rough paths is the closure of all the smooth lifts with respect to the rough path metric $\varrho_\gamma$.

Equations of interest to us are driven by $\sum_{i=1}^{d}F_i(u_t)dX^i_t$, where $F_i:\HH_\alpha\to \HH_\beta$ 
and $X = (X^1,\ldots,X^d)$ is a rough path. We will typically use instead the shorthand notation $F(u_t)dX_t$, 
where we view $F:\HH_\alpha \to \LL(\R^d,\HH_\beta)$. For simplicity we will denote the space $\LL(\R^d,\HH_\beta)$ by $\HH^d_\beta$. 
With this notation, our spaces of integrands are defined as follows.
\begin{defn}
	Let $\XX \in \cC^{\gamma}([0,T],\R^d)$ for some $\gamma\in (1/3,1/2]$ and let $m\in \N$. We say that $(Y,Y') \in \hat{\CC}^\gamma \HH^{m}_\alpha \times \hat{\CC}^\gamma \HH^{m\times d}_\alpha$ is controlled by $\XX$ according to the semigroup $S$ if the remainder term $R^Y$ defined by
	\begin{equ} 
		R^Y_{t,s}=\ddh  Y_{t,s} - S_{t-s}Y'_s X_{t,s}\;, \label{controlled}
	\end{equ}
belongs to $\CC^{2\gamma}_2\HH_\alpha^m$. We then write $(Y,Y') \in \cD^{2\gamma}_{S,X}([0,T],V)$ and define a seminorm on 
this space by: 
 \begin{equ}
	\| Y,Y'\| _{X,2\gamma,\alpha} =  \| Y'\| _{\gamma,\alpha} + |R^Y|_{2\gamma, \alpha}.
\end{equ}
Similarly, its norm is given by
$$\| Y,Y'\| _{\cD^{2\gamma}_{S,X}} = \| Y_0\| _{\HH_\alpha^m} + \| Y'_0\| _{\HH_\alpha^{m\times d}}+	\| Y,Y'\| _{X,2\gamma,\alpha}.$$
\end{defn}

\begin{rem}
In the special case $S = \id$, this is nothing but the usual notion of a controlled rough path
introduced by Gubinelli in \cite{controlled} (see also \cite{Danyu} for a different perspective).
In this case, we will omit the subscript $S$ in the notations introduced above.
\end{rem}

Note that we have the bound
 \begin{equ}
\| Y\| _{\gamma,\alpha} \leq |R^Y|_{\gamma,\alpha} + \|Y'\|_{\infty,\alpha}|X|_\gamma \leq C (1 + |X|_\gamma)(\| Y'_0\| _{\HH_\alpha^{m\times d}}+	\| Y,Y'\| _{X,2\gamma,\alpha})\;,
\end{equ}
where the constant $C$ depends only on $\gamma$ and $T$ and can be chosen uniform over $T \in (0,1]$.
Given a controlled rough path according to $S$, we can define a corresponding `stochastic convolution'. 
\begin{thm} \label{integral}
	Let $T>0$ and $\XX = (X,\X) \in \cC^{\gamma}([0,T],\R^d)$ for some $\gamma\in (1/3,1/2]$. Let $(Y,Y') \in \cD^{2\gamma}_{S,X}([0,T],\HH_\alpha^d)$ Then 
	the integral defined by 
	\begin{equ}[e:integral1]
	\int_s^tS_{t-u}Y_udX_u  := \lim_{|\CP|\to 0} \sum_{[u,v]\in \CP}S_{t-u}(Y_u X_{v,u} + Y'_u\X_{v,u}), 
	\end{equ} 
	exists as an element of $\hat{\CC}^\gamma \HH_{\alpha}$ and satisfies for every $0 \leq \beta < 3\gamma$:
	\begin{equs}
	\Big\| \int_s^tS_{t-u}Y_udX_u - S_{t-s}Y_s X_{t,s}-& S_{t-s}Y'_s\X_{t,s}\Big\| _{\HH_{\alpha+\beta}} \lesssim\\ \lesssim \big(|R^Y|_{2\gamma,\alpha}|&X|_{\gamma}+\| Y'\| _{\gamma,\alpha}|\X|_{2\gamma}\big) |t-s|^{3\gamma-\beta}.\label{e:integral2}
	\end{equs}
	Moreover the map 
	\begin{equ}
	(Y,Y') \to (Z,Z') := \Big(	\int_0^\cdot S_{\cdot-u}Y_udX_u, Y\Big),
	\end{equ}
	is continuous from $\cD^{2\gamma}_{S,X}([0,T],\HH_\alpha^d)$ to $\cD^{2\gamma}_{S,X}([0,T],\HH_\alpha)$ and one has the bound: 
	\begin{equ}
	\| Z,Z'\| _{X,2\gamma\alpha} \lesssim 	\| Y\| _{\gamma, \alpha} + (\|Y'_0\|_{\HH_{\alpha}}+\| (Y,Y') \| _{X,2\gamma,\alpha})(|X|_\gamma+|\X|_{2\gamma}). \label{e:integral3}
	\end{equ}
	Here the underlying constant depends on $\gamma,\,d$ and $T$ and is uniform over $T \in (0,1]$.
\end{thm}
\begin{proof}
	For $\Xi_{v,u} = S_{v-u}Y_u X_{v,u}+ S_{v-u}Y'_u \X_{v,u}$ we have 
	$$\ddh  \Xi _{v,m,u} = - S_{v-m}R_{m,u}^Y X_{v,m}- S_{v-m}\ddh Y'_{m,u}\X_{v,m}.$$ 
	This follows from the definition of controlled rough path~(\ref{controlled}) and Chen's relation~(\ref{chen}). 
Since $\ddh \Xi_{v,m,u} = S_{v-m}\tilde{\Xi}_{v,m,u}$ for some $\tilde{\Xi}$ satisfying~(\ref{e:sewing2}) with $\mu = 3\gamma > 1$ and $M = |R^Y|_{2\gamma,\alpha}|X|_{\gamma}+\| Y'\| _{\gamma,\alpha}|\X|_{2\gamma}$, the existence of the limit in~(\ref{e:integral1}) and the bound~(\ref{e:integral2}) follow directly 
	from~(\ref{e:sewing3}).
	If we define $Z_t = \int_0^t S_{t-u}Y_udX_u$ then it is not hard to see that $\ddh  Z_{t,s} = \int_s^tS_{t-u}Y_udX_u$, 
so that~(\ref{e:integral3}) follows immediately from~(\ref{e:integral2}). We will address the continuity of integration map in Section~\ref{stability}
below. \end{proof}

\subsection{Composition with regular functions}

Now, we need to restrict our study to a suitable class of non-linearities.
\begin{defn}
	For some fixed $\alpha,\beta \in \R$ and $k \in \N_0$ we define the space
	$C^k_{\alpha,\beta}(\HH^m,\HH^n)$ as the space of $k$-differentiable functions $G : \HH^m_\theta \to \HH^n_{\theta+\beta}$ for every $\theta\geq\alpha$, and $n,m \in \mathbf{N}_0$ and such that $D^iG$ sends bounded subsets of $\HH^m_\theta$ to bounded sets of $\HH^n_{\theta+\beta}$, for all $i = 0, \dots k$. For such functions $\|G\|_{C^k}$ will represent some norm which depends on the first $k$ derivatives and its exact form will be clear from the context. When $m = n$ we will simply write $C^k_{\alpha,\beta}(\HH^m)$.
\end{defn}

\noindent With this at hand:

\begin{lem} \label{compos1}
Let $F \in C^2_{\alpha,0}(\HH^m,\HH^n)$ with all derivatives up to order $2$ bounded, let $T>0$ and $(Y,Y') \in \cD^{2\gamma}_{S,X}([0,T],\HH_\alpha^m)$ for some $(X,\X) \in \cC^{\gamma}([0,T],\R^d)$, $\gamma\in (1/3,1/2]$. Moreover assume that in addition $Y \in L^\infty([0,T],\HH_{\alpha+2\gamma}^m)$ and $Y' \in L^\infty([0,T],\HH_{\alpha+2\gamma}^{m\times d})$. Define $(Z_t,Z'_t) = (F(Y_t),DF(Y_t)\circ Y'_t)$ then $(Z,Z') \in \cD^{2\gamma}_{S,X}([0,T],\HH_\alpha^n)$ and satisfies the bound: 
\begin{equ} 
\| (Z,Z') \| _{X,2\gamma,\alpha} \leq C_{F} (1+|X|_\gamma)^2(1+\|Y\|_{\infty,{\alpha+2\gamma}}+\|Y'\|_{\infty,{\alpha+2\gamma}}+\| (Y,Y')\| _{X,2\gamma})^2. \label{e:compos1}
\end{equ}
The constant $C_F$ depends on the bounds on $F$ and its derivatives. It also depends on time $T$, but is uniform over $T\in (0,1]$.
\end{lem}
\begin{proof}
We only consider the case $d=m=n=1$, the generalisation to higher dimensions being purely a matter of notations.
From~(\ref{semigroup}) since $0< \gamma< 2\gamma \leq 1$ we have for any $V\in \CC_1([0,T])$ and for any $u \in [0,T]$ and any $\alpha \in \R$ the following:
\begin{equ}
\| S_{t-s}V_u-V_u\| _{\HH_\alpha} \lesssim |t-s|^{2\gamma} \|V_u\|_{\HH_{\alpha+2\gamma}}, \quad
|V|_{\gamma,\alpha} \lesssim \|V\|_{\gamma,\alpha}+\|V\|_{\infty,\alpha+2\gamma}. \label{e:compos2}
\end{equ}
\noindent\textbf{Bound of $Z'$:} First, we write $\ddh Z'$ as  
\begin{equs}
(\ddh Z')_{t,s} &= (DF(Y_t)Y'_t - DF(Y_t)S_{t-s}Y'_s)+(DF(Y_t)S_{t-s}Y'_s-DF(Y_s)S_{t-s}Y'_s) \\
&\quad +(DF(Y_s)S_{t-s}Y'_s-DF(Y_s)Y'_s)+(DF(Y_s)Y'_s-S_{t-s}DF(Y_s)Y'_s) \\
&= \rI + \rII + \rIII +\rIV\;.
\end{equs}
Using~(\ref{e:compos2}) we can bound these terms as follows:
\begin{equs}
	\| \rI\| _{\HH_\alpha}  &\leq \| DF(Y_t)\| _{\LL^{\alpha,\alpha}}\| Y'\| _{\gamma,\alpha}|t-s|^\gamma \lesssim C_F \| Y'\| _{\gamma,\alpha}|t-s|^\gamma\;,\\
	\| \rII\| _{\HH_\alpha}  &\leq \| DF(Y_t)-DF(Y_s)\| _{\LL^{\alpha,\alpha}}\| Y'_s\| _{\HH_\alpha} \lesssim C_F |Y|_{\gamma,\HH_{\alpha}}|t-s|^\gamma\|Y'\|_{\infty,{\alpha+2\gamma}} \\
	&\lesssim C_F (\| Y\| _{\gamma,\alpha}+\|Y\|_{\infty,{\alpha+2\gamma}})\|Y'\|_{\infty,{\alpha+2\gamma}}|t-s|^\gamma\;,\\
	\| \rIII\| _{\HH_\alpha}  &\lesssim C_F \|Y'\|_{\infty,{\alpha+2\gamma}}|t-s|^\gamma\;,\\
	\| \rIV\| _{\HH_\alpha} &\lesssim \| DF(Y_s)Y'_s\| _{\HH_{\alpha+2\gamma}}|t-s|^\gamma \lesssim C_F \|Y'\|_{\infty,\alpha+2\gamma}|t-s|^\gamma\;.
\end{equs}
Combining these bounds all together we obtain:
$$\| Z'\| _{\gamma,\alpha} \lesssim C_F(\| Y'\| _{\gamma,\alpha}+\|Y'\|_{\infty,\HH_{\alpha+2\gamma}})(1+\| Y\| _{\gamma,\alpha}+\|Y\|_{\infty,{\alpha+2\gamma}}).$$
\textbf{Bound of $R^Z$: }
\begin{equs}
	R^Z_{t,s}  &= \ddh Z_{t,s} -  S_{t-s}Z'_s X_{t,s} \\
	&= \ddh Z_{t,s} - DF(Y_t) S_{t-s}Y'_s X_{t,s}+ DF(Y_t) S_{t-s}Y'_s X_{t,s}-  S_{t-s}DF(Y_s)Y'_sX_{t,s}\\
	&= \Big(F(Y_t)-S_{t-s}F(Y_s) - DF(Y_t)(Y_t - S_{t-s}Y_s)\Big) \\
	&\qquad- \Big(DF(Y_t)R^Y_{t,s} + (\rII+\rIII+\rIV) X_{t,s}\Big) = \rV - \rVI.
\end{equs}
The term $\rVI$ can easily be bounded using bounds for $\rII$, $\rIII$ and $\rIV$: 
\begin{equ}
	\| \rVI\| _{\HH_\alpha} \lesssim C_F(|R^Y|_{2\gamma,\alpha}+|X|_\gamma\|Y'\|_{\infty,{\alpha+2\gamma}}(1+\| Y\| _{\gamma,\alpha}+\|Y\|_{\infty,{\alpha+2\gamma}}))|t-s|^{2\gamma}.
\end{equ}
For $\rV$ we have 
\begin{equs}
	\rV &= \Big(F(Y_t) - F(S_{t-s}Y_s) - DF(Y_t)(Y_t - S_{t-s}Y_s)\Big) + \Big(F(S_{t-s}Y_s)-S_{t-s}F(Y_s)\Big)\\
	&= \rVII + \rVIII.
\end{equs}
By Taylor's theorem we obtain 
\begin{equ}
	\| \rVII\| _{\HH_\alpha} \lesssim C_F \| Y\| ^2_{\gamma,\alpha}|t-s|^{2\gamma}\;,
\end{equ}
while the `commutator' $\rVIII$ is bounded by
\begin{equs}
	\| \rVIII\| _{\HH_\alpha} &\lesssim \| F(S_{t-s}Y_s)-F(Y_s)\|_{\HH_\alpha}+\| F(Y_s)-S_{t-s}F(Y_s)\| _{\HH_\alpha}\\ 
	&\lesssim (C_F\|Y\|_{\infty,{\alpha+2\gamma}} + \| F(Y_s)\| _{\HH_{\alpha+2\gamma}})|t-s|^{2\gamma}\\
	& \leq C_F(1+\|Y\|_{\infty,{\alpha+2\gamma}})|t-s|^{2\gamma}\;,
\end{equs}
where we used~(\ref{e:compos2}) to go from the first to the second line. Combining both bounds on $Z'$ and $R^Z$ we obtain the desired result. 
\end{proof}

\subsection{$\cD^{2\gamma,\beta,\eta}_{S,X}$ spaces}

\begin{defn}
	Let $\XX \in \cC^{\gamma}([0,T],\R^d)$ for some $\gamma\in (1/3,1/2]$. Then for and $\beta \in \R$ and $\eta \in [0,1]$ define a space 
	$$\cD^{2\gamma,\beta,\eta}_{S,X}([0,T],\HH_\alpha) = \cD^{2\gamma}_{S,X}([0,T],\HH_\alpha)\cap  \big(\hat{\CC}^\eta([0,T],\HH_{\alpha+\beta})\times L^\infty([0,T],\HH^d_{\alpha+\beta})\big)\;.$$
\end{defn}
We also wrote $\hat{\CC}^0 = L^\infty$ for $\eta = 0$ and, as usual, we will drop the subscript $S$ when $S = \id$.
Note that by Lemma~\ref{compos1}, composition with regular functions maps $\cD^{2\gamma,2\gamma,\eta}_{S,X}([0,T],\HH_\alpha)$ to $\cD^{2\gamma,2\gamma,0}_{S,X}([0,T],\HH_\alpha)$ for every $\eta \in[0,1]$. For simplicity, we also introduce the 
useful notation 
\begin{equ}
\DD^{2\gamma}_{X}([0,T],\HH_{\alpha}):= \cD^{2\gamma,2\gamma,\gamma}_{S,X}([0,T],\HH_{\alpha-2\gamma})\;.
\end{equ}
\textbf{Warning:} we have shifted the space regularity in the definition of $\DD^{2\gamma}_X$ by $2\gamma$ in the right hand side. 
We will later solve our equations in the space $\DD^{2\gamma}([0,T],\HH)$ ($\alpha = 0$). 

\medskip
With this at hand we now show:
\begin{prop} \label{space_equiv1}
For $1/3<\varepsilon\leq \gamma \leq1/2$, the spaces $\cD^{2\varepsilon,2\gamma,0}_{X}([0,T],\HH_\alpha)$ and $\cD^{2\varepsilon,2\gamma,0}_{S,X}([0,T],\HH_\alpha)$ are the same.
\end{prop}
\begin{proof}
First consider $(Y,Y') \in \cD^{2\gamma,2\gamma,0}_{S,X}([0,T],\HH_\alpha)$, we rewrite \eqref{controlled} as
	$$Y_t-Y_s = Y'_s X_{t,s} + R^Y_{t,s}+S_{t-s}Y_s-Y_s+(S_{t-s}Y'_s-Y'_s) X_{t,s}.$$
Combining this with
\begin{equs}
\|S_{t-s}Y_s-Y_s\|_{\HH_\alpha} &\lesssim |t-s|^{2\gamma}\|Y\|_{\infty,\alpha+2\gamma}\;,\\ 
\|(S_{t-s}Y'_s-Y'_s) X_{t,s}\|_{\HH_\alpha} &\lesssim |t-s|^{3\gamma}\|Y'\|_{\infty,\alpha+2\gamma}|X|_\gamma\;,
\end{equs}
we conclude that the remainder $R^Y_{t,s}+S_{t-s}Y_s-Y_s+(S_{t-s}Y'_s-Y'_s) X_{t,s}$ is of regularity $|t-s|^{2\varepsilon}$. We can similarly show that $Y' \in \CC^\varepsilon$ and therefore $\cD^{2\varepsilon,2\gamma,0}_{X}([0,T],\HH_\alpha)\subseteq \cD^{2\varepsilon,2\gamma,0}_{S,X}([0,T],\HH_\alpha)$. The proof of the converse implication is analogous. 
\end{proof}

For the next proposition we use that the inner product on $\HH$ extends uniquely 
to a continuous bilinear map $\langle \cdot,\cdot \rangle : \HH_{-\alpha}\times \HH_{\alpha} \to \R$ for every $\alpha \geq 0$. 

\begin{prop} \label{space_equiv2}
For any $(Y,Y') \in \cD^{2\gamma,2\gamma,0}_{S,X}([0,T],\HH_{-2\gamma})$ and $\psi \in \HH_{2\gamma}$, one has $(\langle Y,\psi \rangle, \langle Y',\psi \rangle)\in \cD^{2\gamma}_{X}([0,T],\R)$. Also for any fixed $t \leq T$ we get a controlled rough path $(\langle S_{t-\cdot}Y,\psi \rangle, \langle S_{t-\cdot}Y',\psi \rangle ) \in \cD^{2\gamma}_{X}([0,t],\R)$. Moreover for fixed $t>0$ and $h \in \HH$, setting $Z_v = \int_v^t\langle S_{s-v}Y_v,h \rangle ds$ and $Z'_v = \int_v^t\langle S_{s-v}Y'_v,h \rangle ds$, we have $(Z,Z') \in \cD^{2\gamma}_{X}([0,t],\R)$ satisfying
	$$\|(Z,Z')\|_{X,2\gamma} \lesssim_T (1+|X|_\gamma)\|(Y,Y')\|_{\cD^{2\gamma,2\gamma,0}_{S,X}}\|h\|.$$
	Similar bound holds for the other two controlled rough paths stated in the proposition, but with $\|h\|$ replaced by $\|\psi\|_{\HH_{2\gamma}}$. 
\end{prop}

\begin{proof}
The fact that the first two functions are controlled rough paths follows easily from Proposition~\ref{space_equiv1}. 
For the third one we cannot use Cauchy-Schwarz straight away because $h$ is not regular enough. Instead, we use~\eqref{controlled} and write 
\begin{equs}
Z_v - Z_u &= \int_v^t(\langle S_{s-v}Y_v,h \rangle - \langle S_{s-u}Y_u,h \rangle)ds - \int_u^v\langle S_{s-u}Y_u,h \rangle ds\\ 
&=\int_v^t\langle S_{s-v}\ddh Y_{v,u},h \rangle ds - \int_u^v\langle S_{s-u}Y_u,h \rangle ds \\
&=\int_v^t\langle S_{s-v}S_{v-u}Y'_u X_{v,u} + S_{s-v}R^Y_{v,u},h \rangle ds  - \int_u^v\langle S_{s-u}Y_u,h \rangle ds \\
&=\int_v^t\langle S_{s-u}Y'_u,h \rangle ds \, X_{v,u} + \int_v^t\langle S_{s-v}R^Y_{v,u},h \rangle ds - \int_u^v\langle S_{s-u}Y_u,h \rangle ds \\
&= Z'_u X_{v,u} + R^Z_{v,u}\;,
\end{equs}
where $$R^Z_{v,u} = \int_v^t\langle S_{s-v}R^Y_{v,u},h \rangle ds - \int_u^v\langle S_{s-u}Y_u,h \rangle ds - \int_u^v\langle S_{s-u}Y'_u,h \rangle ds\;.$$
	Since $\|Y\|_{\infty,\HH}$ and $\|Y'\|_{\infty,\HH^d_0}$ are finite we see that the last two terms of $R^Z$ are bounded by
	$|v-u| (\|Y\|_{\infty,\HH}+\|Y'\|_{\infty,\HH^d_0})\|h\|$. For the first term we have: 
\begin{equs}
\Big|\int_v^t\langle S_{s-v}R^Y_{v,u},h \rangle ds\Big| &\leq \int_v^t\|S_{s-v}R^Y_{v,u}\|\,\|h\| ds \lesssim \int_v^t |s-v|^{-2\gamma}\|R^Y_{v,u}\|_{\HH_{-2\gamma}}\,\|h\| ds \\
= |t-v|^{1-2\gamma}|R^Y|&_{2\gamma,-2\gamma}|v-u|^{2\gamma}\|h\| \lesssim T^{1-2\gamma}|R^Y|_{2\gamma,-2\gamma}|v-u|^{2\gamma}\|h\|\;,
\end{equs}
where we have used  that $2\gamma < 1 $. One similarly shows that $|Z'|_\gamma < \infty$. 
\end{proof}

We finish this subsection by extending Lemma~\ref{compos1} to functions that lose some space regularity. 
Since the proof is identical to that of Lemma~\ref{compos1}, we omit it.
\begin{lem} \label{compos4}
	Let $\sigma \geq 0$ and $F \in C^2_{\alpha,-\sigma}(\HH,\HH^d)$ with all respective derivatives bounded. Let $T>0$ and $(Y,Y') \in \cD^{2\gamma,2\gamma,0}_{S,X}([0,T],\HH_{\alpha})$ for some $(X,\X) \in \cC^{\gamma}([0,T],\R^d)$, $\gamma\in (1/3,1/2]$. Then $(Z_t,Z'_t) = (F(Y_t),DF(Y_t)\circ Y'_t) \in \cD^{2\gamma,2\gamma,0}_{S,X}([0,T],\HH_{\alpha-\sigma}^d)$ and one has
	\begin{equ}
	\| (Z,Z') \| _{X,2\gamma,\alpha-\sigma} \lesssim_F (1+|X|_\gamma)^2(1+\|Y,Y\|_{\cD^{2\gamma,2\gamma,0}_{S,X}})^2.
	\end{equ}
\end{lem}

\subsection{Stability of integration and composition} \label{stability}

First we will give a meaning to the ``distance'' between two controlled paths that are controlled by two different rough paths. Then with the notion of these two distances we will state the continuity of two maps: integration and composition.
\begin{defn}
	For $(Y,Y') \in \cD^{2\gamma}_{S,X}([0,T],\HH^m_\alpha)$ and $(V,V') \in \cD^{2\gamma}_{S,\tilde{X}}([0,T],\HH^m_\alpha)$ define a distance:
	\begin{equ}
		d_{X,\tilde{X},2\gamma,\alpha}(Y,V) = \| Y'-V'\| _{\gamma,\alpha} + |R^Y-R^V|_{2\gamma,\alpha}.
	\end{equ}
	We also measure the distance between two functions $(Y,Y') \in \cD^{2\gamma,\beta,\eta}_{S,X}([0,T],\HH^m_\alpha)$ and $(V,V') \in \cD^{2\gamma,\beta,\eta}_{S,\tilde{X}}([0,T],\HH^m_\alpha)$ with:
	\begin{equ} 	
	d_{2\gamma,\beta,\eta}(Y,V) = \|Y'-V'\|_{\infty,\alpha+\beta} +\|Y-V\|_{\eta,\alpha+\beta}+ d_{X,\tilde{X},2\gamma,\alpha}(Y,V).
	\end{equ}
	We make an abuse of notation by not writing dependence of $d_{X,\tilde{X},2\gamma}$ and $d_{2\gamma,\beta,\eta}$ on $Y'$ and $V'$.
\end{defn}

For the next two lemmas we are going to assume with $X,\tilde{X},Y,V$ as above that there exists $M>0$ such that 
$|X|_\gamma,|\X|_{2\gamma},\| (Y,Y')\| _{\cD^{2\varepsilon,2\gamma,\eta}_{S,X}} < M$ and the same is true for $\tilde{X}$ and $V$. 
We are not presenting the proofs of the following stability results since the ideas are exactly the same as in the proofs 
of their analogues Theorems 4.16 and 7.5 from \cite{friz}. The modifications needed for our case only involve replacing
the sewing lemma by Lemma~\ref{sewing} and exploiting the fact that the regularity assumptions on $F$ yield control
on commutators of the type $F(S_{t-s}Y_s)-S_{t-s}F(Y_s)$. These modifications were already used to similar effect 
in the proofs of Theorem~\ref{integral} and Lemma~\ref{compos1}.
\begin{lem} \label{stab_int}
	 Let $1/3 < \varepsilon \leq \gamma \leq 1/2$, $0\leq \eta < 3\varepsilon - 2\gamma$ and  $\XX,\tilde{\XX} \in \cC^\gamma([0,T],\R^d)$. Consider $(Y,Y')\in \cD^{2\varepsilon,2\gamma,0}_{S,X}([0,T],\HH_\alpha^d)$ and
	$(V,V')\in \cD^{2\varepsilon,2\gamma,0}_{S,\tilde{X}}([0,T],\HH_\alpha^d)$ that both satisfy the bounds with respect to $M$ as above. Define 
	$$(Z,Z') := \Big(	\int_0^\cdot S_{\cdot-u}Y_udX_u, Y\Big), $$
	and similarly $(W,W')$ as a rough integral of $(V,V')$. Then the following local Lipschitz estimates are true: 
	\begin{equs} 
	d_{X,\tilde{X},2\varepsilon,\alpha}(Z,W) &\lesssim_M \varrho_\gamma(\XX,\tilde{\XX}) + \| Y'_0-V'_0\|_{\HH_\alpha} +d_{X,\tilde{X},2\varepsilon,\alpha}(Y,V)T^{\gamma-\varepsilon}.\qquad  \label{e:estabInt1}\\
	\| Z-W\| _{\eta,\alpha+2\gamma} &\lesssim_M \varrho_\gamma(\XX,\tilde{\XX})  + \| Y_0-V_0\|_{\HH_{\alpha+2\gamma}} \label{e:estabInt2} \\
	&\quad + \| Y'_0-V'_0\|_{\HH_{\alpha+2\gamma}} +d_{2\varepsilon,2\gamma,0}(Y,V)T^{\gamma-\varepsilon}\;,
	\end{equs}
	with the underlying $T$-dependent constants uniform for $T\leq 1$.
\end{lem}

It may look like we are far from obtaining the stability result in the same H\"{o}lder regularity as our rough path $X$, but here $\varepsilon$ can be taken arbitrarily close to $\gamma$ which itself allows to take $\eta$ arbitrarily close to $\varepsilon$. 
Note also that inequality~(\ref{e:estabInt1}) is true in spaces $\cD^{2\varepsilon}_{S,X}$ and not just in $ \cD^{2\varepsilon,2\gamma,\varepsilon}_{S,X}$.

\begin{lem} \label{stab_comp}
	Let $\XX,\tilde{\XX}$, $1/3 < \varepsilon \leq \gamma \leq 1/2$ and $\eta \in [0,1]$. Let $(Y,Y')\in \cD^{2\varepsilon,2\gamma,\eta}_{S,{X}}([0,T],\HH_\alpha)$ and
	$(V,V')\in \cD^{2\varepsilon,2\gamma,\eta}_{S,\tilde{X}}([0,T],\HH_\alpha)$ satisfy the bounds with respect to $M$ as above. Let $\sigma \geq 0$ and $F \in C^3_{\alpha,-\sigma}(\HH,\HH^d)$. Define 
	$$(Z,Z') := (	F(Y), DF(Y)\circ Y')\quad  \text{and} \quad (W,W') :=(	F(V), DF(V)\circ V').$$
	Then the following local Lipschitz estimates are true: 
	\begin{equ} 
	d_{2\varepsilon,2\gamma,0}(Z,W) \lesssim_M \varrho_\gamma(\XX,\tilde{\XX}) + \| Y_0-V_0\|_{\HH_{\alpha+2\gamma}} + d_{2\varepsilon,2\gamma,\eta}(Y,V)\;. \label{e:stab_comp1}
	\end{equ}
	Here $d_{2\varepsilon,2\gamma,0}(Z,W)$ contains $\HH_{\alpha-\sigma}$ and $\HH_{\alpha+2\gamma-\sigma}$ spatial norms and $d_{2\varepsilon,2\gamma,\eta}(Y,V)$ contains $\HH_{\alpha}$ and $\HH_{\alpha+2\gamma}$ spatial norms.
\end{lem}

\section{Rough PDEs} \label{RPDEs}

We now use the results obtained in the previous section to solve RPDEs in the Hilbert space $\HH$.
First we consider equations without non-linear drift of the type
\begin{equ} 
	dY_t = LY_tdt + F(Y_t)dX_t\quad  \text{and}\quad  Y_0 = \xi \in \HH.
\end{equ}
Here $L$ is as above, $F$ is a $C^3$ function on $\HH$ and $\XX = (X, \X) \in \cC^\gamma(\R_+,\R^d)$ (meaning $|X|_{\gamma, [S,T]}$ and $|\X|_{2\gamma, [S,T]}$ are finite for all intervals $[S,T]$).

We will show that Lemma~\ref{compos1} and Theorem~\ref{integral} guarantee that if $(Y,Y') \in \DD^{2\gamma}_{X}([0,T],\HH)$, then 
\begin{equ} 
\MM_T(Y,Y')_t  :=  \Big(S_t\xi + \int_0^t S_{t-u}F(Y_u)d\XX_u,\, F(Y_t)\Big) \label{solnMap1}
\end{equ}
yields again an element of $\DD^{2\gamma}_{X}([0,T],\HH)$.
We now show that for $T$ small enough this map has a unique fixed point:
\begin{thm}[Rough Evolution Equation] \label{RPDE1}
	Given $\xi \in \HH$, $F \in C^3_{-2\gamma,0}(\HH,\HH^d)$ and $\XX = (X, \X) \in \cC^\gamma(\R_+,\R^d)$, there exists $\tau >0$ and a unique element $(Y,Y') \in \DD^{2\gamma}_{X}([0,\tau),\HH)$ such that $Y' = F(Y)$ and
\begin{equ}[e:MildSoln]
	Y_t = S_t\xi + \int_0^t S_{t-u}F(Y_u)d\XX_u\, , \quad \, t<\tau.
\end{equ}
\end{thm}
\begin{proof}
	First note $\XX = (X, \X) \in \cC^\gamma \subset \cC^\varepsilon$ for $1/3 < \varepsilon < \gamma \leq 1/2$. Let $T<1$ we will find a solution $(Y,Y') \in \DD^{2\varepsilon}_{X}([0,T],\HH_{2\varepsilon-2\gamma})$ as a fixed point of the map $\MM_T$ given by~(\ref{solnMap1}). Then in the end we will briefly describe that one can actually make an improvement and show that $(Y,Y') \in \DD^{2\gamma}_{X}([0,T],\HH)$. 
	The proof is analogous to \cite[Thm~8.4]{friz}, the  only difference being that we have two different scales of space regularity for which we need to be able to obtain the bound~(\ref{e:compos1}). We will therefore show only invariance of the solution map~(\ref{solnMap1}), because proving it already contains all the techniques that are not present in the \cite[Thm~8.4]{friz}.
	
	Any semi-norm $\|\cdot\|_{X,2\varepsilon}$ will be taken in the $\HH_{-2\gamma}$ space so sometimes we won't indicate this.
	Note that if $(Y,Y')$ is such that $(Y_0,Y'_0) = (\xi,F(\xi))$ then the same is true for $\MM_T(Y,Y')$. We can therefore view $\MM_T$ as a map on the complete metric space: $$\{(Y,Y') \in \DD^{2\varepsilon}_{X}: \, Y_0 = \xi,\; Y'_0 = F(\xi)\}.$$
	This is also true for the closed unit ball $B_T$ centred at $t \to (S_t\xi + S_tF(\xi) X_{t0},S_tF(\xi) )\in \DD^{2\varepsilon}_{X}([0,T],\HH_{2\varepsilon-2\gamma})$. One can show using $\| (S_\cdot \xi+S_\cdot F(\xi) X_{\cdot0},S_\cdot F(\xi) )\| _{X,2\varepsilon,-2\gamma}$ $= 0$ (since that $\ddh (S_\cdot\xi)_{t,s} = 0$) that in fact:
\begin{equs}
B_T &= \{(Y,Y') \in \DD^{2\varepsilon}_{X}([0,T],\HH_{2\varepsilon-2\gamma}): \, Y_0 = \xi,\; Y'_0 = F(\xi), \\
&\qquad \|Y -S_\cdot F(\xi) X_{\cdot0}\|_{\varepsilon,2\varepsilon-2\gamma}  +\|Y' - S_\cdot F(\xi)\|_{\infty,2\varepsilon-2\gamma}+ \| (Y, Y' )\| _{X,2\varepsilon} \leq 1\}\;.
\end{equs}
	Note that by the triangle inequality for $(Y,Y') \in B_T$ we have $$\|(Y,Y')\|_{\DD^{2\varepsilon}_{X}}  \lesssim (1 +\|\xi\| + \|F(\xi)\|)(1+|X|_\gamma).$$
	It remains to show that for $T$ small enough $\MM_T$ leaves $B_T$ invariant and is contracting there, so that the claim follows
	from the Banach fixed point theorem. Constants below denoted by $C$ may change from line to line and depend on $\gamma,\varepsilon,X, \X$ and $\xi$ without mentioning. Nevertheless they are uniform in $T \in (0,1]$. Without loss of generality we assume that $F$ is $C^3_b$, since by definition of $C^3_{-2\gamma,0}(\HH,\HH^d)$ function $F$ sends bounded sets to bounded sets, which is the case for us since for $(Y,Y') \in B_T$, both $|Y|_{\infty,2\varepsilon-2\gamma}$ and $|Y'|_{\infty,2\varepsilon-2\gamma}$ are uniformly bounded by a constant depending on $\xi$.
	For $(Z_t,Z'_t) = (F(Y_t),DF(Y_t)\circ Y'_t)$ we have by Lemma~\ref{compos1}  
	$$\| (Z,Z') \| _{X,2\varepsilon} \leq C_F (1+\| (Y,Y')\| _{\DD^{2\varepsilon}_{X}})^2\leq C_F(1+\| \xi\| +\| F(\xi)\|)^2 \leq C_{F,\xi}\;,$$ 
	and from~(\ref{e:integral3}): 
	\begin{equs}
		\| \MM_T(Y, Y')\| _{X,2\varepsilon} &= \Big\| \Big(\int_0^\cdot S_{\cdot-u}Z_udX_u, Z\Big)  \Big\| _{X,2\varepsilon} \\
		&\lesssim \| Z\| _{\varepsilon,-2\gamma} +(\|Z'_0\|_{\HH_{-2\gamma}}+\| (Z,Z') \| _{X,2\varepsilon,-2\gamma}) \varrho_\varepsilon(X)\\
		&\lesssim \| Z\| _{\varepsilon,-2\gamma} + (\|Z'_0\|_{\HH_{-2\gamma}}+\| (Z,Z') \| _{X,2\varepsilon})T^{\gamma-\varepsilon}.
	\end{equs}
	Since $(Y,Y') \in B_T$, we obtain from~(\ref{controlled}) that $\| Y\| _{\varepsilon,-2\gamma} \leq (|X|_\gamma+1)T^{\gamma-\varepsilon}$. One can also show along the same lines as in Lemma~\ref{compos1} that 
	\begin{equs}
		\| \ddh Z_{t,s}\| _{\HH^d_{-2\gamma}} &\lesssim C_F \| \ddh Y_{t,s}\| _{\HH_{-2\gamma}}+C_F\| S_{t-s}Y_s - Y_s\| _{\HH_{-2\gamma}} + |t-s|^{2\varepsilon}\| F(Y_s)\| _{\HH_{2\varepsilon-2\gamma}}\\
		&\lesssim C_F\big(T^{\gamma-\varepsilon}|t-s|^\varepsilon + |t-s|^{2\varepsilon}\| Y_s\| _{\HH_{2\varepsilon-2\gamma}}+T^\varepsilon|t-s|^{\varepsilon}\big)\\ &\lesssim C_{F,\xi}\big(T^{\gamma-\varepsilon}+T^{\gamma+\varepsilon}+T^\varepsilon\big)|t-s|^\varepsilon.
	\end{equs}
	Therefore since $T < 1$ we conclude that $\| Z\| _{\varepsilon,-2\gamma} \lesssim C_{F,\xi}T^{\gamma-\varepsilon}$, where 
	$C_{F,\xi}$ is a constant that also depends on initial condition.
	
 To estimate $\|\MM_T(Y) -S_\cdot F(\xi) X_{\cdot,0}\|_{\varepsilon,2\varepsilon-2\gamma}$ we use $\ddh (S_\cdot F(\xi) X_{\cdot,0})_{t,s} = S_tF(\xi) X_{t,s}$ and since $2\varepsilon < 1$ we can use a better bound from~(\ref{e:integral2}) to deduce:
	\begin{equs}
		\|\ddh (\MM_T(&Y) - S_\cdot F(\xi) X_{\cdot,0})_{t,s}\|_{\HH_{2\varepsilon-2\gamma}} = \Big\|\int_s^t S_{t-u}F(Y_u)dX_u - S_tF(\xi) X_{t,s}\Big\|_{\HH_{2\varepsilon-2\gamma}}\\
		&\leq (\|F(\xi)\| + \|Z\|_{\infty,-2\gamma})|X|_\varepsilon|t-s|^\varepsilon + \|Z'\|_{\infty,-2\gamma} |\X|_{2\varepsilon}|t-s|^{2\varepsilon} \\ &\qquad+C(|X|_{\varepsilon}|R^Z|_{2\varepsilon}+|\X|_{2\varepsilon}\| Z'\| _{\varepsilon})|t-s|^{3\varepsilon-2\varepsilon} \\
		&\lesssim (\|F(\xi)\|  + |Z'_0|_{\HH_{-2\gamma}} + \| (Z,Z') \| _{X,2\varepsilon}) \varrho_\varepsilon(X)|t-s|^{\varepsilon} \leq C_{F,\xi}T^{\gamma-\varepsilon}|t-s|^\varepsilon.
	\end{equs}
	Finally we estimate the term $\|\MM_T(Y)'_t - S_tF(\xi)\|_{\HH_{2\varepsilon-2\gamma}}$:
	\begin{equs}
		\|\MM_T(Y)_t' &- S_t F(\xi)\|_{\HH_{2\varepsilon-2\gamma}} = \\
		&= \|F(Y_t) -F(S_t\xi) + F(S_t\xi) - F(\xi) + F(\xi)-S_tF(\xi)\|_{\HH_{2\varepsilon-2\gamma}}\\
		&\lesssim_F\|Y_t-S_t\xi\|_{\HH_{2\varepsilon-2\gamma}} + \|S_t\xi-\xi\|_{\HH_{2\varepsilon-2\gamma}} + \|F(\xi)-S_tF(\xi)\|_{\HH_{2\varepsilon-2\gamma}}\\
		&\lesssim_F \|Y_t-S_t\xi-S_tF(\xi) X_{t,0}\|_{\HH_{2\varepsilon-2\gamma}}
		+\|F(\xi)\| |X|_\gamma T^\gamma \\
		&\qquad +t^{2\gamma-2\varepsilon}\|\xi\|+t^{2\gamma-2\varepsilon}\|F(\xi)\|\\
		&\lesssim_{F,\xi}(\|Y -S_\cdot F(\xi) X_{\cdot,0}\|_{\varepsilon,2\varepsilon-2\gamma}T^\varepsilon +T^\gamma+ T^{2\gamma-2\varepsilon})\leq C_{F,\xi}T^{\gamma-\varepsilon}.
	\end{equs}
	Putting it all together we can get that
	\begin{equs}
		\|\MM_T(Y) -S_\cdot F(\xi) X_{\cdot,0}\|_{\varepsilon,2\varepsilon-2\gamma}  +&\|\MM_T(Y)' - S_\cdot F(\xi)\|_{\infty,2\varepsilon-2\gamma}+\| \MM_T(Y,Y')\| _{X,2\varepsilon}\\
		&\lesssim C_{F,\xi}T^{\gamma-\varepsilon}.
	\end{equs}
	If $T$ is small enough we guarantee that the left hand side of the above expression is smaller than $1$, thus proving that $B_T$ is invariant under $\MM_T$.
	In order to show contractivity of $\MM_T$, one can use analogous steps to first show
	\begin{equ}
		\| \MM_T(Y, Y')-\MM_T(V, V')\| _{\DD^{2\varepsilon}_{X}} \leq C_{F,\xi} \| (Y-V,Y'-V')\| _{\DD^{2\varepsilon}_{X}}T^{\gamma-\varepsilon}.
	\end{equ} 
This guarantees contractivity for small enough $T$, completing the fixed point argument and thus showing the existence of 
the unique maximal solution to~\eqref{e:MildSoln}.

	Let now $(Y,Y') \in {\DD^{2\varepsilon}_{X}}([0,T],\HH_{2\varepsilon-2\gamma})$ be the solution constructed above, we sketch an argument 
	showing that in fact it belongs to ${\DD^{2\gamma}_{X}}([0,T],\HH)$. We know that 
	\begin{equs} 
		Y_t &= S_t\xi + S_tF(\xi) X_{t,0}+S_tDF(\xi)F(\xi)+R_{t,0},\label{e:more_reg1}\\ 
		Y_t - S_{t-s}Y_s &= S_{t-s}F(Y_s) X_{t,s}+S_{t-s}DF(Y_s)F(Y_s)\X_{t,s}+R_{t,s}. \label{e:more_reg2}
	\end{equs}
	Here $R_{t,s} = \int_s^tS_{t-r}F(Y_r)dX_r - S_{t-s}F(Y_s) X_{t,s}-S_{t-s}DF(Y_s)F(Y_s)\X_{t,s}$. From the  estimate on $R_{t,0}$ using~(\ref{e:integral2}) and since $\xi \in \HH$, we see that~(\ref{e:more_reg1}) implies $Y \in  L^\infty([0,T],\HH)$. Moreover~(\ref{e:more_reg2}) implies  $Y\in \hat{\CC}^\gamma([0,T],\HH_{-2\gamma})$ which, together with $Y\in L^\infty([0,T],\HH)$, implies $F(Y) \in \hat{\CC}^\gamma([0,T],\HH^d_{-2\gamma})\cap L^\infty([0,T],\HH^d_{2\varepsilon-2\gamma})$. This itself implies that $(Y,F(Y)) \in {\cD^{2\gamma}_{S,X}}([0,T],\HH_{-2\gamma})$ (using again~(\ref{e:more_reg2})) and $(F(Y),DF(Y)F(Y)) \in {\cD^{2\gamma}_{S,X}}([0,T],\HH_{-2\gamma})$ which enables us to get an estimate for every  $\beta < 3\gamma$:
	$$\|R_{t,s}\|_{\HH_{\alpha+\beta}} \lesssim_X \|F(Y),DF(Y)F(Y)\|_{X,2\gamma}|t-s|^{3\gamma - \beta}.$$
	Taking $\beta = 2\gamma$ and using~(\ref{e:more_reg2}) again we show that  $Y\in \hat{\CC}^\gamma([0,T],\HH)$, which completes the
proof that $(Y,Y') \in {\DD^{2\gamma}_{X}}([0,T],\HH)$. 
\end{proof}

For $N$ satisfying the same assumptions as the nonlinearities $F_i$ in Theorem~\ref{RPDE1}, we immediately get 
local solutions to equations of the type 
$$du_t = Lu_tdt + N(u_t)dt + \sum_{i=1}^d F_i(u_t)dX^i_t,$$
 for the rough path $X_t = (X^1_t,\ldots,X^d_t) \in \cC^\gamma([0,T],\R^d)$ for  $\gamma \in (1/3,1/2]$. This is because we can simply treat this equation as driven by the rough path $\tilde{X}_t = (X^1_t,\ldots,X^d_t,t)$. However, we can do a bit better than that and obtain weaker assumptions on $N$.
\begin{defn}
Let $k,n \in \N_0$, we call a function $N \in C^k_{\alpha,\beta}(\HH)$ to be of polynomial type $n$ and write $N \in \Poly^{k,n}_{\alpha,\beta}(\HH)$ if for all $\sigma\geq\alpha$ and $0 \leq i \leq k $ there exists $C_{\sigma,i}>0$ such that for all $x,y \in \HH_\sigma$
\begin{equ}
	\| D^iN(x)-D^iN(y)\|_{\LL(\HH^{\otimes i}_\sigma, \HH_{\sigma+\beta})} \leq C_{\sigma,i} \| x-y\|_\sigma(1+\| x\|_\sigma+\| y\|_\sigma)^{n-i-1}\;.
\end{equ}
\end{defn}
\begin{thm}[Rough Nonlinear PDE] \label{RPDE2}
	Let $\gamma \in (1/3,1/2]$ and $\XX = (X, \X) \in \cC^\gamma(\R_+,\R^d)$. Then, given $\xi \in \HH$, $F \in C^3_{-2\gamma,0}(\HH,\HH^d)$, and $N \in \Poly^{0,n}_{-\kappa,-\delta}(\HH)$
	for some $n\geq1$, some $1-\delta > \gamma$ and some small $\kappa > 0$, there exists $\tau >0$ a unique element $(u,u') \in \DD^{2\gamma}_{X}([0,\tau),\HH)$ such that $u' = F(u)$ and
	\begin{equ} 
	u_t = S_t\xi + \int_0^t S_{t-r}N(u_r)dr+ \int_0^t S_{t-r}F(u_r)dX_r\, , \quad t<\tau. \label{e:rpde1}
	\end{equ}
	We call such $u_t$ a mild solution to the Rough PDE: 
	\begin{equ} 
	du_t = Lu_tdt +N(u_t)dt + F(u_t)dX_t\quad \text{and}\quad  u_0 = \xi \in \HH. \label{e:rpde2}
	\end{equ}
\end{thm}
\begin{proof}
The proof is almost identical to that of Theorem~\ref{RPDE1} once we can deal with the non-linearity $N$. 
First we take $\varepsilon  \in (1/3,1/2]$ so that $\varepsilon < \gamma$ hence $1-\delta-\varepsilon > 0$
by our assumption on $\delta$. 
For $(u,u') \in B_T$, we show that $V_t = \int_0^t S_{t-r}N(u_r)dr \in \DD^{2\varepsilon}_{X}$. This is possible if we take $V' = 0$ and thus $R^V = \ddh V$. Since the assumption $1-\delta>\gamma$ implies $\delta < 2/3$ it is possible to find $\beta > 0$ such that $\beta \leq \delta $ and $1-2\varepsilon > \delta - \beta$ whence
\begin{equs}
	\| &V_t - S_{t-s}V_s\| _{\HH_{-2\gamma}} = \Big\| \int_s^t S_{t-r}N(u_r)dr\Big\| _{\HH_{-2\gamma}} \\ 
	&\lesssim \int_s^t |t-r|^{\beta-\delta}\| N(u_r)\| _{\HH_{-2\gamma+\beta-\delta}}dr \lesssim |t-s|^{2\varepsilon} T^{1-\delta+\beta-2\varepsilon} (1+\| u\| _{\infty,-2\gamma+\beta})^n\\ 
	&\lesssim C_{F,N,\xi} |t-s|^{2\varepsilon}T^{1+\beta-\delta-2\varepsilon}.
\end{equs}
We have used above that $\beta \leq \delta < 2/3 \leq 2\varepsilon$ and hence: $$\| u\| _{\infty,-2\gamma+\beta} \lesssim \| u\| _{\infty,2\varepsilon-2\gamma} \leq \| \xi\| + \| u\| _{\varepsilon,2\varepsilon-2\gamma} \lesssim C_{F,\xi}.$$ 
Here we need to further impose $\beta-2\gamma \geq -\kappa$ (which is possible by an appropriate choice of $\beta$) 
so that we can evaluate $N(u)$ for $u \in \HH_{\beta-2\gamma}$.
Similarly to above we get $$\| V_t - S_{t-s}V_s\| \lesssim |t-s|^\gamma T^{1-\delta-\gamma + (2\varepsilon-2\gamma) } (1+\| u\|_{\infty,2\varepsilon-2\gamma})^n.$$
The last inequality serves two roles:

First, since $\varepsilon<\gamma$ can be taken arbitrarily close to $\gamma$, it follows from 
$1-\gamma-\delta > 0$ that for some $\sigma>0$ we have
$$\|V\|_{\varepsilon,2\varepsilon-2\gamma}+\|V'\|_{\infty,2\varepsilon-2\gamma}+\| (V,0)\| _{X,2\varepsilon} \leq C_{F,N,\xi}T^\sigma.$$
Together with the invariance estimates established in the proof of Theorem~\ref{RPDE1}, we 
conclude that the fixed point map 
$\MM_T$ leaves $B_T$ invariant for sufficiently small $T$.
This bound also shows that $\|\int_{0}^{\cdot}S_{\cdot -r}N(u_r)dr\|_{\gamma} < \infty$ which is needed to prove that this solution actually lives in $\DD^{2\gamma}_{X}$.

The contractivity of $\MM_T$ is obtained in a similar way, now using the local Lipschitz property of $N$.
\end{proof}
\begin{rem} \label{rem nonlinearity}
Assumption $N \in \Poly^{0,n}_{-\kappa,-\delta}$ leads to a small problem when we want for instance to take $\HH = L^2$ because then $N$ is nonlinear function that needs to act on the space of distributions $\HH_{-\kappa}$. One can actually remove this problem and show the existence of the solution in $\DD^{2\gamma}_{X}$ for $N \in \Poly^{0,n}_{0,-\delta}$. This can be achieved by first solving the equation in the spaces $$ \cD^{2\varepsilon}_{S,X}([0,T],\HH_{-2\gamma})\cap  \big(\hat{\CC}^\eta([0,T],\HH)\times L^\infty([0,T],\HH_{2\varepsilon-2\gamma}^{d})\big),$$ for some $\eta<\varepsilon<\gamma$ and then again show that all the regularities can be improved and that the solution is indeed in $\DD^{2\gamma}_{X}([0,T],\HH)$. We decided to avoid this and not to use even more norms on the different space time scales for simplicity.
\end{rem}

Since we proved that $\cD^{2\gamma,2\gamma,0}_{S,X}=\cD^{2\gamma,2\gamma,0}_{X}$ and since both integration and composition 
with smooth functions preserves $\cD^{2\varepsilon,2\gamma,0}_{X}$, one might ask why not to solve these equations in 
$\cD^{2\gamma,2\gamma,0}_{X}$ or even in $\cD^{2\gamma}_{X}$ in the first place. First if we would solve our equations in 
$\cD^{2\gamma}_{X}([0,T],\HH)$ with initial condition in $\HH$ then we will run into problem of estimating the term 
$\|S_t\xi-S_s\xi\|_{\HH}$.
This term would have to be bounded by $|t-s|^{2\gamma}$ which is not true for general 
$\xi \in \HH$ but true for $\xi \in \HH_{2\gamma}$. This suggests that one must look for the solution in the space 
like $\cD^{2\gamma,2\gamma,0}_{X}$. We believe that this indeed can be done. This approach would have an advantage 
that estimates on the composition with the regular function in space $\cD^{2\gamma,2\gamma,0}_{X}$ automatically 
follows from the usual estimate on the control rough paths. Nevertheless we decided to stick to the space 
$\cD^{2\gamma,2\gamma,0}_{S,X}$ because the operator $\ddh $ acts nicely on the integrals of the form 
$\int_0^tS_{t-s}Y_sdX_s$. Otherwise we would always have to deal with estimating two kinds of expressions: 
$\int_s^tS_{t-r}F(u_r)dX_r$ and $\int_0^s(S_{t-r}-S_{s-r})F(u_r)dX_r$. In conclusion, it seems that working in 
spaces $\cD^{2\gamma,2\gamma,0}_{S,X}$ and $\cD^{2\gamma,2\gamma,0}_{X}$ is essentially equivalent but in one space 
it is easier to estimate integrals and in the other it is easier to estimate composition with the functions.

\subsection{Continuity of the solution map}
\label{sec:continuityRPDE}

In this subsection we are going to use stability results for integration and composition in order to prove continuity of the solution map of the RPDEs (which in the classic literature for solutions of RDE's is called It\^{o}-Lyons map).
\begin{thm}[Stability of solution to RPDE] \label{stab_sol1}
Let $\gamma \in (1/3,1/2]$ and $\XX,\tilde{\XX} \in \cC^\gamma$. Let $\xi,\tilde{\xi} \in \HH$, let $F \in C^3_{-2\gamma,0}(\HH,\HH^d)$, and $N \in \Poly^{0,n}_{0,-\delta}(\HH)$ for some $n\geq1$ is a function of polynomial type for some $1-\delta > \gamma$. Define $(u,F(u))\in \DD^{2\gamma}_{X}([0,\tau_1),\HH)$ to be a maximal solution to the RPDE:
$$du_t = Lu_tdt +N(u_t)dt + F(u_t)dX_t\, ,\quad u_0 = \xi \in \HH;$$
and $(v,F(v))\in \DD^{2\gamma}_{S,\tilde{X}}([0,\tau_2),\HH)$ to be a maximal solution of the same RPDE but driven by the rough path $\tilde{\XX}$ and initial condition $\tilde{\xi}$. Assume that  $\varrho_\gamma(X) = |X|_\gamma + |\X|_{2\gamma} < M$ and $\| \xi\|  <M$ and same with $\tilde{\XX}$ and $\tilde{\xi}$. Then for every $1/3 < \varepsilon < \gamma$ and $0\leq \eta < 3\varepsilon-2\gamma$ there exists time $T< 1\wedge \tau_1 \wedge \tau_2$ such that for the following seminorm taken with respect to this time $T$ we have:
\begin{equ} 
d_{2\varepsilon,2\gamma,\eta}(u,v) \leq C_{M} (\varrho_\gamma(\XX,\tilde{\XX}) + \| \xi-\tilde{\xi}\|). \label{e:cty1}
\end{equ}
Moreover if both solutions are global (i.e. $\tau_1 = \tau_2 = \infty$) then \eqref{e:cty1} holds for all $T > 0$.
\end{thm}
\begin{proof}
Note that continuity of the solution is proven in a bit worse H\"{o}lder regularity, but the space regularity remains the same. Moreover case when $\eta = 0$ is immediate by the if we prove the case for $\eta >0$ simply because $d_{2\varepsilon,2\gamma,0}(u,v) \lesssim_{T,\eta} \|\xi-\tilde{\xi}\| + d_{2\varepsilon,2\gamma,\eta}(u,v)$. First we will take $T$ small enough such that both solutions $u$ and $v$ satisfy for some constant $C_{F,N,\xi,\tilde{\xi}}$.
$$\| (u,F(u))\| _{\DD^{2\gamma}_{X}},\| (v,F(v))\| _{\DD^{2\gamma}_{\tilde{X}}} \leq C_{F,N,\xi,\tilde{\xi}}.$$
The fact that such $T$ exists was shown in the proof of invariance in Theorems~\ref{RPDE1} and~\ref{RPDE2}. This guarantees that all bounds with respect to $M$ in Lemma~\ref{stab_int} and Lemma~\ref{stab_comp} are satisfied and moreover the right hand side of inequality~(\ref{e:cty1}) is independent of solutions $u$ and $v$. From now on we will use without further mentioning that $\| F(\xi)-F(\tilde{\xi})\|  \lesssim_F\| \xi - \tilde{\xi}\|$.

First let's write $U_t = \int_0^t S_{t-r}N(u_r)dr$ and $V_t  = \int_0^t S_{t-r}N(v_r)dr$. Recall that $R^U = \ddh U$ and $R^V = \ddh V$ and $V' = U' = 0$. Since $N$ is locally Lipschitz then similarly as in Theorem~\ref{RPDE2} we can show that we can pick a $\beta>0$ such that $1+\beta-\delta-2\varepsilon>0$ and: 
\begin{equs}
	d_{2\varepsilon,2\gamma,\eta}(U,V) &= \|U-V\|_{\eta} + \| 0-0\| _{\infty}+ \| 0-0\| _{\varepsilon,-2\gamma} + |R^U-R^V|_{2\varepsilon,-2\gamma}\\
	&\lesssim(T^{1-\delta-\gamma} +T^{1+\beta-\delta-2\varepsilon} )\| u-v\| _{\infty} (1+\| u\| _{\infty}+\| v\| _{\infty})^{n-1} \\
	&\lesssim_M \| \xi - \tilde{\xi}\|   +d_{2\varepsilon,2\gamma,\eta}(u,v)T^{\sigma};
\end{equs}
for some $\sigma > 0$. In the last step we have used inequality $\|Y\|_{\infty} \lesssim \|Y_0\| + \|Y\|_{\eta}T^\eta$ and that $T\leq 1$. Denote $Z_t = \int_0^t S_{t-r}F(u_r)dX_r$ and $W_t = \int_0^t S_{t-r}F(v_r)d\tilde{X}_r$ as well as $\Xi_t = F(u_t)$ and $\tilde{\Xi}_t = F(v_t)$. Note that $Z' = \Xi_t$ and $W' = \tilde{\Xi}_t$.  
\begin{equs}
	d_{2\varepsilon,2\gamma,\eta}(Z,W) &= \|Z-W\|_{\eta} + \|\Xi-\tilde{\Xi}\|_{\infty} + d_{X,\tilde{X},2\varepsilon}(Z,W) \\ 
	&\lesssim \varrho_\gamma(\XX,\tilde{\XX})  + \| \xi - \tilde{\xi}\| +d_{2\varepsilon,2\gamma,0}(\Xi,\tilde{\Xi})T^{\gamma-\varepsilon} + \|\Xi-\tilde{\Xi}\|_{\infty} \\
	&\lesssim \varrho_\gamma(\XX,\tilde{\XX})  + \| \xi - \tilde{\xi}\|+d_{2\varepsilon,2\gamma,\eta}(u,v)T^{\gamma-\varepsilon}+\|\Xi-\tilde{\Xi}\|_{\infty}.
\end{equs}
We have used above Lemma~\ref{stab_int} for the first inequality and Lemma~\ref{stab_comp} for the second inequality. To deal with the term $\|\Xi-\tilde{\Xi}\|_{\infty}$ we use 
$$\|F(u_t)-F(v_t)\| \lesssim_F \|u_t-v_t\| \lesssim \| \xi - \tilde{\xi}\|+\|u-v\|_{\eta}T^\eta,$$ 
to finally deduce that for some potentially even smaller $\sigma>0$:
$$d_{2\varepsilon,2\gamma,\eta}(Z,W) \lesssim_{F,M} \varrho_\gamma(\XX,\tilde{\XX})  + \| \xi - \tilde{\xi}\|+d_{2\varepsilon,2\gamma,\eta}(u,v)T^{\sigma}.$$
Now $(u,F(u))$ is a fixed point of the map:
$$\MM_T(u,u')_t = (S_t\xi + \int_0^t S_{t-r}N(u_r)dr+ \int_0^t S_{t-r}F(u_r)dX_r,\, F(u_t)) \in \DD^{2\gamma}_{X}([0,T],\HH),$$ 
and similarly for $(v,F(v))$ with $\tilde{X}$. Putting the bounds on $d_{2\varepsilon,2\gamma,\eta}(Z,W)$ and on $d_{2\varepsilon,2\gamma,\eta}(U,V)$ together we get: 
\begin{equs}
	d_{2\varepsilon,2\gamma,\eta}(u,v) &\leq d_{2\varepsilon,2\gamma,\eta}(Z,W)+d_{2\varepsilon,2\gamma,\eta}(U,V) + \|\xi-\tilde{\xi}\|\\
	 &\leq C_{M} \big(\varrho_\gamma(\XX,\tilde{\XX})  + \| \xi - \tilde{\xi}\|   +d_{2\varepsilon,2\gamma,\eta}(u,v)T^{\sigma}\big).
\end{equs}
If we take $T = \tau(M,F,\varepsilon,\gamma,\eta)$ sufficiently small such that $C_{M} T^{\sigma} \leq 1/2$ then we get:
$$d_{2\varepsilon,2\gamma,\eta}(u,v) \leq 2C_{M} (\varrho_\gamma(\XX,\tilde{\XX}) + \| \xi-\tilde{\xi}\| ).$$

Now if we know that both solutions $(u,F(u))$ and $(v,F(v))$ are global in time we can iterate stability result~\ref{e:cty1} in order to obtain it for an arbitrary $T>0$. This can be done by investigating more carefully the proof of Theorems~\ref{RPDE1} and~\ref{RPDE2} and observing that the inverse of time $T_0$ within which invariance and contraction holds bounded from above by some powers of $\varrho_\gamma(X)$ and $\|\xi\|$. This then allows to show that we can bound from above the number of times we would need to iterate~\ref{e:cty1} to get to the time $T$. 
\end{proof}

 Next we state that for every global solution $u$ with the noise $X$ and initial condition $\xi$ there is a small ball around $X$ and small ball around $\xi$ such that for every noise $\tilde{X}$ and initial condition $\tilde{\xi}$ inside these balls the size of solution $v = v(\tilde{X},\xi)$ is not much bigger than the size of the solution $u$. Proof of the following proposition is quite standard and again uses the iteration of~\ref{e:cty1}.

\begin{prop} \label{stab_sol3}
Let $\gamma \in (1/3,1/2]$ and $\XX\in \cC^\gamma$ $\xi \in \HH$ and $(u,F(u)) \in \DD^{2\gamma}_{{X}}(\R_+,\HH)$ be a global solution to~(\ref{e:rpde2}) with $F$ and $N$ as in Theorem~\ref{stab_sol1}. Let $T \in [0,\infty)$ and assume that $\| u\| _{\infty,[0,T]} \leq M$ and $\varrho_{\gamma,[0,T]}(X) \leq R$ for some $M,R > 0$. Then there exists $\sigma = \sigma(M,R,T)$ such that for all $\tilde{\XX}$ and $\tilde{\xi}$ with $\varrho_\gamma(\XX,\tilde{\XX}) + \| \xi-\tilde{\xi}\|  \leq \sigma$ and for every $1/3 < \varepsilon < \gamma$ and $0< \eta < 3\varepsilon-2\gamma$ we have that the solution $(v,F(v))$ of the RPDE~(\ref{e:rpde2}) with data $(\tilde{\XX},\tilde{\xi})$ satisfies: 
\begin{equs} 
d_{2\varepsilon,2\gamma,\eta}(u,v) &\leq C (\varrho_\gamma(\XX,\tilde{\XX}) + \| \xi-\tilde{\xi}\| ), \label{e:cty2}\\
\| u-v\| _{\eta, [0,T]} &\leq C (\varrho_\gamma(\XX,\tilde{\XX}) + \| \xi-\tilde{\xi}\| )\quad \text{and}\quad \| v\| _{\infty,[0,T]} \leq 2\| u\| _{\infty,[0,T]}. 
\end{equs}
Constant $C = C(M,R,T)$ is locally bounded function in all three variables.
 \end{prop}

 \subsection{Solutions to SPDEs} \label{SPDEs}
 
First of all we define the spaces of controlled rough paths that are allowed to blow up in finite time. For every Banach space $V$ define a new Banach space $\bar{V} = V \sqcup \{\infty\}$. The topology on this space is induced by the basis containing open balls of $V$ and the sets of the form $\{v \in V : \|v\|_V \geq N \} \sqcup \{\infty\}$ for every $N > 0$.
Using this we define the space of controlled rough paths that might blow up in finite time:  
\begin{equs}
	\hat{\cD}^{2\gamma,\beta,0}_{S,X}(&\R_+,\HH_\alpha) = \Big\{(u,u') \in C(\R_+,\bar{\HH}_{\alpha+2\gamma})\times C(\R_+,\bar{\HH}^d_{\alpha+2\gamma}): \exists\,\tau > 0 \\
	&(u,u')\restriction_{[0,\tau)} \in \cD^{2\gamma,\beta,0}_{S,X}([0,\tau ),\HH_\alpha),\quad (u_t,u'_t) = (\infty,\infty)\quad  \forall t \geq \tau \Big\}.
\end{equs}
The $\tau$ in the above definition is denoting a blow up time of $(u,u')$ and can be taken $\infty$ for the controlled rough paths which have finite $\cD^{2\gamma,\beta,0}_{S,X}$ norm on every compact interval. 

All our analysis was purely deterministic so far. There is a wide class of Gaussian processes that can be lifted almost surely to a rough path and thus our theory is giving a pathwise notion of solution for such SPDEs driven by these Gaussian processes. 
Equations that we are going to investigate are driven by Brownian motion and we now briefly recall how one defines 
a Brownian rough path. The following definition requires a proof, which can be found in \cite{friz}.
\begin{defn}
	\textbf{(i)} Let $(B_t)_{t\geq0} : \Omega \to \R^d $ be a $d$-dimensional Brownian motion defined on the probability space $(\Omega,\FF,\Prob)$ and define $\B^\Ito_{t,s} := \int_s^t \delta B_{r,s}\otimes dB_r$ as an It\^{o} integral. Then $\forall \gamma\in (1/3,1/2)$ and $T > 0$ for a.e. $\omega \in \Omega$
	$$ \textbf{B}^\Ito(\omega)=(\delta B(\omega),\B^\Ito(\omega)) \in \cC^{\gamma}([0,T],\R^d) \;.$$
	\textbf{(ii)} In addition define $\B^\Strat_{t,s} = \B^\Ito_{t,s} + \frac{1}{2}(t-s)\id$ then for a.e. $\omega \in \Omega$
	$$ \textbf{B}^\Strat(\omega)=(\delta B(\omega),\B^\Strat(\omega)) \in \cC_g^{\gamma}([0,T],\R^d)\;.$$
\end{defn}
One would like to know that the rough integrals defined earlier against these Brownian lifts 
coincide with It\^{o} (resp.\ Stratonovich)  integrals for a suitable class of integrands:
\begin{prop} \label{stoch_int}
	Let $(B_t)_{t\geq0} : \Omega \to \R^d $ be a $d$-dimensional Brownian motion defined on the filtered 
	probability space $(\Omega,(\FF_t)_{t \ge 0},\Prob)$ and let $(Y,Y') \in \hat{\cD}^{2\varepsilon,2\gamma,0}_{S,B}(\R_+,\HH_{-2\gamma})$ be such that $(Y,Y')$ is adapted to the filtration $(\FF_t)_{t\geq 0}$ and such that, for every $L>0$ there exists a stopping time $T_L$ and $\hat L > 0$ such that 
	 $\|Y_t\| + \|Y_t'\| \le \hat L$ almost surely for $t \le T_L$. For 
	$L > 0$ and $t>0$, set
	\begin{equ}\label{ito_L}
		Z^{L}_{t} = \int_0^t S_{t-r}Y_r\ind_{\{r < T_L\}} dB_r\;,
	\end{equ}
	where the integral is an It\^o integral.
	Then the process $Z^L$ has a continuous version (still denoted by $Z^L$) such that for any random time $t(\omega)$ with $0 \leq t(\omega) < T_L(\omega)$, the following equality with the rough integral holds almost surely:
	\begin{equ}\label{ito=rough}
		Z^{L}_{t(\omega)}(\omega) =\int_0^{t(\omega)} S_{t(\omega)-r}Y_r(\omega) d\textbf{B}^\Ito(\omega).
	\end{equ}
\end{prop}
\begin{proof}
	Existence of the It\^o integral~\eqref{ito_L} follows from 
	$\int_0^{t}\|Y_s\|^2 \ind_{\{r < T_L\}}\,ds \leq \hat L^2 t$ and, since $T_L$ is a stopping time, the integrand is adapted to the 
	Brownian filtration. Let $\CP_n = \{s^n_k\}^\infty_{k=0}$ be a sequence of increasing countable subsets of $\R_+$ such that $\bigcup_n \CP_n$ is dense in $\R_+$ and $s^n_k < s^n_k+1$ for all $n,k \in \mathbb{N}$.  Denote by $\pi_n = \{[s^n_k,s^n_{k+1}]: s^n_k \in \CP_n, k\in \mathbb{N}\}$ the sequence of partitions formed from $\CP_n$ and $|\pi_n| = \sup_{k \geq 1}\{|s^n_{k+1}-s^n_k|\}$ is the size of partition. It follows that $Z^{L}_{t}$ is defined as a limit in probability:
	\begin{equ}[e:Itoprob]
		Z^{L}_{t} = \lim_{n \to \infty}\sum_{\substack{[u,v]\in \pi_n \\ u < t}} S_{t-u}Y_u\delta B_{v,u}\ind_{\{u < T_L\}} \; ,
	\end{equ}
	We can now extract a subsequence of partitions (which we still denote $\CP_n$) such that the above limit holds almost surely. On the other hand, since $(Y,Y') \in \hat{\cD}^{2\varepsilon,2\gamma,0}_{S,B}(\R_+,\HH_{-2\gamma})$, the rough integral 
	\begin{equ}
		\tilde{Z}^{L}_{t}(\omega) =\int_0^{t\wedge T_L(\omega)} S_{t-r}Y_r(\omega) d\textbf{B}^\Ito(\omega),
	\end{equ}
	exists and one can verify that it is equal to:
	\begin{equs}
		\tilde{Z}^{L}_{t}&= \lim_{n \to \infty}\sum_{\substack{[u,v]\in \pi_n \\ u < t\wedge T_L}}\bigl(S_{t-u}Y_u \delta B_{v,u} + S_{t-u}Y'_u \B^\Ito_{v,u}\bigr)\;.
	\end{equs} 
	We can therefore easily see that for every $t > 0$ the $L^2(\Omega)$ norm of the difference of these two integrals is:
	\begin{equ}[ito-error]
		\Big|Z^L_t - \tilde{Z}^{L}_{t}\Big|_{L^2(\Omega)} = \Big|\lim_{n \to \infty} \sum_{\substack{[u,v]\in \pi_n \\ u < t\wedge T_L}} S_{t-u}Y'_u\B^\Ito_{v,u} \Big|_{L^2(\Omega)}.
	\end{equ}
	%
	We will show now that the right hand side is zero. Define a (discrete time) martingale started at $M^n_0 = 0$ and with increments $M^n_{k+1} - M^n_k = S_{t-s^n_k}Y'_{s^n_k}\ind_{\{s^n_k < t\wedge T_L\}} \B^\Ito_{s_{k+1},s^n_k}$.
	\begin{equs}
		\Big|\sum_{\substack{[u,v]\in \pi_n \\ u < t\wedge T_L}} S_{t-u}Y'_u\B^\Ito_{v,u}\Big|^2_{L^2(\Omega)} &= \Big|\sum_{k=0}^{\infty} (M^n_{k+1} - M^n_k)\Big|^2_{L^2(\Omega)} = \sum_{k=0}^{\infty} |M^n_{k+1}- M^n_k|^2_{L^2(\Omega)}\\ 
		&\lesssim \hat L^2 \sum_{k=0}^{\infty}|\B^\Ito_{s^n_{k+1},s^n_k}\ind_{\{s^n_k < t\}} |^2_{L^2(\Omega)} \lesssim L^2 t |\pi_n|\; .
	\end{equs}
	We use the fact that all the infinite sums above are finite because of the presence of the indicator function. Moreover the last inequality is true because the Brownian scaling gives $|\B^\Ito_{v,u}|^2_{L^2(\Omega)} \lesssim |v-u|^2$. Since $\bigcup_n \CP_n$ is dense in $\R_+$ we have $|\pi_n| \to 0$ as $n \to \infty$. Therefore by Fatou's lemma right hand side of~\eqref{ito-error} is indeed zero thus showing that for all $t>0$ we have almost surely $Z^L_t = \tilde{Z}^L_t$. Now one can choose a continuous version of the It\^o integral $Z^L$ which is still equal almost surely to $\tilde{Z}^L_t$ for every $t > 0$. We can therefore evaluate $Z^L_t$ at a random time $0 \leq t(\omega) < T_L(\omega)$ to deduce that the identity
	\begin{equ}
		Z^L_{t(\omega)}(\omega) = \tilde{Z}^L_{t(\omega)}(\omega) = \int_0^{t(\omega)} S_{t(\omega)-r}Y_r(\omega) d\textbf{B}^\Ito(\omega)
	\end{equ}
holds almost surely.
\end{proof}

Before will now formalize the notion of a local in time solution for an It\^o SPDE.
\begin{defn}\label{ito_local} Let $(B_t)_{t\geq0} : \Omega \to \R^d $ be a $d$-dimensional Brownian motion defined on the filtered 
	probability space $(\Omega,(\FF_t)_{t \ge 0},\Prob)$. Let $\xi \in \HH$, $\delta \in [0,1)$ and consider locally Lipschitz continuous maps $N : \HH \to \HH_{-\delta}$ and $F: \HH \to \HH^d$.
	\begin{enumerate}[label=(\roman*)] 
		\item A local mild solution to an It\^o SPDE
		\begin{equ}\label{ito}
			du_t = Lu_t dt + N(u_t) dt + F(u_t) dB_t\; , 
		\end{equ}
		is a continuous stochastic process $u$ together with the stopping time $\tau$ such that almost surely on the event $\{t \leq \tau\}$, $u_t$ satisfies
		\begin{equ}\label{ito_mild}
			u_t = S_t \xi +  \int_0^t S_{t-s}N(u_s) \ind_{\{s < \tau\}}ds +\int_0^t S_{t-s}F(u_s) \ind_{\{s < \tau\}}dB_s\, ,
		\end{equ}
	where the last integral is taken in the sense of It\^o. We furthermore impose that there exists $L>0$ such that $\sup_{{0\leq t \leq \tau}} \|u_t\| \leq L$ almost surely. 
		\item We say $(u, \tau)$ is a maximal mild solution of~\eqref{ito} if $\lim_{t \to \tau} \|u_t\| = \infty$ almost surely and there exists a sequence of local mild solutions $(u^n, \tau^n)$ with increasing $\tau^n$ such that $\lim_{n \to \infty} \tau^n = \tau$ almost surely and $u^n_t = u_t$ almost surely on $\{t < \tau^n\}$.
	\end{enumerate}
\end{defn}
\begin{thm} \label{spde=rpde}
	Let $\xi \in \HH$ and functions $F$ and $N$ be as in Theorem~\ref{RPDE2}. Let $(B_t)_{t\geq0} : \Omega \to \R^d $ be a $d$-dimensional Brownian motion defined on the probability space $(\Omega,\FF,\Prob)$. Then there exist random blow up times $\tau_1, \tau_2 > 0$ and controlled rough paths $(u,F(u)) \in \DD^{2\gamma}_{B}([0,\tau_1),\HH)$, $(v,F(v)) \in \DD^{2\gamma}_{B}([0,\tau_2),\HH)$ such that they are almost surely maximal solutions of~\eqref{e:rpde1} with $\XX$ replaced by $\textbf{B}^\Ito$ and $\textbf{B}^\Strat$ respectively.
	
 In addition the above solutions are adapted processes when viewed as elements of $\hat{\cD}^{2\varepsilon,2\gamma,0}_{S,B}\big(\R_+,\HH_{-2\gamma}\big)$. As a consequence the following holds:
	\begin{enumerate}[label=(\roman*)]
		\item\label{ito_spde*} $(u,\tau_1)$ is a maximal mild solution to the It\^{o} SPDE:
		\begin{equ}\label{ito_spde}
			du_t = Lu_tdt + N(u_t)dt+F(u_t)dB_t\,,\quad u_0 =\xi \in \HH\, ,
		\end{equ}
		\item\label{strat_spde*} $(v,\tau_2)$ is a maximal mild solution to the It\^{o} SPDE:
		\begin{equ}\label{strat_spde}
			dv_t = Lv_tdt + \big(N(v_t)+\textstyle{\frac{1}{2}}DF(v_t)F(v_t)\big)dt+F(v_t)dB_t\,,\quad u_0 =\xi \in \HH\, .
		\end{equ}
	\end{enumerate}
\end{thm}
\begin{proof}
	We first show the result for $(u,F(u))$. Local solution theory for~\eqref{e:rpde1} with $\XX$ replaced by almost every realization of $\textbf{B}^\Ito$ is provided by Theorem~\ref{RPDE2}. The fact that $\tau_1$ is a stopping time is easy to verify. Note that the map 
	$$B\restriction_{[0,t]} \mapsto (B,\B^\Ito)\restriction_{[0,t]} \in \cC^{\gamma}([0,t],\R^d)$$
	is measurable. For almost every $\omega$ and every $t < \tau_1(\omega)$, the solution $(u,F(u)) \in \hat{\cD}^{2\varepsilon,2\gamma,0}_{S,B}\big([0,t],\HH_{-2\gamma}\big)$ to \eqref{e:rpde1} is a continuous image of 
	the noise $(B,\B^\Ito)\restriction_{[0,t]}$. Viewing $(u_t(\omega),F(u_t(\omega)))$ as an element of 
	$\hat{\cD}^{2\varepsilon,2\gamma,0}_{S,B(\omega)}\big(\R_+,\HH_{-2\gamma}\big)$ we deduce that is 
	adapted to 
	$$\sigma( B_{s,r}, \B^\Ito_{s,r}: 0\leq r\leq s\leq t) = \sigma(B_s : 0\leq s \leq t) = \FF_t\,.$$
	Let $L > 0$
	and define a stopping time $T_L = \inf\{t : \|u_t\| \geq L \}$ then the local boundedness of $F$ implies that there exists $\hat L >0$ such that almost surely for $t < T_L$:
	\begin{equ}\label{bddns_of_integrand}
		\|F(u_t)\| + \|DF(u_t)F(u_t)\| < \hat L\; .
	\end{equ}
	 For $t > 0$ define the process $u^L_t$ as:
	\begin{equ}
		u^L_t(\omega) = S_{t}\xi + \int_0^{t} S_{t-s}N(u_s(\omega))\ind_{\{s < T_L(\omega)\}} ds + \int_0^{t} S_{t-s}F(u_s(\omega))\ind_{\{s < T_L(\omega)\}} dB_s\,(\omega)\,,
	\end{equ}
where the existence of the It\^o integral is guaranteed by an almost sure bound~\eqref{bddns_of_integrand}.
By definition of our notion of solution to~\eqref{e:rpde1}, we furthermore know that, for any  
(random) time $t(\omega) < T_L(\omega)$, one has the identity
	\begin{equ}
		u_{t(\omega)}(\omega)=S_{t(\omega)}\xi + \int_0^{t(\omega)} S_{t(\omega)-s}N(u_s(\omega))ds + \int_0^{t(\omega)} S_{t(\omega)-s}F(u_s(\omega))d\textbf{B}^\Ito_s(\omega)\,.
	\end{equ}
By Proposition~\ref{stoch_int} and equation~\eqref{bddns_of_integrand}, we conclude that, almost surely, $u_{t(\omega)}(\omega) = u^L_{t(\omega)}(\omega)$, provided that 
we consider a continuous version of $u^L$. This also shows that $(u^L, T_L)$ is a local mild solution of the It\^o SPDE~\eqref{ito_spde}. Whenever $\|u_t\|$ is finite we can always restart the equation~\eqref{e:rpde1} (with $\XX$ replaced by $\textbf{B}^\Ito$) with initial condition $u_t$ and extend the solution further in time therefore almost surely $T_L \to \tau_1$ as $L \to \infty$. Moreover $T_L$ clearly increases as $L$ increases and $u^L_t = u_t$ on $\{t < T_L\}$ thus showing that $(u,\tau_1)$ is indeed a maximal solution of~\eqref{ito_spde}.

	Regarding the solution $(v,F(v))$ the proof is the same once we notice that
	\begin{equs}
		\int_0^t S_{t-s}F(v_s)d\textbf{B}^\Strat_s &= \int_0^t S_{t-s}F(v_s)d\textbf{B}^\Ito_s + \frac{1}{2}\int_0^t S_{t-s}v'_s ds \label{strat - ito}\\
		&= \int_0^t S_{t-s}F(v_s)d\textbf{B}^\Ito_s + \frac{1}{2}\int_0^t S_{t-s}DF(v_s)F(v_s) ds\; , 
	\end{equs}
	and that all the above integrals make sense as elements of $\HH$. Then we apply 
	Proposition~\ref{stoch_int} again for the rough integral with respect to $\textbf{B}^\Ito$ 
	in~\eqref{strat - ito} and the result follows.
\end{proof}
\begin{rem} \label{rem 4}
Whenever one develops a new approach to solve SPDEs, it is natural to ask that these solutions coincide with 
solutions given by other approaches, whenever both apply. This theorem tells us that indeed this is true. 
For our results, this theorem serves another role: it allows us to transfer properties known for the solutions 
to SPDEs in It\^{o} (or Stratonovich) form to the RPDE solution. This is useful since 
it might be simpler to obtain a priori estimates, global existence and Malliavin differentiability for the SPDEs 
rather than the corresponding RPDEs. For instance, global existence for almost every realisation of 
Brownian motion for the It\^o solutions can be used to show that continuity of the solution map~\eqref{e:cty1} is true for all $T>0$.
\end{rem}

 \subsection{Malliavin differentiability and the Jacobian}

In this subsection we show the Malliavin differentiability of the solutions to RPDEs 
driven by general Gaussian rough paths, using only that the solution does not blow up in finite time. 
Unfortunately the method is non constructive and only gives the knowledge that Malliavin 
derivative exists and lies in the Shigekawa-Sobolev space $\mathbb{D}^{1,2}_{\loc}$. In particular, it 
does not automatically imply that this Malliavin derivative is a controlled rough path itself and \slash or that it solves 
some RPDE. Nevertheless this is not so important for our analysis since our main result regarding the non-degeneracy 
of the Malliavin matrix does not require Malliavin differentiability per se. This is due to the fact that we 
will define Malliavin matrix only using the existence of linearisation of solution. Moreover if say equation 
of our interest is driven by the Brownian motion and we can show Malliavin differentiability then Malliavin 
derivative satisfies an SPDE and therefore almost surely the RPDE by Theorem~\ref{spde=rpde}. This restriction 
to the Brownian case is also performed because for general rough paths the a priori bounds are not easy to obtain. 
Despite the fact that we will later simply assume the Malliavin differentiability and won't explicitly use Theorem~\ref{mal_diff} we still present it together with the proof as a result on its own. 
Before we proceed, we quickly recall the Cameron-Martin theory for general centred Gaussian rough paths.

 Let $\Omega = C([0,T],\R^d)$ and let $X: \Omega \times [0,T] \to \R^d$ be the a canonical centred Gaussian process so that $X_t(\omega) = \omega(t)$. The Gaussian law of $X$ is completely determined by its covariance function $R_X: [0,T]^2 \to \R^{d\times d}$. For $p\geq 1$ define the 2D $p$-variation of $R$  on a rectangle $I\times I' \subseteq [0,T]^2$ to be:
 $$ \|R_X\|_{p, I\times I'} :=\Big( \sup_{P \in \pi(I)\atop {P' \in \pi(I')}} \sum_{[s,t] \in P \atop [s',t'] \in P'} |\E[\delta X_{t,s}\otimes \delta X_{t',s'}]|^p \Big)^{1/p}.$$
$\pi(I)$ denotes here the partitions of $I$. Similarly one can define $\|R_{X^i}\|_{p, I\times I'}$.
The Cameron-Martin space $\CM_T \subset C([0,T],\R^d)$ is a Hilbert space which consists of the paths $v_t = \E[ZX_t]$ for $Z $ 
lying in the first Wiener chaos $\mathcal{W}^1$ which is an $L^2$-closure of $\text{span}\{X^i_t\,:\,t\in [0,T],\; 1\leq i\leq d \}$. 
See \cite[Chap.~10, 11]{friz} for the description of the regularity of the Cameron Martin space and for precise conditions on  
the covariance function which guarantee that $X$ can be lifted almost surely to a rough path in a canonical way. In the case when 
the Gaussian process is a Brownian motion, we have $\CM_T = H^1([0,T],\R^d) = \{h: h(0) = 0\,\&\,\partial_t h \in L^2([0,T],\R^d) \}$.
 
 For $ \gamma \in (1/3,1/2)$ and a generic $(X,\X) \in \cC^\gamma([0,T],\R^d)$ let $h :[0,T] \to \R^d$ be sufficiently smooth, the translation operator of $\XX$ in the direction $v$ is defined by 
 $$T_h(\XX) := (X^h,\X^h),$$
for $X^h = X+\delta h$ and 
 $$\X^h_{t,s} = \X_{t,s} + \int_s^t \delta h_{r,s} dX_r + \int_s^t  X_{r,s} dh_r+\int_s^t \delta h_{r,s} dh_r.$$
 Here by sufficiently smooth we understand that all three integrals above make sense classically and moreover makes the operator $T_h$ a continuous map $\cC^\gamma$ to itself. In fact it is true for $h \in H^1$. Moreover if $\XX$ is the Brownian rough path (either It\^{o} or Stratonovich) we have that for almost every $\omega \in\Omega $ and every $h \in \CM_T$ we have:
\begin{equ} 
	T_h(\XX(\omega)) = \XX(\omega+h). \label{mal 1}
\end{equ}
 A similar result holds for a general Gaussian rough path with the regular enough covariance.
 
 Let $\XX$ be a centred Gaussian rough path which almost surely lies in $\cC^\gamma$ with covariance $R$ and let $\CM_T$ be its Cameron Martin space. Assume that for some $p \in [1,2)$ every $v \in \CM_T$ has finite $p$-variation $\|v\|_{\pvar,[s,t]}$ over the interval $[s,t] \subseteq [0,T]$ and it satisfies an inequality:
 \begin{equ} 
	\|h\|_{\pvar,[s,t]} \lesssim_{T,\gamma} \|h\|_{\CM_T} |t-s|^\gamma. \label{mal 2}
 \end{equ}
Then if $X$ almost surely satisfies equality~(\ref{mal 1}) it is easy to show that, almost surely for every $h \in \CM_T$,
\begin{equ} 
\|X^h-X\|_\gamma \lesssim \|h\|_{\CM_T}\quad \text{and}\quad \|\X^h-\X\|_\gamma \lesssim \|h\|_{\CM_T} (\|h\|_{\CM_T} + \|X\|_\gamma). \label{mal 3}
\end{equ}
We quickly recall the notion of Malliavin differentiability. Let $(\Omega,\CM,\Prob)$ be an abstract Wiener space where $\CM$ is the Cameron-Martin space. For any Hilbert space $\HH$ we say that the random variable $Y:\Omega \to \HH$ is Malliavin differentiable if there exists a random element $\cD Y: \Omega \to \CM\otimes \HH$ such that for all $h \in \CM$ the limit in probability 
$$\langle \cD Y,h\rangle_{\CM} = \lim\limits_{\varepsilon \to 0} \varepsilon^{-1}(T_{\varepsilon h}Y-Y),$$
exists. Here $(T_{h}Y)(\omega) = Y(\omega+h)$. This gives rise to a closed unbounded linear operator
$$\cD:L^2(\Omega, \HH)\to L^2(\Omega,\CM\otimes \HH).$$
The domain of this operator is denoted by $\mathbb{D}^{1,2}$. If we denote by $\FF$ the $\sigma$-algebra of the Wiener space $(\Omega,\CM,\Prob)$ then we say that $Y \in \mathbb{D}^{1,2}_{\loc}$ if there exists a sequence $(\Omega_n,Y_n)_{n\geq1} \subseteq \FF\times \mathbb{D}^{1,2}$ such that $\Omega_n \uparrow \Omega$ and $Y = Y_n$ almost surely on $\Omega_n$. See the book \cite{Nual:06} for an introduction to Malliavin Calculus. 

Having all this at hand we are ready to present a general statement on the Malliavin differentiability of the solution to the Stochastic RPDE driven by quite general Gaussian rough path and given that this solution does not explode until some deterministic time $T$.
 \begin{thm}[Malliavin differentiability] \label{mal_diff}
 	Let $(X_t)_{t\in [0,T]}$ be a d-dimensional, continuous Gaussian process with independent components defined on the probability space $(\Omega,\FF,\Prob)$. Let the covariance $R$ of $X$ be such that there exist $M<\infty$ and $p \in [1,2)$ such that for $i \in \{1,\ldots,d\}$ and $[s,t] \subseteq [0,T]$,
 	$$\|R_{X^i}\|_{p,[s,t]^2} \leq M |t-s|^{1/p}.$$
 	Let $\gamma \in (\frac{1}{3},\frac{1}{2p})$ and for almost every $\omega$ let $(u(\omega),F(u(\omega))) \in \DD^{2\gamma}_{X(w)}([0,\tau(\omega)), \HH)$ be a local mild solution to the Stochastic RPDE:
 	$$du_t(w) = Lu_t(w)dt + N(u_t(w))dt+F(u_t(w))dX_t(w),\quad u_0 =\xi \in \HH,$$
 	such that $\|u\|_{\infty,[0,T]} <\infty$ almost surely. Then for all $0 \leq t \leq T$ the solution $u_t$ is Malliavin differentiable and $u_t \in \mathbb{D}^{1,2}_{\loc}$.
 \end{thm}
\begin{proof}
	Fix $\gamma  \in (\frac{1}{3},\frac{1}{2p})$, assumptions on the covariance $R$ guarantee (see \cite[Chap.~10]{friz}) that there is a canonical lift of $X$ to a rough path in $\cC^\gamma([0,T],\R^d)$ and that, for every $h \in \CM_T$,
\begin{equ}[e:boundh]
	\|h\|_{\pvar,[s,t]} \leq \|h\|_{\CM_T} |t-s|^{1/2p}.
\end{equ}
Moreover for almost every $\omega \in\Omega $ and every $h \in \CM_T$ we have
	$$T_h(\XX(\omega)) = \XX(\omega+h)\;,$$
	where $T_h$ denotes the translation operator by $h$ which is well-defined thanks to \eqref{e:boundh}
	(see Theorem 10.4, Proposition 11.2 and Theorem 11.5 in \cite{friz}). Because of this last property we will not distinguish between $T_h(\XX(\omega))$ and $\XX(\omega+h)$ and from now on we will simply write $\omega+h$ to denote any of them, we also abuse the notation and simply write $X(\omega) = \omega$.
	Fix $t \in [0,T]$ and define an event
	\begin{equ}
		B_n = \{\omega\;: \|u(\omega)\|_{\infty,[0,T]} \leq n/2, \varrho_\gamma(\omega) \leq n \}.
	\end{equ}
	 Let $\omega \in B_n$ then from Proposition~\ref{stab_sol3} we know that for such $\omega$ there exists a small number $\sigma_n$ (independent of $\omega$ and only dependent on $n$, $\xi$ and the equation itself) such that for all $h \in \Omega$ with $\varrho_\gamma(\omega+h,\omega) \leq \sigma_n$ we have $\|u(\omega+h)\|_{\infty,[0,T]} \leq 2 \|u(\omega)\|_{\infty,[0,T]}$. Assume that $\|h\|_{\CM_T} \leq \sigma'_n := \frac{\sigma_n}{2n}\wedge 1$ then by~(\ref{mal 3}) we have:
	 \begin{equ}
	 	\varrho_\gamma(\omega+h,\omega) \leq \|h\|_{\CM_T} (\|h\|_{\CM_T} + \varrho_{\gamma}(\omega)) \leq \sigma'_n(1+n) \leq \sigma_n\;.
	 \end{equ}
	This is showing that for the event
 \begin{equ}
 	A_n = \{\omega\;: \sup_{\|h\|_{\CM_T} \leq\sigma'_n}\|u_t(\omega+h)\| \leq n, \varrho_\gamma(\omega) \leq n \},
 \end{equ}
	we have $B_n \subseteq A_n$ and since  $\|u\|_{\infty,[0,T]} <\infty$ almost surely we have that $B_n \uparrow \Omega$ thus implying $A_n \uparrow \Omega$. It is also possible to find a $\sigma$-compact set $G_n \subset A_n$ such that $A_n/G_n$ is a null set. For $A \in \FF$ we define:
	$$\rho_A(\omega) = \inf\{\|h\|_{\CM_T}: h \in \CM_T\; \text{and}\; \omega+h \in A \}.$$
	Then define $u^n_t := \phi(\rho_{G_n}/\sigma'_n)u_t$ for $\phi \in C^\infty_c(\R)$ non-negative function such that $|\phi(t)| \leq 1$ and $|\phi'(t)| \leq 4$  for $\forall t$, $\phi(t) = 1 $ for $|t| \leq 1/3$ and $\phi(t) = 0$ for $|t| \geq 2/3$. Then we have that $u^n_t = u_t$ on $G_n$ and 
	$$\|u^n_t\| \leq \ind_{\{\rho_{G_n} \leq \frac{2\sigma'_{n}}{3}\}}\|u_t\| \leq n.$$
	Thus let $\|h\|_{\CM_T} \leq \sigma'_n/3$ therefore:
	\begin{equs}
		\|u^n_t(\omega+h) - u^n_t(\omega)\| &\leq \|(\phi(\rho_{G_n}(\omega+h)/\sigma'_n)-\phi(\rho_{G_n}(\omega)/\sigma'_n))u_t(\omega+h)\| \\
		&\quad+ \|\phi(\rho_{G_n}(\omega)/\sigma'_n)(u_t(\omega+h)-u_t(\omega))\|.
	\end{equs}
	For the second term we use Proposition~\ref{stab_sol3} and equation~(\ref{mal 3}) to deduce 
	$$\|\phi(\rho_{G_n}(\omega)/\sigma'_n)(u_t(\omega+h)-u_t(\omega))\| \lesssim_{T,n} \ind_{\{\rho_{G_n} \leq \frac{2\sigma'_{n}}{3}\}} \varrho_\gamma(\omega+h,\omega) \lesssim_{T,n} \|h\|_{\CM_T}.$$ 
	For the first term we proceed exactly as in Proposition 4.1.3 from \cite{Nual:06} to deduce that for $\|h\|_{\CM_T} \leq \sigma'_n/3$ 
	$$\|u^n_t(\omega+h) - u^n_t(\omega)\| \lesssim_{T,n} \|h\|_{\CM_T}.$$
	For the underlying constant which is deterministic and depends only on $T,\,n$, initial condition and the equation itself. Exercise 1.2.9 in \cite{Nual:06} shows that such local Lipschitz continuity guarantees that $(G_n,u^n_t)$ is the localizing sequence required for the definition of $\mathbb{D}^{1,2}_{loc}$. 
\end{proof}
Assume now that $F \in C^\infty_{-2\gamma,0}(\HH,\HH^d)$, and $N \in \Poly^{\infty,n}_{0,-\delta}(\HH)$ for some $n\geq1$. 
We then show that the Jacobian of the solution is related to the Malliavin derivative for RPDEs in the same way as for SDEs. 
The unique mild RPDE solution to the equation~(\ref{e:rpde2}) driven by the geometric rough path $\XX \in \cC^{\gamma}_g$ 
and passing through the point $u_{s} = \xi$ gives rise to the solution flow $\Phi^\xi_{t,s}(X) = (u_t,u'_t)$ up until the blow 
up time. More formally $$\Phi_{.,s} : \HH\times \cC^{\gamma}_g([s,T]) \to \hat{\DD}^{2\gamma}_X([s,T],\HH).$$
The space $\hat{\DD}^{2\gamma}_X$ is defined similarly to the space $\hat{\cD}^{2\gamma,\beta,0}_{S,X}$ from the previous subsection and takes into account the fact that the solution may blow up before the time $T$.
The derivative of the flow with respect to the starting point is called the Jacobian and is denoted by $J^X_{t,s}$. 
$$J^X_{t,s}\zeta := \frac{d}{d\varepsilon}\Phi^{\xi+\varepsilon\zeta}_{t,s}(X).$$
If $\XX$ is a rough path lifted from a smooth path and we can show that the solution to~(\ref{e:rpde2}) is global for every initial condition then from classical theory of PDEs one can show (say using implicit function theorem) that the Jacobian exists at every $\zeta \in \HH$ and will satisfy linearised equation:
\begin{equ}[e:jacobian]
dJ^X_{t,s}\zeta = LJ^X_{t,s}\zeta dt +DN(u_t)J^X_{t,s}\zeta dt + DF(u_t)J^X_{t,s}\zeta dX_t,\quad \, J^X_{s,s}\zeta = \zeta.
\end{equ}
Or in the mild form:
\begin{equ} 
	J^X_{t,s}\zeta = S_{t-s}\zeta + \int_s^tS_{r-s}DN(u_r)J^X_{r,s}\zeta dr + \int_s^tS_{r-s}DF(u_r)J^X_{r,s}\zeta dX_r. \label{mal 5}
\end{equ} 
From this representation and from the fact that $(u,u')$ is controlled by $\XX$ we deduce that if such $J^X_{t,s}\zeta $ satisfies this equation then it is also controlled by $\XX$.

Consider also the directional derivative of the flow in the direction of the noise:
$$D_h\Phi^\xi_{t,0}(X) = \frac{d}{d\varepsilon}\Phi^\xi_{t,0}(T_{\varepsilon h}X).$$
For $h$ sufficiently smooth. Once again if $X$ is lifted from a smooth path and $h$ is smooth then classical PDE theory is telling us that $D_h\Phi^\xi_{t,0}(X)$ exists and that due to variation of constants formula it satisfies:
$$D_h\Phi^\xi_{t,0}(X) =\int_0^t J^X_{t,s}F(u_s)dh_s.$$
The same passes to the geometric rough path in the limit:
\begin{prop} \label{duhamel}
Let $\gamma \in (1/3,1/2]$ and $\XX \in \cC^{\gamma}_g([0,T],\R^d)$. Assume that for $F \in C^\infty_{-2\gamma,0}(\HH,\HH^d)$, and $N \in 
\Poly^{\infty,n}_{0,-\delta}(\HH)$ solution to the equation~(\ref{e:rpde2}) exists in $\DD^{2\gamma}_X([0,T],\HH)$ for every initial condition $\xi \in \HH$. 
Let $h \in C^{\pvar}([0,T],\R^d)$ with complementary Young regularity $\gamma+1/p  > 1$. Then for $\forall \zeta \in \HH$ both $J^X_{t,s}\zeta$ 
$D_h\Phi^\xi_{t,0}(X)$ exist as elements of $\DD^{2\gamma}_X([0,T],\HH)$, $J^X_{t,s}\zeta$ satisfies RPDE~(\ref{mal 5}) and this Duhamel's formula holds:
\begin{equ} 
		D_h\Phi^\xi_{t,0}(X) =\int_0^t J^X_{t,s}F(u_s)dh_s, \label{mal 6}
\end{equ}
where the right hand side is well-defined as a Young integral.
\end{prop}
\begin{proof}
 We proceed like in the similar result for RDEs from \cite{friz}. Let $\XX^n = \XX^c(X^n)$ be canonical lift of a smooth path $X^n$ such that $\XX^n$ approximates $\XX$ with $\sup_n \varrho_\gamma(\XX^n) \leq \varrho_\gamma(\XX)$. Then the RPDE solution $(u^n,F(\dot{u}^n))$ to the equation~(\ref{e:rpde2}) driven by $\XX^n$ lies in $\DD^{2\gamma}_{X^n}([0,T],\HH)$ and converges to $(u,F(u))$ in the $d_{2\varepsilon,2\gamma,\eta}$ metric from the global continuity result of  Theorem~\ref{stab_sol1}. Now if we take a smooth approximation of $h$ as well say $h^m$ then from above we know that $D_{h^m}\Phi^\xi_{t,0}(X^n)$ and $J^{X^n}_{t,s}\zeta$ and equations~(\ref{mal 5}) and~(\ref{mal 6}) are satisfied for these smooth approximations. Passing to a limit as $n$ and $m$ go to infinity we obtain the desired result. 
\end{proof}
\begin{defn}\label{Mal_matr}
	Assume that~(\ref{e:rpde2}) has global solutions for every initial condition in $\HH$. Then, for any $t>0$, define the Malliavin matrix $\cM_t : \HH \to \HH$ by
	\begin{equ} 
		\langle\cM_t\varphi,\varphi\rangle = \int^t_0 \langle J^X_{t,s}F(u_s),\varphi\rangle^2 ds. \label{mal 7}
	\end{equ}
\end{defn}

\subsection{Smoothing property of the solution}

Due to the smoothing properties of the semigroup we expect that the solution is going to have a better spatial regularity after some time. In fact we are going to show that if we start our equation from $\xi \in \HH$ then the  solution immediately 
belongs to $\HH_\beta$ for every positive $\beta$. 
\begin{prop} \label{smoothing}
  	Let $\gamma \in (1/3,1/2]$ and $\XX \in \cC^\gamma$, $\xi \in \HH$. Let $F$ and $N$ be as in Theorem~\ref{stab_sol1} and $(u,F(u)) \in \DD^{2\gamma}_X([0,T],\HH)$ be a solution to the equation~(\ref{e:rpde2}). Denote $M := \|u\|_{\infty,[0,T]}$ then for every $0<t<T$ and $\beta > 0$ we have that
  	$(u,F(u)) \in \DD^{2\gamma}_X([t,T],\HH_{\beta})$ and moreover there exist $\sigma = \sigma(\delta,\gamma,\beta)$, $C_M = C(M,T,F,N,\XX)$ such that:
  	\begin{equ} 
  		\|u\|_{\infty,\beta,[t,T]} \lesssim t^{-\beta}  \|u\|_{\infty,[0,T]} + C_M T^\sigma. 
  	\end{equ}
  \end{prop}
\begin{proof}
	From the mild formula of solution we get that:
	$$u_t = S_{t-s}u_s + \int_s^tS_{t-r}N(u_r)dr + \int_s^tS_{t-r}F(u_r)dX_r.$$ 
	then taking $0 < \nu < \gamma$ we can deduce from the property of the semigroup and bounds on rough integration:
	$$\|u_t\|_{\HH_\nu} \lesssim |t-s|^{-\nu}\|u_s\| + T^{1-\delta-\sigma}(1+\|u\|_{\infty,[0,T]})^n+\varrho_\gamma(X)\|(u,F(u))\|_{\DD^{2\gamma}_X(\HH)}T^{\gamma-\nu}.$$
	From the similar arguments of iteration as in Theorem~\ref{stab_sol1} we can get that there exist $C'_M$ which is locally bounded in $M$ such that $\|(u,F(u))\|_{\DD^{2\gamma}_X(\HH)}T^{\gamma-\nu} < C'_M$ therefore we indeed deduce that for all $0 \leq s < t <T$
	\begin{equ} 
		\|u_t\|_{\HH_\nu} \lesssim |t-s|^{-\nu}M + C_M T^\sigma, \label{e:smoothing}
	\end{equ}
	where $C_M = \varrho_\gamma(X)C'_M + (1+M)^n$ and $T^\sigma = T^{1-\delta-\nu}\wedge T^{\gamma-\nu}$. Thus since both functions $F$ and $N$ act on $\HH_\nu$ for $\nu > 0$ the same way as on $\HH$ we can now solve our equation with this new initial condition $u_t \in \HH_\nu$ which from~(\ref{e:smoothing}) satisfies $\|u_t\|_{\HH_\nu} \leq C(t^{-\nu}M + C_M T^\sigma) =: M_t$ (here constant $C$ is the constant that is discarded by the $\lesssim$ sign). This gives us that there exists $\tau = \tau(M_t) > 0 $ such that the solution map is invariant and a contraction on the space $\DD^{2\gamma}_X([t,t+\tau],\HH_\nu)$ (or rather on some ball in this space). Thus we get that $(u,F(u)) \in  \DD^{2\gamma}_X([t,t+\tau],\HH_\nu)$. Now from~(\ref{e:smoothing}) we get that all $0\leq s < t+\tau$
	$$\|u_{t+\tau}\|_{\HH_\nu} \lesssim |t+\tau-s|^{-\nu}M + C_M T^\sigma.$$
	Picking $s = \tau$ we get again $\|u_{t+\tau}\|_{\HH_\nu} \leq M_t$ and thus starting the equation from the initial condition $u_{t+\tau}$ and since $\tau$ only depended on $M_t$ we can get the solution on $\DD^{2\gamma}_X([t+\tau,t+2\tau],\HH_\nu)$. Thus summarising we get $(u,F(u)) \in  \DD^{2\gamma}_X([t,t+2\tau],\HH_\nu)$ with again $\|u_{t+2\tau}\|_{\HH_\nu} \leq M_t$. Bootstrapping this further since $\tau>0$ is fixed we can get to time $T$ in finite number of these iterations and indeed get $(u,F(u)) \in  \DD^{2\gamma}_X([t,T],\HH_\nu)$.
	
	In order to prove this proposition for arbitrary $\beta > \nu > 0$ denote $t_0 = t\beta/\nu$ and $M_0 = M$. Without loss of generality let $\beta/\eta$ be an integer. Denote by $M_1 = \|u\|_{\infty,\nu,[t_0,T]}$ thus from~(\ref{e:smoothing}) 
	$$M_1 \lesssim t_0^{-\nu}M_0 + C_{M_0} T^\sigma.$$
	 and we have our solution $(u,F(u)) \in  \DD^{2\gamma}_X([t_0,T],\HH_\nu)$. Therefore proceeding exactly the same like in the beginning of the proof we get for all $t_0\leq s < r$
	$$\|u_r\|_{\HH_{2\nu}} \lesssim |r-s|^{-\nu}M_1 + C_{M_1} T^\sigma.$$
	As a consequence we get the solution $(u,F(u)) \in  \DD^{2\gamma}_X([2t_0,T],\HH_{2\nu})$ with $M_2 = \|u\|_{\infty,2\nu,[2t_0,T]}$. Recursively we get for every $n$ such that $nt_0 < T$, $(u,u') \in  \DD^{2\gamma}_X([nt_0,T],\HH_{n\nu})$ and with $M_n :=\|u\|_{\infty,n\nu,[nt_0,T]}$ we have the recursive inequality 
	$$M_n \leq C(t^{-\nu}M_{n-1}+C_{M_{n-1}}T^\sigma),$$
	where we took into account that $t = t_0\nu/\beta$. Solving this recursive inequality we get that for some other constants $\sigma = \sigma(n)$ and $C_M = C(M,n)$ we have
	$M_n \lesssim t^{-n\nu}M + C_M T^\sigma$. Picking $n=\beta/\nu$ we indeed get the result. 
\end{proof}

Note that we actually use in the proof that time interval on which the solution map is contractive and invariant also depends on the norms of $F$ and $N$ which can be different when acting on $\HH_\nu$ for different values of $\nu$. But since we want to make an improvement of space regularity only by the finite amount $\beta$ we can pick the largest value of the norms of $F$ and $N$ only up to their action on $\HH_\beta$.

This smoothing property is going to help us to overcome the issue that solution $(u,F(u)) \in \DD^{2\gamma}_X([0,T],\HH)$ to an RPDE lives in a space $\HH_{-2\gamma}$ as a controlled rough path while just as a function of time $u_t \in \HH$. This makes it difficult to investigate the properties of $u_t$ in the Hilbert space $\HH$ since in this space we cannot make an advantage of the fact that $u$ is actually a controlled rough path.

 \section{Rough Fubini Theorem} \label{fubini}

In this section we will only work with the usual notion of the controlled rough path. We consider a wide class of 
processes $Y_{t,s}$ which are controlled rough paths in both of their time directions. For such processes double rough integral can be defined and we will show when the order of integration can be swapped.

\begin{defn}
	Let $\gamma \in (1/3,1/2]$ and $\XX \in \cC^{\gamma}$. We say that the process $Y: [0,T]\times [0,T] \to \R^{d \times d}$ is jointly controlled by $\XX$ and write $Y \in \cD^{2\gamma}_{2,X}([0,T]^2,\R^{d \times d})$ if for every fixed $s\in [0,T]$ we have that $Y_{\cdot,s}$ and $Y_{s,\cdot}$ are both controlled rough paths with respect to $\XX$. In this case we write: 
	 \begin{equs}
	 	Y_{v,s}-Y_{u,s} &= Y^{1}_{u,s} X_{v,u} + R^{1}_{v,u}(s),\\
	 	Y_{u,t}-Y_{u,s} &= Y^{2}_{u,s} X_{t,s} + R^{2}_{t,s}(u).
	 \end{equs}
Moreover we require for every fixed $s\in [0,T]$ that both $Y^{1}_{s,\cdot}$ and $Y^{2}_{\cdot,s}$ are controlled rough paths 
such that $Y^{1,2} = Y^{2,1}$, where we write: 
 \begin{equs}
 	Y^1_{u,t}-Y^1_{u,s} &= Y^{1,2}_{u,s} X_{t,s} + R^{1,2}_{t,s}(u),\\
 	Y^2_{v,s}-Y^2_{u,s} &= Y^{2,1}_{u,s} X_{v,u} + R^{2,1}_{v,u}(s).
 \end{equs}
Note that this also ensures that for every fixed $0 \leq u \leq v \leq T$ both $R^{1}_{v,u}(\cdot)$ and $R^{2}_{v,u}(\cdot)$ are controlled rough paths. This can be shown by verifying this nice formula:
\begin{equ} 
R^{1}_{v,u}(t)-R^{1}_{v,u}(s)-R^{2,1}_{v,u}(s) X_{t,s} = R^{2}_{t,s}(v)-R^{2}_{t,s}(u)-R^{1,2}_{t,s}(u) X_{v,u}. 
\end{equ}
Denote any side of this equality by $R(t,s,v,u)$. We call $Y^1,Y^2$ first order Gubinelli  derivatives and $Y^{1,2}$ second order Gubinelli derivatives. Strictly speaking, the whole tuple $(Y,Y^1,Y^2,Y^{1,2})$ is an element of $\cD^{2\gamma}_{2,X}$
since  $Y^1,Y^2,Y^{1,2}$ need not be unique.

We introduce the following seminorms: for any function $Z : [0,T]^2 \to \R^{d \times d}$
\begin{equ}
	\|Z\|_{\infty,\gamma} =\sup_{{0\leq s \leq T}}\sup_{{0\leq u<v\leq T}}\frac{| Z_{s,v}-Z_{s,u}|}{|v-u|^{\gamma}};\quad \|Z\|_{\gamma,\infty} = \sup_{{0\leq u \leq T}}\sup_{{0\leq s<t\leq T}}\frac{| Z_{t,u} - Z_{s,u}|}{|t-s|^{\gamma}}.
\end{equ}
For function $ Q: [0,T]^3 \to \R^{d \times d}$ that is written for times $v,u,s$ like $Q_{v,u}(s)$ we write:
$$\|Q\|_{2\gamma,\infty} = \sup_{{0\leq s \leq T}}\sup_{{0\leq u<v\leq T}}\frac{| Q_{v,u}(s)|}{\;|v-u|^{2\gamma}}.$$
Now we consider $Y \in \cD^{2\gamma}_{2,X}([0,T]^2,\R^{d \times d})$ with $(Y,Y^1,Y^2,Y^{1,2},R^1,R^2,R^{1,2},R^{2,1})$ as above.
\begin{equs}
	\vv Y^{1,2} \vv_\gamma &= \|Y^{1,2}\|_{\infty,\gamma}+\|Y^{2,1}\|_{\gamma,\infty},\\ [.4em]
	\|Y'\|_{\gamma} = \|&Y^1\|_{\gamma,\infty} + \|Y^2\|_{\infty,\gamma}+\vv Y^{1,2} \vv_\gamma 
\end{equs}
For convenience we write $\delta R^{1,2}_{t,s}(v,u) := R^{1,2}_{t,s}(v)-R^{2,1}_{t,s}(u)$ and $\delta R^{2,1}_{v,u}(t,s) := R^{2,1}_{v,u}(t)-R^{2,1}_{v,u}(s)$ and denote:
\begin{equs}
	\vv R \vv_{2\gamma,2\gamma} &= \sup_{{0\leq s<t\leq T}}\sup_{{0\leq u<v\leq T}}\frac{|R(t,s,v,u)|}{|v-u|^{2\gamma}\,|t-s|^{2\gamma}},\\
	\vv R^{1,2}\vv_{\gamma,2\gamma} &= \sup_{{0\leq s<t\leq T}}\sup_{{0\leq u<v\leq T}}\frac{|\delta R^{1,2}_{t,s}(v,u)|}{|v-u|^{\gamma}\,|t-s|^{2\gamma}},\\
	\vv R^{2,1}\vv_{2\gamma,\gamma} &= \sup_{{0\leq s<t\leq T}}\sup_{{0\leq u<v\leq T}}\frac{|\delta R^{2,1}_{v,u}(t,s)|}{|v-u|^{2\gamma}\,|t-s|^{\gamma}},\\
\end{equs}
The total norm of the remainder is then defined as
\begin{equs}
	\|R^Y\|_{2\gamma} = \|R^1\|_{2\gamma,\infty} + \|R^2\|_{2\gamma,\infty} &+\|R^{1,2}\|_{2\gamma,\infty}+ \vv R^{1,2}\vv_{\gamma,2\gamma}+\\
	&+\|R^{2,1}\|_{2\gamma,\infty}+\vv R^{2,1}\vv_{2\gamma,\gamma}+   \vv R \vv_{2\gamma,2\gamma}.
\end{equs}
Finally putting it all together we can define a seminorms on
$\cD^{2\gamma}_{2,X}([0,T]^2,\R^{d \times d})$ by: 
\begin{equ}
	\|(Y,Y')\|_{X,2\gamma} =\|Y'\|_{\gamma} + \|R^Y\|_{2\gamma},
\end{equ}
and a norm
\begin{equ}
\|Y;R\|_{\cD^{2\gamma}_{2,X}} = |Y_{0,0}| + |Y^1_{0,0}| +|Y^2_{0,0}| +|Y^{1,2}_{0,0}| + \|(Y,Y')\|_{X,2\gamma}.
\end{equ}
\end{defn}
The following lemma about properties of the above seminorms is easy to verify.
\begin{lem} \label{fubini_lem1}
Let $Z : [0,T]^2 \to \R^{d \times d}$ be any function so that $\vv Z\vv_\gamma := \|Z\|_{\infty,\gamma}+\|Z\|_{\gamma,\infty}$ is finite. Then: 
$$\|Z\|_{\infty,\infty} = \sup_{{0\leq s \leq T}}\sup_{{0\leq u \leq T}} |Z_{s,u}| \lesssim_T |Z_{0,0}| + \vv Z\vv_\gamma.$$
Moreover for $Y \in \cD^{2\gamma}_{2,X}([0,T]^2,\R^{d \times d})$ the following properties hold: 
\begin{equs}
	\|Y^1\|_{\infty,\gamma} &\lesssim_T (|Y^{1,2}_{0,0}|+\vv Y^{1,2} \vv_\gamma)|X|_\gamma+\|R^{1,2}\|_{2\gamma,\infty},\\
	\|Y^2\|_{\gamma,\infty} &\lesssim_T (|Y^{1,2}_{0,0}|+\vv Y^{1,2} \vv_\gamma)|X|_\gamma+\|R^{2,1}\|_{2\gamma,\infty},\\
	\|Y\|_{\infty,\gamma} &\lesssim_T (|Y^{2}_{0,0}|+\vv Y^{2} \vv_\gamma)|X|_\gamma+\|R^{2}\|_{2\gamma,\infty},\\
	\|Y\|_{\gamma,\infty} &\lesssim_T (|Y^{1}_{0,0}|+\vv Y^{1} \vv_\gamma)|X|_\gamma+\|R^{1}\|_{2\gamma,\infty}.
\end{equs}
\end{lem} 

\begin{ex}\label{ex:control}
Let $K$ and $H$ be Banach spaces and $V \in \cD^{2\gamma}_{X}([0,T]^2,K)$, $Z \in \cD^{2\gamma}_{X}([0,T]^2,H)$. Let $B: K\times H \to \R^{d \times d}$ be a bilinear map. Then defining $Y_{u,s} := B(V_u,Z_s)$ we have $Y\in\cD^{2\gamma}_{2,X}([0,T]^2,\R^{d \times d})$.
Moreover with the abuse of notation when $B$ act on $K\otimes \R^d$ or $H\otimes \R^d$ component wise: 
\begin{equs}[3]
Y^{1}_{u,s} &= B(V'_u,Z_s) &\quad Y^{2}_{u,s} &= B(V_u,Z'_s)&\quad Y^{1,2}_{u,s} &= B(V'_u,Z'_s)\\
R^{1}_{v,u}(s) &=B(R^V_{v,u},Z_s) &\quad R^{2}_{t,s}(u) &= B(V_u,R^Z_{t,s}) &\quad R^{1,2}_{v,u}(s) &= B(R^V_{v,u},Z'_s)\\
R^{2,1}_{t,s}(u) &= B(V'_{u},R^Z_{t,s}) &\quad R(t,s,v,u) &= B(R^V_{v,u},R^Z_{t,s})\;.
\end{equs}
Later we will see a more sophisticated example where one cannot split $Y$ so easily into the inner product of two controlled rough paths.
\end{ex}
First we will show that integrating the jointly controlled rough path along one of the directions is creating a usual one time variable controlled rough path.
\begin{lem}\label{integrands}
	Let $\XX \in \cC^\gamma([0,T],\R^d)$ and $Y \in \cD^{2\gamma}_{2,X}([0,T],\R^{d\times d})$. Then writing
	\begin{equs}
		(V_r,V'_r) &:= \Big(\int_0^rY_{r,s}dX_s,Y_{r,r}+\int_0^rY^1_{r,s}dX_s\Big),\\
		(Z_r,Z'_r) &:= \Big(\int_r^tY_{s,r}dX_s, - Y_{r,r}+\int_r^tY^2_{s,r}dX_s\Big),
	\end{equs}
defines controlled rough paths $V \in \cD^{2\gamma}_X([0,T],\R^d)$ and $Z \in \cD^{2\gamma}_X([0,t],\R^d)$. 
\end{lem}
\begin{proof}
A straightforward computation shows that:
	\begin{equs}
		V_v - V_u &= \int_0^vY_{v,s}dX_s - \int_0^uY_{u,s}dX_s = \Big(Y_{u,u} + \int_0^uY^1_{u,s}dX_s\Big) X_{v,u} \\
		&\quad+ \int_u^v Y^1_{u,s}dX_s \, X_{v,u} + \int_0^v R^1_{v,u}(s)dX_s + \int_u^v (Y_{u,s} - Y_{u,u}) dX_s\\
		& = V'_u  X_{v,u} + R^V_{v,u}.	
	\end{equs}
Using assumptions on the uniform norms on $Y^1$ and $R^1$ as well as bounds on the rough integrals one can indeed show that $|R^V_{v,u}| \lesssim_{Y,T} |v-u|^{2\gamma}$ and $|\delta V'_{v,u}| \lesssim_{Y,T} |v-u|^\gamma$. For $Z$ we have: 
\begin{equs}
	Z_v - Z_u &= \int_v^tY_{s,v}dX_s - \int_u^tY_{s,u}dX_s = \Big(-Y_{u,u} + \int_u^tY^2_{s,u}dX_s\Big) X_{v,u} \\
	&\quad-  \int_u^v Y^2_{s,u}dX_s \, X_{v,u} +\int_v^t R^2_{v,u}(s)dX_s - \int_u^v (Y_{s,u} - Y_{u,u}) dX_s\\
	&= Z'_u  X_{v,u} + R^Z_{v,u}.	
\end{equs}
One can easily show that $|R^Z_{v,u}| \lesssim_{Y,t} |v-u|^{2\gamma}$ and $|\delta Z'_{v,u}| \lesssim_{Y,t} |v-u|^\gamma$. 
\end{proof}
Before we proceed we need a good notion of a smooth approximation of the jointly controlled rough path with respect to the smooth approximation of the rough path. We refer the reader to the paper \cite{geometric} where the authors showed that 
\begin{equ}
	\cC^{\gamma}([0,T],\R^d) \cong  \cC^{\gamma}_g([0,T],\R^d)\oplus C^{2\gamma}([0,T], \R^{d \times d}),
\end{equ}
where $C^{2\gamma} \subset \CC^{2\gamma}$ is a closure of smooth functions with respect to the ${2\gamma}$-H\"{o}lder norm. This means that for every $\XX = (X,\X) \in \cC^{\gamma}([0,T],\R^d)$ there exists a unique $\XX^g = (X,\X^g) \in \cC^{\gamma}_g([0,T],\R^d)$ and a unique $f \in C^{2\gamma}([0,T], \R^{d \times d})$ with $f_0 = 0$ such that
\begin{equ}[e:decomp]
	\X_{t,s} = \X^g_{t,s} + \delta f_{t,s}\;.
\end{equ}
Having this decomposition one can show that for $(Y,Y') \in \cD^{2\gamma}_X([0,T], \R^d)$ the following integral formula holds:
\begin{equ}[e:decomp_int]
	\int_a^b Y_s\, d\XX_s = \int_a^b Y_s\,d\XX^g + \int_a^b Y'_s\cdot \,df_s\;,
\end{equ}
and the second integral on the right hand side makes perfect sense as a Young's integral, where we understand the product $Y'_s\cdot df_s$ of two $d\times d$ matrices as a Frobenius inner product: for $A, B \in \R^{d\times d}$ set $A\cdot B = \tr(A^TB)$.
 \begin{defn}\label{smoothapprox}
 	Let $\XX \in \cC^{\gamma}([0,T],\R^d)$, we say that $Y \in \cD^{2\gamma}_{2,X}([0,T]^2,\R^{d \times d})$ admits a smooth approximation if there exist sequences $X^n \in C^\infty([0,T],\R^d), f^n \in C^\infty([0,T],\R^{d\times d})$ and $Y^n \in \cD^{2\gamma}_{2,X^n}([0,T]^2,\R^{d \times d})$ such that for $\XX^n = \XX^c(X^n) + (0,\delta f^n)$ the following approximations hold:
\begin{equs}
\lim\limits_{n\to\infty}\varrho_\gamma(\XX,\XX^n) &= 0\;,\\
\lim_{n\to \infty}	|Y_{0,0}-Y^n_{0,0}| + |Y^1_{0,0}-Y^{1,n}_{0,0}| +|Y^2_{0,0}-Y^{2,n}_{0,0}| +|Y^{1,2}_{0,0}-Y^{1,2,n}_{0,0}| &= 0\;,\\
\lim_{n\to \infty} \|Y'-Y'^{\,n}\|_\gamma + \|R^Y-R^{Y,n}\|_{2\gamma} &= 0\;.
\end{equs}
 \end{defn}
An example of $Y$ that admits a smooth approximation can be $Y$ from Example~\ref{ex:control} since 
both $V$ and $Z$ are the usual controlled rough paths by $\XX$ and can be each smoothly 
approximated.
\begin{rem}\label{approx1d}
	We sketch an argument on why a classical (one time variable) controlled rough path $V$ can always be 
	smoothly 
	approximated. To see this, one first shows that the equality
	$\delta V_{t,s} = V'_{s} X_{t,s} + R^V_{t,s}$ is equivalent to showing that
	$V = V'\prec X + U$
	for some $U \in \CC^{2\gamma}$ and $\prec$ denotes some sort of the paraproduct (which is a 
	continuous bilinear map, see \cite{gubinelli2015paracontrolled} for the definition and properties of the Bony paraproduct on the Fourier space). This paraproduct $\prec$ can for example be defined by
	\begin{equ}
		f \prec g = \sum_{x,n} f(x) \scal{g, \psi_{n,x}} \psi_{n,x}\;,
	\end{equ}
	where $\psi_{n,x}$ is an $L^2$-normalised wavelet basis.
	Then one will simply take a smooth approximation of $X$ say $X^n$ and define $V'^{\,n}$ and $U^n$ by 
	mollifying $V'$ and $U$. Then defining $V^n := V'^{\,n}\prec X^n + U^n$ from continuity of the paraproduct 
	$\prec$ one can then show that $(V^n, V'^{\,n})$ converges to $(V,V')$ in the rough path metric.
	Though this smooth approximation of the classical rough path is not canonical we will see later on that for 
	our purposes we will only need an existence of some smooth approximation and we do not care which 
	particular one is it.
\end{rem}

Since for $Y \in \cD^{2\gamma}_{2,X}([0,T]^2,\R^{d \times d})$ we can integrate with respect to $\XX$ 
in both of its time directions, a natural question is if the order of integration matters. Even though we believe 
that Fubini like theorem holds for every 
$Y \in \cD^{2\gamma}_{2,X}([0,T]^2,\R^{d \times d})$ we only show the proof for $Y$ that can be smoothly 
approximated. For our purposes this is going 
to be sufficient since we will apply these results later to a case similar to Example~\ref{ex:control}. First we 
show how to swap the order of integration where the limits of the second integral are time variables that are 
also integrated. It turns out that unlike for usual integration, there is in general a correction term appearing for
non-geometric 
rough paths. For the purpose of clarity we use in this section notation for the rough integral using the 
bold letter $d\XX$, reserving $dX$ for Young integrals. 
\begin{thm} \label{fubini2}
	Let $\gamma \in (1/3,1/2]$, $\XX \in \cC^{\gamma}([0,T],\R^d)$ and let $Y \in \cD^{2\gamma}_{2,X}([0,T]^2,\R^{d \times d})$ admit a smooth approximation. Then 
	\begin{equ}
	\int_0^t\int_s^tY_{r,s}d\XX_rd\XX_s + \int_0^tY_{s,s}\cdot\,df_s = \int_0^t\int_0^rY_{r,s}d\XX_sd\XX_r - \int_0^tY_{s,s}\cdot\,df_s\;, \label{e:swap1}
	\end{equ}
where $f$ is the $C^{2\gamma}$ function appearing in \eqref{e:decomp}. In particular, note that one has a usual change of order of integration in the case where $\XX$ is geometric:
\begin{equ}
	\int_0^t\int_s^tY_{r,s}d\XX_rd\XX_s = \int_0^t\int_0^rY_{r,s}d\XX_sd\XX_r\;. \label{e:swap2}
\end{equ}
\end{thm}
\begin{proof}
First note that all the double integrals are well-defined due to Lemma~\ref{integrands}. We will first prove the theorem for the case of geometric rough path and then will use decomposition~(\ref{e:decomp}) in order to show the general case. Let's call the left hand side of~(\ref{e:swap2}) $L_t$ and right hand side $R_t$. The main idea is of approximation. Basically we want to show that:
\begin{equ}[e:approx]
	L_t = \lim\limits_{n \to \infty}\int_0^t\int_s^tY^{n}_{r,s}d\XX^n_rd\XX^n_s = \lim\limits_{n \to \infty}\int_0^t\int_0^rY^{n}_{r,s}d\XX^n_sd\XX^n_r = R_t.
\end{equ}
Here $\XX^n$ and $Y^n \in \cD^{2\gamma}_{2,X^n}([0,T]^2,\R^{d\times d})$ are as in the definition of smooth approximation. 
Since we can take $\XX^n$ to be geometric rough paths themselves then the middle equality in~(\ref{e:approx}) is perfectly valid. This because in the smooth and geometric case rough integrals agree with the classical integrals for which the middle equality in~(\ref{e:approx}) is certainly true. 
It remains to establish the two other equalities. We will only show the third equality of~(\ref{e:approx}).
Once again since rough path $\XX^n$ is smooth and geometric then all the rough integrals with respect to $\XX^n$ are in fact the classical integral. We denote $I^{n}_t = \int_0^t\int_0^rY^{n}_{r,s}d\XX^n_sd\XX^n_r = \int_0^t\int_0^rY^{n}_{r,s}dX^n_sdX^n_r$, $V_r = \int_0^rY_{r,s}d\XX_s$ and $V^n_r = \int_0^rY^n_{r,s}d\XX^n_s$. From Lemma~\ref{integrands} both $V$ and $V^n$ are controlled rough paths (respectively w.r.t. $\XX$ and $\XX^n$) and for Gubinelli derivative of $V$ we write $\dot{V}_r = Y_{r,r}+\int_0^rY^1_{r,s}d\XX_s$, and similarly $\dot{V}^n$ (In fact here you can see why does $Y^1$ and $Y^2$ also have to be rough paths). First write $M = \max\{\varrho_\gamma(X),\,\|(Y,Y')\|_{X,2\gamma}\}$. Because of convergence we can guarantee that eventually $\varrho_\gamma(X^n)+\|(Y^n,Y'^{\,n})\|_{X^n,2\gamma} \leq 3M$ and so using stability of integration similar to Lemma~\ref{stab_int} we get:
	\begin{equs}
		|R_t-I^n_t|&\leq t^\gamma|R-I|_\gamma \\
		&\lesssim \varrho_\gamma(\XX,\XX^n)  + |V_0-V^n_0| +  |\dot{V}_0-\dot{V}^n_0|+d_{X,X^n,2\gamma}(V,V^n) \\
		&= \varrho_\gamma(\XX,\XX^n) +|V_0-V^n_0| +  |\dot{V}_0-\dot{V}^n_0|+ |\dot{V}-\dot{V}^n|_\gamma+|R^{V}-R^{V^n}|_{2\gamma}.
	\end{equs}
	Now first three terms clearly converges to $0$ by approximation assumption. For the term $|\dot{V}-\dot{V}^n|_\gamma$ we can again use stability of integration and nice approximation assumption on $Y$ to deduce that $|\dot{V}-\dot{V}^n|_\gamma \to 0$ as $n \to \infty$. We will show in more details on how to treat $|R^{V}-R^{V^n}|_{2\gamma}$ term.
	Therefore we have:
	\begin{equs}
		R^{V}_{v,u}-R^{V^n}_{v,u} = \int_u^vY_{v,s}d\XX_s-Y_{u,u} X_{v,u} &- \int_u^vY^n_{v,s}dX^n_s+Y^n_{u,u} X^n_{v,u} \\
		&+ \int_0^uR^1_{v,u}(s)d\XX_s - \int_0^uR^{1,n}_{v,u}(s)dX^n_s.
	\end{equs}
	The last two terms are bounded by
	\begin{equs}
		|\int_0^uR&^1_{v,u}(s)d\XX_s - \int_0^uR^{1,n}_{v,u}(s)dX^n_s| \leq\\
		&\le |R^1_{v,u}(0)-R^{1,n}_{v,u}(0)|\,| X_{u,0}|+|R^{1,n}_{v,u}(0)|\,| X_{u,0}- X^n_{u,0}|\\
		&\quad+|R^{2,1}_{v,u}(0)-R^{2,1,n}_{v,u}(0)|\,|\X_{u,0}|+|R^{2,1,n}_{v,u}(0)|\,|\X_{u,0}-\X^n_{u,0}| + u^{3\gamma}| \Xi^n_{v,u}|_{3\gamma}.
	\end{equs}
	Here $| \Xi^n_{v,u}|_{3\gamma} = \sup_{0\leq s<t\leq T}\frac{|\Xi^n_{v,u}(t,s)|}{|t-s|^{3\gamma}}$ with
	\begin{equs}
		\Xi^n_{v,u}(t,s) &= R(t,s,v,u)( X_{t,s}- X^n_{t,s}) + (R(t,s,v,u)-R(t,s,v,u)^n) X^n_{t,s}\\
		&+\delta R^{2,1}_{v,u}(t,s)(\X_{t,s}-\X^n_{t,s}) + (\delta R^{2,1}_{v,u}(t,s) - \delta R^{2,1,n}_{v,u}(t,s))\X^n_{t,s}.
	\end{equs}
	We see that from approximation assumptions we indeed have 
	$$\sup_{0\leq u<v\leq T}\frac{|\int_0^uR^1_{v,u}(s)d\XX_s - \int_0^uR^{1,n}_{v,u}(s)dX^n_s|}{|v-u|^{2\gamma}}| \to 0 \quad \text{as}\quad n\to\infty.$$
	For the remaining terms in equality for $R^{V}_{v,u}-R^{V^n}_{v,u}$ we add and subtract $Y_{v,u} X_{v,u} - Y_{v,u}^n X^n_{v,u}$.  Then we use similar bounds for integrals to deduce: 
	\begin{equs}
		&\frac{|\int_u^vY_{v,s}dX_s-Y_{u,u} X_{v,u} -  \int_u^vY^n_{v,s}dX^n_s+Y^n_{u,u} X^n_{v,u}|}{|v-u|^{2\gamma}} \lesssim_{M,T}\\ \lesssim_{M,T}\;\varrho_\gamma(&\XX,\XX^n) +|Y^2_{0,0}-Y^{2,n}_{0,0}|+ \| Y^2-Y^{2,n}\| _{\infty,\gamma}+\| R^2-R^{2,n}\| _{\infty,2\gamma}.
	\end{equs}
	All terms converge to zero by assumption and so we are done proving the third equality in~(\ref{e:approx}). Proving the first equality of~(\ref{e:approx}) may seem to be more difficult since the integral inside is also $t$ dependent. But in fact it is easy to 
	check that it plays almost no role but requires a bit more computations similar to above. Thus we finish showing formula~(\ref{e:swap1}) for the geometric rough path.
	
	For the non geometric rough path the middle equality of~(\ref{e:approx}) is no longer true because of the presence of the correction term $f$. In fact using the integral formula~(\ref{e:decomp_int}) and Lemma~\ref{integrands}, denoting by $f^n$ a smooth approximation of $f$ we can show that:  
	\begin{equs}
		V^n_r &= \int_0^r Y^n_{r,s}d\XX^n_s = \int_0^r Y^n_{r,s}dX^n_s + \int_0^r Y^{2,n}_{r,s} \cdot df^n_s;\\
		\dot{V}^n_r &= Y^n_{r,r} + \int_0^r Y^{1,n}_{r,s}d\XX^n_s =  Y^n_{r,r} + \int_0^r Y^{1,n}_{r,s}dX^n_s +\int_0^r Y^{1,2,n}_{r,s}\cdot df^n_s; \\
		\int_0^t V^n_r\XX^n_r &= \int_0^t V^n_rdX^n_r + \int_0^t\dot{V}^n_r \cdot df^n_r.
	\end{equs}
Using that $\int_0^t\int_0^rY^{n}_{r,s}d\XX^n_sd\XX^n_r = \int_0^t V^n_r\XX^n_r$ and  putting all the above formulas together we get that: 
	\begin{equs}
	\int_0^t\int_0^rY^{n}_{r,s}d\XX^n_s&d\XX^n_r= \int_0^t\int_0^rY^{n}_{r,s}dX^n_sdX^n_r
		+ \int_0^t\int_0^r Y_{r,s}^{1,n} \,dX_s^n\cdot df_r^n\\
		&+ \int_0^t\int_0^r Y_{r,s}^{2,n} \,df_s^n\,dX_r^n
		+ \int_0^t\int_0^r Y_{r,s}^{1,2,n} \cdot df_s^n\cdot df_r^n
		+ \int_0^t Y_{r,r}^n\cdot df_r^n\;.
	\end{equs}
Similarly:
	\begin{equs}
	\int_0^t\int_s^tY^{n}_{r,s}d\XX^n_r&d\XX^n_s= \int_0^t\int_s^tY^{n}_{r,s}dX^n_rdX^n_s
	+ \int_0^t\int_s^t Y_{r,s}^{2,n} \,dX_r^n\cdot df_s^n\\
	&+ \int_0^t\int_s^t Y_{r,s}^{1,n} \cdot df_r^n\,dX_s^n
	+\int_0^t\int_s^t Y_{r,s}^{2,1,n} \cdot df_r^n\cdot df_s^n
	- \int_0^t Y_{s,s}^n\cdot df_s^n\;.
	\end{equs}
Therefore using $Y^{1,2}  = Y^{2,1}$ we get:
\begin{equs}
	\int_0^t\int_s^tY^{n}_{r,s}d\XX^n_rd\XX^n_s + \int_0^t Y_{s,s}^n \cdot df_s^n =  \int_0^t\int_0^rY^{n}_{r,s}d\XX^n_sd\XX^n_r - \int_0^t Y_{s,s}^n \cdot df_s^n.
\end{equs}
Letting $n$ go to infinity we indeed get~\eqref{e:swap1}.
\end{proof}
 \begin{thm}[Rough Fubini Theorem] \label{fubini1}
 	Let $\gamma \in (1/3,1/2]$, $\XX \in \cC^{\gamma}([0,T],\R^d)$ and $Y \in \cD^{2\gamma}_{2,X}([0,T]^2,\R^{d \times d})$ 
 	admitting a smooth approximation. Then for $[s,t] \subseteq [0,T]$ and $[u,v] \subseteq [0,T]$, one has the identity
 	\begin{equ}
 		\int_s^t\int_u^vY_{r,m}d\XX_rd\XX_m = \int_u^v\int_s^tY_{r,m}d\XX_md\XX_r.
 	\end{equ}
 \end{thm}
One can prove this theorem using the same argument of approximation and it is even easier to show than the Theorem~\ref{fubini2}. Notice that when say in both integrals limits of integration are from $0$ to $t$ then Theorem~\ref{fubini1} is a corollary of Theorem~\ref{fubini2}. This is because the controlled rough path $(\int_0^t Y_{r,m}d\XX_m,\int_0^t Y^{1}_{r,m}d\XX_m)$ is a sum of two controlled rough paths
$(\int_0^r Y_{r,m}d\XX_m,Y_{r,r} + \int_0^r Y^{1}_{r,m}d\XX_m)$ and $(\int_r^t Y_{r,m}d\XX_m, -Y_{r,r} + \int_r^t Y^{1}_{r,m}d\XX_m)$ for every $r \in [0,t]$. Thus splitting the rough path $(\int_0^t Y_{r,m}d\XX_r,\int_0^t Y^{2}_{r,m}d\XX_r)$ similarly we get:
\begin{equs}
	\int_0^t\int_0^tY&_{r,m}d\XX_md\XX_r = \\
	&=\int_0^t\int_0^rY_{r,m}d\XX_md\XX_r - \int_0^tY_{s,s}\cdot df_s + \int_0^t\int_r^tY_{r,m}d\XX_md\XX_r + \int_0^tY_{s,s}\cdot df_s \\
	&=\int_0^t\int_m^tY_{r,m}d\XX_rd\XX_m + \int_0^tY_{s,s}\cdot df_s + \int_0^t\int_0^mY_{r,m}d\XX_rd\XX_m - \int_0^tY_{s,s}\cdot df_s \\ &=\int_0^t\int_0^tY_{r,m}d\XX_rd\XX_m.
\end{equs}
A natural question to ask is whether a double integral of $Y \in \cD^{2\gamma}_{2,X}([0,T]^2,\R^{d \times d})$ is itself an element of $\cD^{2\gamma}_{2,X}([0,T]^2,\R)$. We give the answer below but we do not study this question in much details since only Theorem~\ref{fubini2} is needed for our purposes.
\begin{lem} \label{fubini_lem2}
Let $\gamma \in (1/3,1/2]$, $\XX \in \cC^{\gamma}([0,T],\R^d)$ and $Y \in \cD^{2\gamma}_{2,X}([0,T]^2,\R^{d \times d})$ define:
$$Z_{t,v} = \int_0^t\int_0^vY_{s,r}d\XX_rd\XX_s.$$
Then $Z \in  \cD^{2\gamma}_{2,X}([0,T]^2,\R)$ and $Z^1_{t,v} = \int_0^vY_{t,r}d\XX_r$; $Z^2_{t,v} = \int_0^vY_{s,v}d\XX_r$; $Z^{1,2}_{t,v} = Y_{t,v}$. Moreover the map $Y \mapsto Z$ is continuous as a map $\cD^{2\gamma}_{2,X}([0,T]^2,\R^{d \times d}) \to \cD^{2\gamma}_{2,X}([0,T]^2,\R)$.
\end{lem}
 \begin{rem}
 	One can easily generalise the above results to functions in more time variables, giving rise to a generalised spaces $\cD^{2\gamma}_{k,X}([0,T]^k,\R^{d^k})$ for $k \in \N$. Another approach of defining the double integral is the approach of so called ``Rough sheets'' introduced in \cite{sheets}. However, 
to best of our knowledge, no statement like Theorem~\ref{fubini2} is known for Rough sheets.
 \end{rem}
It will also be useful to be able to rough integrals with usual Riemann integrals. Let $Y : [0,T]^2 \to \R^d$ be a process such that for each fixed $s \in [0,T]$, $Y_{\cdot,s} \in \cD^{2\gamma}_{X}([0,T],\R^d)$ is a controlled rough path and $Y_{s,\cdot} \in C([0,T],\R^d)$ is a continuous function. For such $Y$ we say that it admits a smooth approximation if there exist sequences $X^n \in C^\infty([0,T],\R^d), f^n \in C^\infty([0,T],\R^{d\times d})$ like in Definition~\ref{smoothapprox} and $Y^n : [0,T]^2 \to \R^d$ such that for each fixed $s \in [0,T]$, $Y^n_{\cdot,s} \in \cD^{2\gamma}_{X^n}([0,T],\R^d)$ and $Y^n_{s,\cdot} \in C([0,T],\R^d)$ such that 
$$
\lim_{n \to \infty} \sup_{{0\leq s \leq T}}(|Y_{0,s}-Y^n_{0,s}|+|Y'_{0,s}-Y^{',n}_{0,s}|+d_{X,X^n,2\gamma}(Y_{\cdot,s},Y^n_{\cdot,s})) = 0\;,
$$ 
where $X^n$ is a smooth function such that $\varrho_\gamma(\XX,\XX^n) \to 0$ as $n\to \infty$. The following 
theorem can then be proved using the same method as Theorem~\ref{fubini2}.
\begin{thm} \label{fubini3}
	Let $\gamma \in (1/3,1/2]$, $\XX \in \cC^{\gamma}([0,T],\R^d)$. Let $Y : [0,T]^2 \to \R^d$ be such that for each fixed $s \in [0,T]$, $Y_{\cdot,s} \in \cD^{2\gamma}_X([0,T],\R^d)$ and $Y_{s,\cdot} \in C([0,T],\R^d)$. Assume that $Y$ admits a smooth approximation as described above. Then we can perform the following exchange of the integrals: 
\begin{equ}
\int_0^t\int_s^tY_{r,s}d\XX_rds = \int_0^t\int_0^rY_{r,s}dsd\XX_r.
\end{equ}
\end{thm}
Note that no correction term with $f$ from decomposition~(\ref{e:decomp}) arise in this case. This is 
because in the left hand side the rough integrand has a Gubinelli derivative $Y^{1}_{r,s}$ (meaning in the first 
time variable) and the rough integrand in the right hand side has a Gubinelli derivative $\int_0^rY^{1}_{r,s}ds$ 
which will not create any correction terms when proving the analogue of the middle equality 
of~(\ref{e:approx}).

\section{Weak formulation and It\^{o}'s formula for RPDEs} \label{weak&ito}

In this section we are going to give an equivalent notion of solution for~(\ref{e:rpde2})\dash the weak solution. Recall that in 
Theorem~\ref{RPDE1} where we obtain solutions to the fixed point problem~\eqref{e:MildSoln}, we used the spaces $\cD^{2\varepsilon,2\varepsilon,\varepsilon}_{S,X}([0,T],\HH_{-2\gamma})$ for $0<\varepsilon<\gamma$ in order to obtain suitable bounds on the 
term $\|F(u_t)-S_tF(\xi)\|_{\HH_{2\varepsilon-2\gamma}}$. On the other hand, the right hand side of~\eqref{e:MildSoln} makes sense as an 
element of $\HH$ for any controlled rough path $(u,u') \in \cD^{2\gamma,2\gamma,0}_{S,X}([0,T],\HH_{-2\gamma})$. This motivates us to give the following notions of solution:
\begin{defn} \label{weak}
	Let $\gamma \in (1/3,1/2]$ and $\XX = (X, \X) \in \cC^\gamma(\R_+,\R^d)$. Let $\xi \in \HH$, $F \in C^2_{-2\gamma,0}(\HH,\HH^d)$, and $N \in \Poly^{0,n}_{0,-\delta}(\HH)$ for some $n\geq 1$ and $1-\delta > \gamma$. We say that $(u,F(u)) \in \cD^{2\gamma,2\gamma,0}_{S,X}([0,T],\HH_{-2\gamma})$ is a mild solution of the equation
	\begin{equ}
		du_t = Lu_tdt +N(u_t)dt + F(u_t)dX_t\quad \text{and}\quad  u_0 = \xi \in \HH\,,
	\end{equ}
	if, for each $0\leq t \leq T$, the following identity holds:
	\begin{equ}
		u_t = S_t\xi + \int_0^t S_{t-r}N(u_r)dr+ \int_0^t S_{t-r}F(u_r)dX_r\,.
	\end{equ}
	We say that $(u,F(u)) \in \cD^{2\gamma,2\gamma,0}_{S,X}([0,T],\HH_{-2\gamma})$ is a weak solution if for every $h \in \HH_1$ and $0\leq t \leq T$ the following integral formula holds:
	\begin{equ} 
		\langle u_t, h \rangle = \langle u_0, h \rangle+ \int_0^t \langle u_s, Lh \rangle ds + \int_0^t \langle N(u_s), h \rangle ds + \int_0^t \langle F(u_s), h \rangle dX_s\,. \label{e:weak1}
	\end{equ}
\end{defn}
Note that since $N(u_s) \in \HH_{-\delta}$ and $\delta < 2/3$, $\langle N(u_s), h \rangle$ is well-defined. Moreover, since $2\gamma<1 $ and therefore $\HH_1 \subseteq \HH_{2\gamma}$, Proposition~\ref{space_equiv2} guarantees that $\langle F(u_s), h \rangle$ is a controlled rough path in the classical sense and the integral $\int_0^t \langle F(u_s), h \rangle dX_s$ is well-defined. We are going to prove that these two notions of solution are in fact equivalent. To prepare this proof, we have the following preliminary result:
\begin{lem} \label{weak_lem}
	Let $\XX \in \cC^{\gamma}([0,T],\R^d)$ for $\gamma \in (1/3,1/2]$. Then for every $h\in \HH$ and $(Y,Y') \in \cD^{2\gamma,2\gamma,0}_{S,X}([0,T],\HH^d_{-2\gamma})$ we have for each $0\leq t\leq T$:
\begin{equ} 
	\int_{0}^{t}\langle \int_{0}^{s}S_{s-r}Y_rd\XX_r, h \rangle ds = \int_{0}^{t} \int_{r}^{t}\langle S_{s-r}Y_r, h \rangle ds \,d\XX_r. 
\end{equ}
\end{lem}
\begin{proof}
	First note that by Remark~\ref{approx1d} and since by Proposition~\ref{space_equiv1}
	$\cD^{2\gamma,2\gamma,0}_{S,X} = \cD^{2\gamma,2\gamma,0}_{X}$, we can find a smooth 
	approximation of $(Y,Y')$, meaning that there exists a sequence of $\XX^n = (X^n,\X^n) \in \cC^\gamma$ with 
	$X^n$ smooth such that $\varrho_\gamma(\XX,\XX^n) \to 0$ as $n \to \infty$ and a 
	sequence $(Y^n,Y'^{\,n})\in \cD^{2\gamma,2\gamma,0}_{S,X^n}([0,T],\HH^d_{-2\gamma})$ such that
	$$d_{2\gamma,2\gamma,0}(Y,Y^{n}) \to 0\quad \text{as}\quad n\,\to \infty.$$ 
	By Proposition~\ref{space_equiv2}, $W^n_{t,r} = \int_{r}^{t}\langle S_{s-r}Y^n_r, h \rangle ds$ is a 
	controlled rough path with respect to $X^n$ and $W_{t,r} = \int_{r}^{t}\langle S_{s-r}Y_r, h \rangle 
	ds$ is a controlled rough path with respect to $X$. Therefore the following integrals can be defined in the 
	rough path sense: $Z^n_t = \int_0^t W^n_{t,r}d\XX^n_r$, $Z_t = \int_0^t W_{t,r}d\XX_r$. (The fact that $W$ also 
	depends on $t$ does not cause any difficulties in defining the integral). Similar arguments as in the proof of Theorem~\ref{fubini2} allow us to deduce that
	$$\|Z-Z^n\|_{\infty,[0,T]} \lesssim (d_{2\gamma,2\gamma,0}(Y,Y^{n}) + \varrho_\gamma(\XX,\XX^n))\|h\| 
	\to 0 \quad \text{as}\quad n\,\to\,\infty\,,$$
	thus $Z^n \to Z$ uniformly in time. Moreover we know from the stability of integration Lemma~\ref{stab_int} that for $V^n_s = 
	\int_{0}^{s}S_{s-r}Y^n_rd\XX^n_r$ and $V_s = \int_{0}^{s}S_{s-r}Y_rd\XX_r$ we have:
	$$\|V^n-V\|_{\infty,\HH} \lesssim d_{2\gamma,2\gamma,0}(Y,Y^{n}) + \varrho_\gamma(\XX,\XX^n)\to 
	0\quad \text{as}\quad n\,\to\,\infty\,.$$
	It is easy to see that smoothness of $X^n$ implies $\int_0^t\langle V^n_s,h\rangle ds = Z^n_t$ and thus:
	\begin{equs}
		\Big|&\int_{0}^{t}\langle \int_{0}^{s}S_{s-r}Y_rd\XX_r, h \rangle ds -\int_{0}^{t} \int_{r}^{t}\langle S_{s-r}Y_r, h \rangle ds \,d\XX_r\Big| =\\
		&=\Big|\int_0^t\langle V_s,h\rangle ds-\int_0^t\langle V^n_s,h\rangle ds + Z^n_t - Z_t\Big|
		\leq \int_0^t|\langle V_s-V^n_s,h\rangle| ds + |Z^n_t-Z_t| \leq\\ 
		&\leq T \|V^n-V\|_{\infty,\HH} \|h\| + \|Z-Z^n\|_{\infty,[0,T]} \to 0 \quad \text{as}\quad n\,\to\,\infty\,,
	\end{equs}
	therefore showing the result.
\end{proof}

With this at hand:
\begin{thm} \label{weak1}
	In the sense of Definition~\ref{weak}, mild and weak solutions are equivalent.
\end{thm}
\begin{proof}
	Without loss of generality we can assume in both cases that $\xi = 0$ by replacing $(u,F(u))$ by $(u+S_\cdot \xi,F(u))$ (using $\ddh S_\cdot\xi = 0$).
	
\noindent\textbf{Mild $\Rightarrow$ Weak}. Assume also for simplicity that $N = 0$ since dealing with the drift term term is easier than with 
the diffusion term $F$. Now let $(u,F(u)) \in \cD^{2\gamma,2\gamma,0}_{S,X}([0,T],\HH_{-2\gamma})$ satisfy for $0\leq t \leq T$ 
	$$u_t = \int_0^t S_{t-s}F(u_s)dX_s\,.$$
	Let $h \in \HH_1$ be arbitrary. Then taking the inner product with $Lh$ and integrating from $0$ to $t$ gives
	\begin{equs}
		\int_0^t \langle u_s, Lh \rangle ds &= \int_0^t \langle \int_0^s S_{s-r}F(u_r)dX_r, Lh \rangle ds = \int_{0}^{t} \int_{r}^{t}\langle S_{s-r}F(u_r), Lh \rangle ds\,dX_r \\
		&=\int_{0}^{t} \langle F(u_r),\int_{r}^{t} S_{s-r}(Lh)ds \rangle dX_r \\
		&= \int_{0}^{t} \langle F(u_r),S_{t-r}h \rangle dX_r - \int_{0}^{t} \langle F(u_r),h \rangle dX_r\;,
	\end{equs}
	where we used Lemma~\ref{weak_lem} in the second equality together with $Lh \in \HH$. To conclude, it suffices to note that by Proposition~\ref{space_equiv2}  $\langle F(u_r), S_{t-r}h \rangle = \langle S_{t-r}F(u_r), h \rangle$ is itself a rough path and therefore
	$$\int_{0}^{t} \langle F(u_r),S_{t-r}h \rangle dX_r = \langle \int_{0}^{t} S_{t-r}F(u_r)dX_r, h \rangle  = \langle u_t, h \rangle\;.$$
\noindent\textbf{Weak $\Rightarrow$ Mild}. The proof is almost identical to the standard proof for SPDEs and can be found either in \cite{spde} or \cite{daprato}.
\end{proof}

The next lemma is a slight generalisation of Theorem~\ref{weak1}, but it has exactly the same proof, so we omit it.
\begin{lem} \label{weak2}
 Let $\sigma \geq 0$, $\alpha \in \R$. Let $(u,u') \in \cD^{2\gamma,2\gamma,0}_{S,X}([0,T],\HH_{\alpha-2\gamma})$ and $(v,v') = (v,F(u)) \in \cD^{2\gamma,2\gamma,0}_{S,X}([0,T],\HH_{\alpha-2\gamma-\sigma})$ satisfy the following weak equation for every $h \in \HH_{1-\alpha+\sigma}$:
 \begin{equ} 
 	\langle v_t, h \rangle = \langle v_0, h \rangle + \int_0^t \langle v_s, Lh \rangle ds + \int_0^t \langle N(u_s), h \rangle ds + \int_0^t \langle F(u_s), h \rangle dX_s. \label{e:weak2}
 \end{equ}
Here $F \in C^2_{\alpha-2\gamma,0}(\HH,\HH^d)$, and $N \in \Poly^{0,n}_{\alpha,-\delta}(\HH)$ is of polynomial type for some $n\geq 1$ and $1-\delta > \gamma$. Then the following mild formula holds in $\HH_{\alpha-\sigma}$: 
  \begin{equ} 
	v_t = S_tv_0 + \int_0^t S_{t-r}N(u_r)dr+ \int_0^t S_{t-r}F(u_r)dX_r. \label{e:weak3}
 \end{equ}
Moreover the converse is also true: (\ref{e:weak3}) implies~(\ref{e:weak2}) for every $h \in \HH_{1-\alpha+\sigma}$.
\end{lem}
Note that $h \in \HH_{1-\alpha+\sigma}$ guarantees that $\langle v_s, Lh \rangle$ is well-defined because of $v_s \in \HH_{\alpha-\sigma}$ and $Lh \in \HH_{-\alpha+\sigma}$.

A particularly important case is the choice $v_t = A(u_t)$ for some regular function $A \in C^2_{\alpha-2\gamma,-\sigma}(\HH)$. By 
Lemma~\ref{compos4}, for every $(u,u') \in \cD^{2\gamma,2\gamma,0}_{S,X}([0,T],\HH_{\alpha-2\gamma})$, we then have $(A(u)),DA(u)u') \in \cD^{2\gamma,2\gamma,0}_{S,X}([0,T],\HH_{\alpha-2\gamma-\sigma})$. The question we want to ask is whether $A(u_t)$ satisfies some mild formula like~\eqref{e:weak3}?
Before we answer this question we recall the definition of the bracket of a rough path:
\begin{defn}
	Let $V$ be a Banach space and $\XX \in \cC^\gamma([0,T],V)$, then its bracket is given by $[\XX]_t =  X_{t,0} \otimes  X_{t,0} - 2\Sym(\X_{t,0})$.
\end{defn}
From Chen's relation~\eqref{chen} it follows that
$$\delta [\XX]_{t,s} =  X_{t,s} \otimes  X_{t,s} - 2\Sym(\X_{t,s})\;,$$
and therefore $[\XX] \in C^{2\gamma}([0,T],V)$.
In particular $[\XX] = 0$ if and only if $\XX$ is a geometric rough path. Moreover the above implies that $[\XX]_t = -2\Sym(f_t)$ where $f \in C^{2\gamma}$ is as in the decomposition~\eqref{e:decomp}.
For example $[\mathcal{B}^\Ito]_t = t$ and $[\mathcal{B}^\Strat]_t=0$ almost surely. Now for a moment assume that $[\XX] = 0$ so that 
there is no ``It\^{o} correction'' and we can just apply the chain rule. Assume that $A$ is Fr\'echet differentiable and $u$ formally satisfies an equation $du_t = Lu_tdt +N(u_t)dt + F(u_t)dX_t$. Then heuristically we have:
\begin{equs}
	d(A(u_t)) &= DA(u_t)du_t =  DA(u_t)Lu_tdt +DA(u_t)N(u_t)dt + DA(u_t)F(u_t)dX_t  \\
	&=L A(u_t)dt +\big( DA(u_t)N(u_t) + [L,A](u_t)\big)dt + DA(u_t)F(u_t)dX_t .
\end{equs}
Here for any two differentiable functions $G \in C^1_{\alpha_1,\beta_1}(\HH)$, $H\in C^1_{\alpha_2,\beta_2}(\HH)$ we define the Lie bracket 
$$[G,H](u):= DH(u)G(u) - DG(u)H(u) \in C_{\alpha_1\vee \alpha_2,\beta_1+\beta_2}(H)\;.$$ 
Since $L$ is linear, we have $DL(u) = L$ for each $u \in \HH_\alpha$, therefore $[L,A](u_t) = DA(u_t)Lu_t-L A(u_t)$. Writing $\tilde{N}(u) = DA(u)N(u) + [L,A](u)$ and $\tilde{F}(u) = DA(u)F(u)$ then on a formal level $A(u_t)$ solves
$$d(A(u_t)) = L (A(u_t))dt + \tilde{N}(u_t)dt + \tilde{F}(u_t)dX_t.$$
This suggests that $A(u_t)$ satisfies the identity
$$A(u_t) = S_tA(u_0) + \int_0^t S_{t-r}\tilde{N}(u_r)dr+ \int_0^t S_{t-r}\tilde{F}(u_r)dX_r.$$
Before showing this result for the mild formulation rigorously we state a weak version of it: 
\begin{thm}[Weak It\^{o} formula] \label{ito1}
	 Let $\gamma \in (1/3,1/2]$, $\XX \in \cC^\gamma$, $\sigma \geq 0$, and $\alpha \in \R$. Let $(u,u') \in \cD^{2\gamma,2\gamma,0}_{S,X}([0,T],\HH_{\alpha-2\gamma})$ and $(v,F(u)) \in \cD^{2\gamma,2\gamma,0}_{S,X}([0,T],\HH_{\alpha-2\gamma-\sigma})$ be such that~(\ref{e:weak2}) 
	 holds for every $h \in \HH_{1-\alpha+\sigma}$ with $F$ and $N$ as stated there. Then, for every $\nu\geq 0$ and $A \in C^2_{\alpha-2\gamma-\sigma,-\nu}(\HH)$, one
	 has $(A(v),DA(v)F(u)) \in \cD^{2\gamma,2\gamma,0}_{S,X}([0,T],\HH_{\alpha-2\gamma-\sigma-\nu})$ and the following identity holds for every $h \in \HH_{{1-\alpha+\sigma+\nu}}$:
	 \begin{equs}
	 	\langle &A(v_t), h \rangle = \langle A(v_0), h \rangle + \int_0^t \langle DA(v_s)Lv_s+DA(v_s)N(u_s), h\rangle ds +\\ 
	 	+&\int_0^t \langle DA(v_s)F(u_s), h \rangle dX_s 
	 	+ \frac{1}{2}\int_0^t \langle D^2A(v_s)(F(u_s),F(u_s)),h \rangle d[\XX]_s. \label{e:ito1}
	 \end{equs}
\end{thm}
\begin{proof}
	Without loss of generality we only consider the case $\sigma = 0$. By Proposition~\ref{space_equiv1}, $(v,v') = (v,F(u)) \in \cD^{2\gamma,2\gamma,0}_{X}([0,T],\HH_{\alpha-2\gamma})\subset \cD^{2\gamma}_{X}([0,T],\HH_{\alpha-2\gamma})$ and by the mild representation of Lemma~\ref{weak2} it satisfies: 
	$$v_t-v_s = F(u_s) X_{t,s} +DF(u_s)F(u_s)\X_{t,s} + R_{t,s}.$$
	Setting $v'_t = F(u_t)$ and $v''_t = DF(u_t)F(u_t)$, we note that $$(G,G') = (DA(v)v', D^2A(v)(v',v')+DA(v)v'')$$ is itself a controlled rough path in $\cD^{2\gamma}_X([0,T],\HH_{\alpha-2\gamma-\nu})$.
	As in \cite[Remark 4.11]{friz},  one can define $\mathbb{V}_{t,s} = \int_s^t(v_r-v_s)\otimes dv_s \in \CC_2^{2\gamma}(\HH_{\alpha-2\gamma}\otimes \HH_{\alpha-2\gamma})$, yielding a rough path $\textbf{v} = (v,\mathbb{V}) \in \cC^\gamma(\HH_{\alpha-2\gamma})$. From now on we are going to omit writing $\otimes$ and always understand say $\phi \psi$ for two elements of some Hilbert space $\phi$ and $\psi$ as their tensor product $\phi\otimes \psi$. By It\^{o}'s formula for 
	rough paths \cite[Prop.~5.6]{friz}  we have that:
	\begin{equs} 
		\delta A(v)_{t,0} = A(v_t) - A(v_0) &= \lim_{|\CP|\to 0} \sum_{[m,r] \in \CP} (DA(v_m)\delta v_{r,m} + D^2A(v_m)\mathbb{V}_{r,m}) \\ &\qquad+\frac{1}{2} \lim_{|\CP|\to 0} \sum_{[m,r] \in \CP} D^2A(v_m)\delta[\textbf{v}]_{r,m}. 
		\label{e:ito2}
	\end{equs}
	With convergence in $\HH_{\alpha-2\gamma-\nu}$.
	Now one can show that $\delta[\textbf{v}]_{r,m} = v'_m v'_m \delta[\XX]_{r,m} + o(|r-m|)$ and using that $\mathbb{V}_{r,m} = v'_mv'_m \X_{r,m} + o(|r-m|)$ we can take an inner product of~(\ref{e:ito2}) with $h \in \HH_{1-\alpha+\nu}$ on both sides we get:
	\begin{equs}
		\langle \delta A(v)_{t,0},h\rangle &= \lim_{|\CP|\to 0} \sum_{[m,r] \in \CP} \langle DA(v_m)(\delta v_{r,m})+\big(D^2A(v_m)v'_mv'_m\big)\X_{r,m},h\rangle \\ 
		&\qquad+\frac{1}{2} \lim_{|\CP|\to 0} \sum_{[m,r] \in \CP} \langle D^2A(v_m)v'_m v'_m,h\rangle \delta[\XX]_{r,m} \\
		&= \lim_{|\CP|\to 0} \sum_{[m,r] \in \CP} \rI_{r,m} + \lim_{|\CP|\to 0} \sum_{[m,r] \in \CP}\rII_{r,m}.\label{e:expression}
	\end{equs}
The second term in this expression converges to $\frac{1}{2}\int_0^t \langle D^2A(v_s)(v'_s,v'_s), h\rangle d[\XX]_s$,
interpreted  as a Young integral. 
	Since $v'_s = F(u_s)$, this gives the very last term in~(\ref{e:ito1}).
	
	To deal with the first term in \eqref{e:expression}, note that for a fixed value of $m$, one has 
	$\langle DA(v_m)\delta v_{r,m}, h\rangle = \langle \delta v_{r,m}, DA^*(v_m)h\rangle$ with $DA^*(v_m)h \in \HH_{1-\alpha}$,
	so that we can apply~(\ref{e:weak2}) for fixed $m$, yielding
	\begin{equs}
		\langle DA(v_m)&\delta v_{r,m}, h\rangle =\\
		&= \int_m^r\langle v_s, LDA^*(v_m)h\rangle ds + \int_m^r\langle N(u_s), DA^*(v_m)h\rangle ds \\ 
		&\quad+\int_m^r\langle F(u_s), DA^*(v_m)h\rangle dX_s \\
		&=\, \langle v_m, LDA^*(v_m)h\rangle (r-m)+\langle N(u_m), DA^*(v_m)h\rangle (r-m) \\
		&\quad+\langle F(u_m) X_{r,m}+DF(u_m)F(u_m)\X_{r,m}, DA^*(v_m)h\rangle + o(|r-m|) \\
		&=\langle DA(v_m)Lv_m+DA(v_m)N(u_m), h\rangle (r-m) \\ 
		&\quad+\langle DA(v_m)v'_m X_{r,m}+DA(v_m)v''_m\X_{r,m}, h\rangle + o(|r-m|).
	\end{equs}
	We conclude that one has
	\begin{equs}
		\rI_{r,m} 
		&= \langle DA(v_m)Lv_m+DA(v_m)N(u_m), h\rangle (r-m) \\
		&\quad+ \langle G_m X_{r,m} + G'_m \X_{r,m},h\rangle + o(|r-m|).
	\end{equs}
	Since $(G,G')$ is a controlled rough path, we then obtain
	$$\lim_{|\CP|\to 0} \sum_{[m,r] \in \CP} \rI_{r,m} = \int_0^t \langle DA(v_s)Lv_s+DA(v_s)N(u_s), h\rangle ds +\langle \int_0^t G_s dX_s, h\rangle.$$
	We conclude by recalling that $\int_0^t G_s dX_s = \int_0^t DA(v_s)v'_sdX_s = \int_0^t DA(v_s)F(u_s)dX_s$.
\end{proof}

Finally we state the main result of this section:
\begin{thm}[Mild It\^{o} formula] \label{ito2}
	Let $\gamma \in (1/3,1/2]$, $\XX \in \cC^\gamma$, $\sigma \geq 0$, and $\alpha \in \R$. Let $(u,u') \in \cD^{2\gamma,2\gamma,0}_{S,X}([0,T],\HH_{\alpha-2\gamma})$ and let $(v,F(v)) \in \cD^{2\gamma,2\gamma,0}_{S,X}([0,T],\HH_{\alpha-2\gamma-\sigma})$ be related to $u$ through  
	$$v_t = S_tv_0 + \int_0^t S_{t-r}N(u_r)dr+ \int_0^t S_{t-r}F(u_r)dX_r\;,$$
	with $F$ and $N$ as in Lemma~\ref{weak2}.
	Then, for any $\nu\geq 0$ and $A \in C^2_{\alpha-2\gamma,-\nu}(\HH)$, we have
	$(A(v),DA(v)F(u)) \in \cD^{2\gamma,2\gamma,0}_{S,X}([0,T],\HH_{\alpha-2\gamma-\sigma-\nu})$ and the following mild It\^{o} formula holds:
\begin{equs} 
	A(v_t) &= S_tA(v_0) + \int_0^t S_{t-r}(DA(v_r)N(u_r) +[L,A] (v_r)) dr \label{e:ito3}\\ 
	&\quad +\int_0^t S_{t-r}DA(v_r)F(u_r)dX_r+\frac{1}{2}\int_0^t S_{t-r} D^2A(v_r)(F(u_r),F(u_r))d[\XX]_r. 
\end{equs}
\end{thm}
\begin{proof}
	By Lemma~\ref{weak2}, equation~\eqref{e:weak2} holds for $(v,F(u))$, so that~(\ref{e:ito1}) holds 
	for every $h \in \HH_{{1-\alpha+\sigma+\nu}}$ by Theorem~\ref{ito1}. We now make use of the fact that 
	\begin{equs}
		\langle DA(v_s)Lv_s+DA(v_s)N(u_s), h\rangle  &= \langle LA(v_s), h \rangle + \langle DA(v_s)N(u_s)+[L,A](v_s), h \rangle \\
		&=\langle A(v_s), Lh \rangle  + \langle DA(v_s)N(u_s)+[L,A](v_s), h \rangle,
	\end{equs}
where $\langle [L,A](v_s), h \rangle$ makes sense since $[L,A](v_s) \in \HH_{\alpha-\sigma-\nu-1}$ and $h \in \HH_{{1-\alpha+\sigma+\nu}}$. Thus we get the following weak equation: 
	\begin{equs}
		\langle A(v_t), h \rangle &= \langle A(v_0), h \rangle + \int_0^t \langle A(v_s), Lh \rangle ds + \int_0^t \langle DA(v_s)N(u_s)+[L,A](v_s), h \rangle ds \\
		&+ \int_0^t \langle DA(v_s)F(u_s), h \rangle dX_s + \frac{1}{2}\int_0^t \langle D^2A(v_s)(F(u_s),F(u_s)),h \rangle d[\XX]_s,
	\end{equs}
which itself implies the mild formula~(\ref{e:ito3}) by Lemma~\ref{weak2} and the fact that the last integral is well-defined as a Young integral.
\end{proof}

 \section{Backwards RPDEs} \label{R_BPDE}

We will briefly describe the method of solving rough backwards PDEs of the form: 
\begin{equ}  
dv_t = -Lv_tdt -N(v_t)dt - F(v_t)dX_t\,,\quad  v_T = \xi \in \HH. \label{e:bpde1}
\end{equ} 
For short we call them backwards RPDEs. We will quickly describe the theory of backwards controlled rough paths according to the semigroup. 
In many instances, the proofs of the results are virtually identical to the corresponding ones for forward controlled rough paths,
so we do not give them. We introduce an increment operator $\ddc : \CC_1 \to  \CC_2$ 
$$\ddc f_{t,s} = S_{t-s}f_t - f_s\;,$$
for a semigroup $S$ acting on a Banach space $V$. (We will actually assume that $S$ consists of selfadjoint operators on some Hilbert space $\HH$.) With this, we define a H\"{o}lder like space
\begin{equ}
\check{\CC}^\gamma = \{f \in \CC_1 : |\ddc f|_{\gamma,V}<\infty\}\;,
\end{equ}
and we endow it with a seminorm $\vv f\vv_{\gamma,V} = |\ddc f|_{\gamma,V}$ and a norm $\| f\|_{\check{\CC}^\gamma} = \vv f\vv_{\gamma,V} + \|f_T\|_V$. (We could have replaced $\|f_T\|_V$ by  $\|f\|_{\infty,V}$, which yields an equivalent norm.)
\begin{defn}
	Let $\XX \in \cC^{\gamma}([0,T],\R^d)$ for some $\gamma\in (1/3,1/2]$ and let $m\in \N$. We say that $(Y,Y')\in \check{\CC}^\gamma([0,T],\HH_\alpha^m) \times \check{\CC}^\gamma([0,T],\HH_\alpha^{m\times d})$ is backwards controlled by $\XX$ according to the semigroup $(S_t)_{t\geq 0}$ if the remainder term defined through
\begin{equ} 
R^Y_{t,s} = \ddc Y_{t,s} -S_{t-s}Y'_t X_{t,s}\;,
\end{equ}
is an element of $\CC^{2\gamma}_2\HH_\alpha^m$.
\end{defn}
This defines a space of controlled rough paths (according to the semigroup) 
$$(Y,Y') \in \cD^{2\gamma}_{S,X,\back}([0,T],\HH_\alpha^m).$$
We endow this space with a semi-norm (omitting $d$ and $m$ for notational convenience)
\begin{equ}
\| Y,Y'\| _{2\gamma,X,\alpha} =  \vv Y'\vv _{\gamma,\alpha} + |R^Y|_{2\gamma, \alpha}.
\end{equ}
It is easy to see that the space $\cD^{2\gamma}_{S,X,\back}([0,T],\HH_\alpha^m)$ is a Banach space with norm: 
$$\| Y,Y'\|_{\cD^{2\gamma}_{S,X,\back}} = \| Y_T\| _{\HH_\alpha^m} + \| Y'_T\|_{\HH_\alpha^{m\times d}}+\| Y,Y'\| _{2\gamma,X,\alpha}.$$
Here the endpoint $Y_T$ plays the same role as the starting point for forward controlled rough paths. This is justified by the inequality $\|Y\|_\infty \lesssim_T \|Y_T\| + \vv Y\vv_\gamma$. This also corresponds to the fact that for backwards RPDEs we don't know the initial condition but rather the terminal condition.

Similarly as for forward controlled rough paths for $\beta \in \R$ and $\eta \in [0,1]$ define a space  $$\cD^{2\gamma,\beta,\eta}_{S,X,\back}([0,T],\HH_\alpha) := \cD^{2\gamma}_{S,X,\back}([0,T],\HH_\alpha)\cap  \big(\check{\CC}^\eta([0,T],\HH_{\alpha+\beta})\times L^\infty([0,T],\HH_{\alpha+\beta}^{d})\big).$$
We introduce a norm on this space to be:
\begin{equ}
\|(Y,Y')\|_{\cD^{2\gamma,\beta,\eta}_{S,X,\back}} = \|Y_T\|_{\HH_{\alpha+\beta}} + \|Y'\|_{\infty,\alpha+\beta} +\vv Y\vv_{\eta,\alpha+\beta} + \|(Y,Y')\|_{2\gamma,X,\alpha}.
\end{equ}
Here we also make an abuse of notation by writing $\check{\CC}^0 = L^\infty$ for $\eta = 0$. Similarly to Lemma~\ref{compos1}, composition with regular functions maps $\cD^{2\gamma,2\gamma,\eta}_{S,X,\back}([0,T],\HH_\alpha)$ to $\cD^{2\gamma,2\gamma,0}_{S,X,\back}([0,T],\HH_\alpha)$ for every $\eta \in[0,1]$. 

For $(Y,Y') \in \cD^{2\gamma}_{S,X,\back}([0,T],\HH_\alpha^d)$ an integral $\int_{t}^{T}S_{r-t}Y_rdX_r$ can be defined  and
$$\Big(\int_{\cdot}^{T}S_{r-\cdot}Y_rdX_r,\,Y\Big) \in \cD^{2\gamma}_{S,X,\back}([0,T],\HH_\alpha).$$
Moreover, results analogous to Theorem~\ref{integral}, Lemma~\ref{compos1}, and Theorem~\ref{stab_sol1} are true and their proofs are 
almost the same. The main difference is that the role of the initial condition $Y_0$ is now played by the terminal condition $Y_T$. 
We can now state a theorem regarding solutions to backwards equations of the type arising in~(\ref{e:bpde1}).
\begin{thm}[Nonlinear backwards RPDEs] \label{BRPDE}
	Let $\gamma \in (1/3,1/2]$ and $\XX = (X, \X) \in \cC^\gamma(\R_+,\R^d)$. Then, given $\xi \in \HH$, $F \in C^3_{-2\gamma,0}(\HH,\HH^d)$, and $N \in \Poly^{0,n}_{0,-\delta}(\HH)$ for some $n\geq1$ and  $1-\delta > \gamma$, there exists $\tau \geq 0$ and a unique element $(v,v') \in \cD^{2\gamma,2\gamma,\gamma}_{S,X,\back}((\tau,T],\HH_{-2\gamma})$ such that $v' = F(v)$ and
\begin{equ}
v_t = S_{T-t}\xi + \int_t^T S_{r-t}N(v_r)dr+ \int_t^T S_{r-t}F(v_r)dX_r\,,\quad v_T = \xi \in \HH.
\end{equ}
We call the pair $(v,F(v))$ the mild local solution to the backwards RPDE
$$dv_t = -Lv_tdt -N(v_t)dt - F(v_t)dX_t\quad  and\quad  v_T = \xi \in \HH.$$
\end{thm}
For a weak solution approach to both forward and backward rough PDEs we refer the reader to~\cite{FrizWeakRPDE}.

One can show that all the continuity results of Section~\ref{sec:continuityRPDE} are true for backwards RPDEs. The same is true for
a smoothing result analogous to Proposition~\ref{smoothing}, except that smoothing now takes place away from the terminal point $v_T = \xi$. 
One can show that solutions to backwards RPDEs coincide with solutions to backwards SPDEs in the case of Brownian motion. From now on, 
we will assume that we are in the setting of Theorem~\ref{BRPDE} with choices of $L$, $N$, $F$ and $X$ such that one can choose $\tau = 0$,
so that solutions exist (and are unique) on the whole of $[0,T]$. The following proposition establishes a connection between the forward and backward controlled rough paths.
\begin{prop} \label{forward_backward1}
Let $\gamma \in (1/3,1/2]$ and $\XX \in \cC^\gamma([0,T],\R^d)$. Let $\alpha,\beta$ be such that $\alpha+\beta+2\gamma \geq 0$ and let $(V,V') \in \cD^{2\varepsilon,2\gamma,0}_{S,X}([0,T],\HH_{\alpha})$ and $(Z,Z') \in \cD^{2\varepsilon,2\gamma,0}_{S,X,\back}([0,T],\HH_\beta)$. Then, setting 
$$Y_t := \langle V_t,Z_t \rangle \quad  \text{and} \quad  Y'_t := \langle V'_t,Z_t \rangle + \langle V_t,Z'_t \rangle\,,$$ 
we have $(Y,Y') \in \cD^{2\varepsilon}_{X}([0,T],\R)$ with bound
\begin{equ} 
|(Y,Y')|_{X,2\varepsilon} \lesssim_T (1+|X|_\gamma)\| V,V'\| _{\cD^{2\varepsilon,2\gamma,0}_{S,X}}\| Z,Z'\| _{\cD^{2\varepsilon,2\gamma,0}_{S,X,\back}}\;. \label{e:forward_backward1}
\end{equ}
\end{prop}
\begin{proof}
	The proof is a straightforward computation where we use the fact that $S_t$ is a selfadjoint operator on $\HH$ for any time $t \geq 0$.
	By the definition of controlled rough path
	\begin{equs}
		\langle V_t,Z_t \rangle - \langle V_s,Z_s \rangle &= \langle V_t,Z_t \rangle -\langle S_{t-s}V_s,Z_t \rangle +\langle V_s,S_{t-s}Z_t \rangle - \langle V_s,Z_s \rangle \\
		&= (\langle S_{t-s}V'_s,Z_t \rangle + \langle V_s,S_{t-s}Z_t' \rangle) X_{t,s} + \langle R^V_{t,s},Z_t \rangle + \langle V_s,R^Z_{t,s} \rangle.
	\end{equs}
	Now 
	$$\langle S_{t-s}V'_s,Z_t \rangle + \langle V_s,S_{t-s}Z_t' \rangle = Y'_s + \langle V'_s,S_{t-s}Z_t - Z_s\rangle + \langle V_s,S_{t-s}Z_t' - Z_s' \rangle.$$
	We can therefore write 
	$$R^Y_{t,s} = \langle R^V_{t,s},Z_t \rangle + \langle V_s,R^Z_{t,s} \rangle + \big(\langle V'_s,S_{t-s}Z_t - Z_s\rangle + \langle V_s,S_{t-s}Z'_t - Z_s' \rangle\big) X_{t,s}.$$
	The bound~(\ref{e:forward_backward1}) is then an easy consequence of decomposition above. The requirement on exponents $\alpha$ and $\beta$ is necessary since we want to bound terms like:
	$$|\langle R^V_{t,s},Z_t \rangle|\leq \|R^V_{t,s}\|_{\HH_\alpha}\|Z_t\|_{\HH_{\beta+2\gamma}}\quad \text{and}\quad |\langle V_s,R^Z_{t,s} \rangle|\leq \|V_s\|_{\HH_{\alpha+2\gamma}}\|R^Z_{t,s}\|_{\HH_{\beta}}.$$
	Here we need $\alpha+\beta+2\gamma\geq 0$ so that we can use the Cauchy-Schwarz inequality. 
\end{proof}

We just showed that the inner product of a forward controlled rough path with a backward controlled rough path is a controlled rough path in 
the usual sense. Assuming that these controlled rough paths solve respectively some RPDE and backwards RPDE in the mild sense, we can ask 
ourselves whether their inner product also satisfies an integral equation. It turns out that this is true and this inner product in fact solves 
an RDE:
\begin{prop} \label{forward_backward2}
	Let $1/3 < \varepsilon < \gamma < 1/2$ and $\XX \in \cC^\gamma([0,T],\R^d)$. Let $\delta \leq 1$ and simultaneously $\alpha+\beta+4\gamma-\delta \geq 0$ and $\alpha+\beta+2\gamma\geq 0$. Let $V \in \cD^{2\varepsilon,2\gamma,0}_{S,X}([0,T],\HH_{\alpha})$ and $Z \in \cD^{2\varepsilon,2\gamma,0}_{S,X,\back}([0,T],\HH_{\beta})$ be such that they satisfy these mild forward and backward equations on $[0,T]$:
	\begin{equs}
		V_t &= S_tV_0 + \int_0^t S_{t-r}N_rdr+ \int_0^t S_{t-r}F_rdX_r,\\
		Z_t &= S_{T-t}Z_T + \int_t^T S_{r-t}\tilde{N}_rdr+ \int_t^T S_{r-t}\tilde{F}_rdX_r,
	\end{equs}
	for some $F \in \cD^{2\varepsilon,2\gamma,0}_{S,X}([0,T],\HH^d_{\alpha})$, $\tilde{F} \in \cD^{2\varepsilon,2\gamma,0}_{S,X,\back}([0,T],\HH^d_{\beta})$ and functions $N\in L^\infty([0,T],\HH_{\alpha+2\gamma-\delta})$, $\tilde{N} \in L^\infty([0,T],\HH_{\beta+2\gamma-\delta})$.
	
	Then $Y_t := \langle V_t,Z_t \rangle \in \cD^{2\varepsilon}_{X}([0,T], \R)$ is a controlled rough path that satisfies the following integral formula: 
	\begin{equs} 
		Y_t = \langle V_0,Z_0 \rangle + \int_0^t(\langle N_s,Z_s \rangle - \langle V_s,\tilde{N}_s \rangle)ds + \int_0^t(\langle F_s,Z_s \rangle - \langle V_s,\tilde{F}_s \rangle) dX_s\\
		+ 2\int_0^t \langle F_s, \tilde{F}_s \rangle \cdot df_s\;,\quad \label{e:forward_backward2}
	\end{equs}
where the function $f \in C^{2\gamma}([0,T],\R^{d \times d})$ is the one appearing in the decomposition~(\ref{e:decomp}) 
of the rough path $\XX$. 
	In particular, if $\XX$ is geometric and $\tilde{N},\tilde{F}$ are the adjoints of $N$ and $F$, then $Y_t$ is constant in time: $$Y_t = \langle V_0,Z_0 \rangle  = \langle Y_T,Z_T \rangle \quad \text{for  every $ t \in [0,T]$.}$$
\end{prop}
\begin{proof}
	Note that the assumptions on $\alpha,\beta$ and $\delta$ are necessary for all the integrals in~(\ref{e:forward_backward2}) to make sense. (If $\delta = 1$ we only need the assumption $\alpha+\beta+4\gamma-\delta \geq 0$). We  will assume that $N= \tilde{N}= 0$ since the 
	drift term $dt$ is even simpler to treat than the $dX_t$ term. (Note though that Theorem~\ref{fubini3}  needs to be used at some point to swap the order of integration in integrals of the type $\int_0^t\int_s^t\langle N_r,S_{r-s}\tilde{F}_s \rangle \,dr\,dX_s$).
	Using the mild equations for $V_t$ and $Z_t$ we get: 
	\begin{equs}
		Y_t - Y_0 &= \langle V_t - S_tV_0,Z_t \rangle + \langle V_0,S_tZ_t - Z_0\rangle \\ 
		&=\Big\langle \int_0^t S_{t-r}F_rdX_r ,Z_t \Big\rangle - \Big\langle V_0,\int_0^t S_{r}\tilde{F}_rdX_r\Big\rangle. 
	\end{equs}
	We can move the inner products inside the integration by examining the proof of the Sewing Lemma, Theorem~\ref{sewing}, 
	and ideas similar to Proposition~\ref{space_equiv2}. Since $S$ is selfadjoint, we get:
	$$Y_t - Y_0 = \int_0^t(\langle F_r,Z_r \rangle - \langle V_r,\tilde{F}_r \rangle) dX_r +R_t.$$
	We will show that $R_t = 2\int_0^t \langle F_s, \tilde{F}_s \rangle \cdot df_s$ for every $t$, which then implies the result. One has the identity
	\begin{equs}
		R_t &= \int_0^t \langle F_r,S_{t-r}Z_t - Z_r \rangle dX_r+ \int_0^t\langle V_r - S_rV_0,\tilde{F}_r \rangle dX_r\\ 
		&= - \int_0^t \Big\langle F_r, \int_r^t S_{s-r}\tilde{F}_s dX_s \Big\rangle dX_r+ \int_0^t\Big\langle \int_0^r S_{r-s}F_sdX_s,\tilde{F}_r \Big\rangle dX_r \\
		&=- \int_0^t  \int_r^t \langle F_r, S_{s-r}\tilde{F}_s \rangle dX_s dX_r + \int_0^t \int_0^r \langle S_{r-s}F_s,\tilde{F}_r \rangle dX_s dX_r. 
	\end{equs}
	Setting $W_{r,s} = \langle  F_s,S_{r-s}\tilde{F}_r \rangle$, we would like to show that one can apply our version of Fubini's theorem, Theorem~\ref{fubini2}. We are almost in the situation of Theorem~\ref{fubini2}: the only difference is that $W_{r,s}$ is defined only for $r \geq s$ because of the presence of the semigroup $S_{r-s}$. But if one examines the proof of Theorem~\ref{fubini2}, one can see that 
	we can always require that $r\geq s$ in our computations. Here we have
	$$W^1_{r,s} = \langle F'_s, S_{r-s}\tilde{F}_r \rangle;\quad  W^2_{r,s} = \langle S_{r-s}F_s, \tilde{F}'_r \rangle;\quad  W^{1,2}_{r,s} = \langle S_{r-s}F'_s, \tilde{F}'_r \rangle.$$
	The remainders $R^1,R^2,R^{1,2},R^{2,1}$ are also easy to determine. 
	Since by Remark~\ref{approx1d} both $F$ and $\tilde{F}$ admit a smooth approximation then so does $W_{r,s}$ in the sense of Definition~\ref{smoothapprox}.
	Thus we can indeed swap the integrals for $W_{r,s}$ like in Theorem~\ref{fubini2}, deducing:
	\begin{equs}
		\int_0^t  \int_r^t \langle F_r, S_{s-r}\tilde{F}_s \rangle dX_s dX_r &+ \int_0^t\langle F_s, \tilde{F}_s \rangle \cdot df_s =\\
		&= \int_0^t \int_0^r \langle S_{r-s}F_s,\tilde{F}_r \rangle dX_s dX_r - \int_0^t\langle F_s, \tilde{F}_s \rangle \cdot df_s.
	\end{equs}
	We conclude that $R_t = 2\int_0^t \langle F_s, \tilde{F}_s \rangle \cdot df_s$ and hence we are done. 
\end{proof}

\subsection{Adjoint of the Jacobian}

From now on for simplicity we denote the Jacobian of the solution to the~(\ref{e:rpde2}) by $J_{t,s}$, omitting the reference to the noise $X$. In the later results we would like to use the adjoint of the Jacobian of the solution $J^*_{t,s}$. For instance, this appears in the expression for the Malliavin matrix $\langle\cM_t\varphi,\varphi\rangle = \int_0^t \langle F(u_s), J^*_{t,s}\varphi\rangle ds$. It would then
be useful to know that $J^*_{t,s}$ also solves an RPDE. Unfortunately, having a mild formulation for $J_{t,s}$ is not enough to deduce a 
mild formulation for $J^*_{t,s}$. Therefore we go the other way around: we `guess' an equation for $J^*_{t,s}$ and then show that the 
solution to this equation is indeed the adjoint of the Jacobian. In fact it will be more convenient to work with the backwards equation for the adjoint. This is because Proposition~\ref{forward_backward2} then gives us an explicit expression for $\scal{J_{r,s}\varphi , J^*_{t,r} \psi}$ for any $\psi,\varphi \in \HH$. A natural guess is to take a backwards analogue of~\eqref{e:jacobian} where we formally take adjoints of the linear maps $DN$ and $DF$, so that our ansatz for $J^*_{t,s}$ is:
\begin{equ}
	d J^*_{t,s}  = - L J^*_{t,s} ds - DN^*(u_s)J^*_{t,s} ds - DF^*(u_s) J^*_{t,s} dX_s,\qquad J^*_{t,t} = \id\;.
\end{equ}
The next proposition shows that this guess is indeed correct.
\begin{prop} \label{adj_jac1}
	Let $\XX \in \cC_g^\gamma([0,T], \mathbb{R}^{d \times d})$ be a geometric rough path. Let $(u,F(u)) \in \DD^{2\gamma}_X([0,T],\HH)$ be 
	the solution to~(\ref{e:rpde2}) with $F$ and $N$ as in Proposition~\ref{duhamel}. For every $t \in [0,T]$ and every $\varphi \in \HH$, 
	let $(K_{t,\cdot},K'_{t,\cdot}):=(K_{t,\cdot},DF^*(u_\cdot)K_{t,\cdot}) \in \cD^{2\gamma,2\gamma,\gamma}_{S,X,\back}([0,t],\HH_{-2\gamma})$ be the solution to the 
	backwards equation
	\begin{equ} 
		K_{t,s}\varphi = S_{t-s}\varphi  + \int^t_sS_{r-s}DN^*(u_r)K_{t,r}\varphi dr + \int^t_sS_{r-s}DF^*(u_r)K_{t,r}\varphi dX_r. \label{e:adj_jac1}
	\end{equ}
Then $K$ is the adjoint of the Jacobian: $K_{t,s} = J^*_{t,s}$ for all $0\leq s\leq t \leq T$.
\end{prop}
\begin{proof}
	We want to show that $\langle J_{t,s}\varphi,\psi\rangle = \langle \varphi,K_{t,s}\psi\rangle $ for all $\varphi,\psi \in \HH$. Set $Y_r = \langle J_{r,s}\varphi,K_{t,r}\psi\rangle$ and note that thanks to the smoothing property of the solutions, Proposition~\ref{smoothing}, 
	the regularity assumptions of Proposition~\ref{forward_backward2} are satisfied for $(Y_r,Y'_r) \in \cD^{2\gamma}_X([s+\varepsilon,t-\varepsilon],\R)$ for all $\varepsilon > 0$. Moreover, since $\XX$ is geometric, $Y_t$ satisfies the equation
	\begin{equs}
		Y_r &= Y_{s+\varepsilon} + \int_{s+\varepsilon}^r (\langle DN(u_v)J_{v,s},K_{t,v} \rangle - \langle J_{v,s},DN^*(u_v)K_{t,v} \rangle)dv \\
		&\quad+ \int_{s+\varepsilon}^r(\langle DF(u_v)J_{v,s},K_{t,v} \rangle - \langle J_{v,s},DF^*(u_v)K_{t,v} \rangle)dX_v = Y_{s+\varepsilon}.
	\end{equs}
	Since the terms inside the integrals cancel each other, we have $Y_{t-\varepsilon} = Y_{s+\varepsilon}$, i.e.
	$$\langle J_{t-\varepsilon,s}\varphi,K_{t,t-\varepsilon}\psi\rangle = \langle J_{s+\varepsilon,s}\varphi,K_{t,s+\varepsilon}\psi\rangle.$$
	But from the mild representation of $K_{t,r}$ and $J_{r,s}$, we see that both of these lie in the space $\CC([s,t],\HH)$ as functions of the $r$ variable. We can therefore take the limit of the above expression as $\varepsilon$ goes to zero to obtain:
	$$\langle J_{t,s}\varphi,K_{t,t}\psi\rangle = \langle J_{s,s}\varphi,K_{t,s}\psi\rangle.$$
	Recalling that $J_{s,s}\varphi = \varphi $ and $K_{t,t}\psi = \psi$, we get the desired result. 
\end{proof}
\begin{prop} \label{adj_jac2}
	Let $\XX \in \cC_g^\gamma([0,T],\R^d)$ and $\gamma \in (1/3,1/2]$. Let $(u,F(u)) \in \DD^{2\gamma}_X([0,T],\HH)$ be the solution to~(\ref{e:rpde2}) with $F$ and $N$ as in Proposition~\ref{duhamel}. Let $K_{t,s}$ be the adjoint of the Jacobian. Let $\nu \geq 0$ and a function $A \in C^2_{0,-\nu}(\HH)$. Fix $0\leq t \leq T$, $\varphi \in \HH$ and set $Z^\varphi_A(r) := \langle A(u_r),K_{t,r}\varphi \rangle$. Then $Z^\varphi_A \in \cD^{2\gamma}_X([s,t],\R)$ for every $0<s\leq t$ and solves the RDE
	$$dZ^\varphi_A(r) = Z^\varphi_{[L+N,A]}(r)dr + Z^\varphi_{[F,A]}(r)dX_r,$$
	which in its integral form reads for all $r \in [s,t]$:
	\begin{equ} 
	Z^\varphi_A(r) = Z^\varphi_A(s) + \int^r_sZ^\varphi_{[L+N,A]}(v)dv + \int^r_sZ^\varphi_{[F,A]}(v)dX_v. \label{e:adj_jac2}
	\end{equ}
\end{prop}
\begin{proof}
	First we use the mild It\^{o} formula Theorem~\ref{ito2} to determine the mild equation for $A(u_s)$. Note that since $\XX$ is 
	geometric, its bracket $[\XX]$ vanishes. We then use the mild formulation~(\ref{e:adj_jac1}) for $K_{t,s}$ and the fact that it is arbitrarily smooth on $[s,t)$ together with mild equation for $u_r$ and the fact that it is arbitrarily smooth on $(0,t]$ to ensure that we can apply Proposition~\ref{forward_backward2} to derive the equation for $Z^\varphi_A(r)= \langle A(u_r),K_{t,r} \rangle$. It is easy then to verify that this is indeed  Equation~\ref{e:adj_jac2}.
\end{proof}

\section{Spectral properties of the Malliavin matrix} \label{spectral}

For this section we consider a special case when the multiplicative noise term is given by 
$$F(u_t)dX_t = \sum_{i=1}^d F_i(u_t)dX^i_t,$$ where $F_i \in C^\infty_{-2\gamma,0}(\HH)$ are smooth functions. We also assume $N$ is smooth and belongs to $\Poly^{\infty,n}_{0,-\delta}$ and 
consider the collections of Lie brackets $\cA_k$ defined recursively by:
$$\cA_0  = \{F_i\, :\, 1\leq i \leq d\};\quad \cA_{k+1}  =\cA_k \cup \{[L+N,A],\,[F_i,A]\, :\, A \in \cA_k,\, 1\leq i \leq d\}.$$ 
Note that, at worst, elements of $\cA_k$ decrease the spatial regularity by $k$ i.e.\ send $\HH_k$ to $\HH$.
Our aim is to show that, under a version of H\"{o}rmander's condition appropriate for this context, 
one obtains a bound on the Malliavin matrix 
of the kind $\Prob (\inf_{\varphi}\langle \cM_T\varphi,\varphi\rangle \leq \varepsilon ) \lesssim_{T,p} \varepsilon^p$ for every $p\geq 1$ (we will specify 
precisely over which class of $\varphi$ we take the infimum later). The proof is in the same spirit as the proof of H\"{o}rmander's theorem for SDEs using Malliavin 
calculus techniques, see for instance \cite{Mal76,ProofShort}. It essentially goes by contradiction: assuming that $\langle \cM_T\varphi,\varphi\rangle$ is small,~(\ref{mal 7}) then implies that $\langle J_{T,s}F(u_s),\varphi\rangle$ is small. 
In the SDE case the solution to the equation with good enough 
vector fields generates a smoothly invertible flow, so it is possible to factor the Malliavin matrix as
$$\cM_T = J_{T,0}\hat{\cM}_TJ^*_{T,0}\,;\quad  \langle\hat{\cM}_T\varphi,\varphi\rangle = \int^T_0 \langle J^{-1}_{s,0}F(u_s),\varphi\rangle^2 ds.$$
Then the process $s \mapsto  \langle J^{-1}_{s,0}F(u_s),\varphi\rangle$ is a semimartingale and one can use Norris's lemma 
\cite{norris1986simplified} to deduce by induction over $k$ that $s \mapsto \langle J^{-1}_{s,0}A(u_s),\varphi\rangle$ is small for every vector field $A \in \cA_k$. 
H\"{o}rmander's condition then guarantees that the span of the $\cA_k$ at every point is dense in $\HH$, which 
contradicts the fact that all the $\langle J^{-1}_{s,0}A(u_s),\varphi\rangle$ are small by considering $s=0$.

The problem with this argument is that solutions to parabolic SPDEs do not produce a smoothly invertible flow, so that 
the Jacobian $J_{s,t}$ is not invertible. In \cite{hypoelipticity} where the authors deal with the case of additive noise
and polynomial nonlinearities, 
they use a version of Norris's lemma for Wiener polynomials instead of semimartingales. In our setting, we 
consider rough integration instead of It\^{o} integration, 
which allows us to use a version of Norris's lemma for rough paths.
Before stating it, we recall the notion of H\"{o}lder roughness from \cite{hairer2013regularity}:
\begin{defn}
	Let $\theta \in (0,1)$, a path $X : [0,T] \to \R^d$ is said to be $\theta$-H\"{o}lder rough if there exists a constant $L_\theta(X)$ such that for all $s \in [0,T]$, all $\varepsilon \in (0,T/2]$ and every $z \in \R^d$ with $|z| = 1$, there exists a $t \in [0,T]$ such that
	$$|t-s|\leq \varepsilon\quad and\quad |(z, X_{t,s})| > L_\theta(X)\varepsilon^p.$$ 
We denote the largest such $L_\theta(X)$ the modulus of $\theta$-H\"{o}lder roughness of X.
\end{defn}
Here $(x,y)$ denotes the scalar product on $\R^d$.
In \cite{hairer2013regularity} it was proved that if $X$ is a fractional Brownian motion with Hurst parameter $H\leq 1/2$, its sample paths are almost surely $\theta$-H\"{o}lder rough for every $\theta > H$. Moreover, there exist constants $M$ and $c$ independent of $X$ such that 
$$\Prob(L_\theta(X)\leq \varepsilon | \FF^X_0) \leq M\exp(-c\varepsilon^{-2}),$$
for every $\varepsilon \in (0,1)$, so that in particular $\E[L^{-p}_\theta(X)] < \infty$ for every $p \geq 0$.

With this at hand we are ready to state one more result from \cite{hairer2013regularity}, namely the
aforementioned version of Norris's lemma for rough paths.
\begin{thm} \label{norris}
	Let $\gamma \in (1/3,1/2]$ and $(X,\X) \in \cC^\gamma([0,T],\R^d)$ be $\theta$-H\"{o}lder rough for $\theta < 2\gamma$. Let $V \in\CC^\gamma([0,T],\R)$ and $(Y,Y')\in \cD^{2\gamma}_X([0,T],\R^{d})$ and set 
	$$Z_t = Z_0 + \int_0^tV_s ds + \int_0^t Y_s dX_s.$$
Then there are constants $q>0$ and $r>0$ such that, setting 
$$R := 1 + L_\theta(X)^{-1} + \varrho_\gamma(X) + |(Y,Y')|_{X,2\gamma} + |V|_{\CC^\gamma}\;,$$
we have the bound $\|Y\|_{\infty}+\|V\|_\infty \lesssim_T R^q \|Z\|^r_\infty$ on $[0,T]$.
\end{thm}

We will now work with a solution $(u,u')$ to~(\ref{e:rpde2}) starting from $u_0 \in \HH$ driven by the path $(X,\X) \in \cC_g^{\gamma}(\R_+,\R^d)$ and vector fields $N,F_i$ as described in the beginning of this section. $K_{\cdot,\cdot}$ denotes the adjoint of Jacobian as in  Section~\ref{R_BPDE}. Fix $T$ and a smaller time $0 < s < T$, let $1/3 < \eta < \gamma < 1/2$ be such that $\eta$ is close to $\gamma$. From now on fix the quantity:
\begin{equs}
	R_{s}(u_0) = 1 + L_\theta(X)^{-1} + \varrho_\gamma(X) &+ \|(u,u')\|_{\cD^{2\eta,2\gamma,0}_{S,X}([s,T],\HH_{\alpha})}\\
	&+\|(K_{T,\cdot}\,,K'_{T,\cdot})\|_{\cD^{2\eta,2\gamma,0}_{S,X,\back}([s,T],\LL^{-2\gamma})}\;, \label{dominater}
\end{equs}
for $\alpha \geq 0$ big enough to be determined later. If any of the quantities above explodes on the interval $[s,T]$ we simply write $R_{s}(u_0) = \infty$.

 As in Proposition~\ref{adj_jac2} define for $\varphi\in \HH$ and a vector field $A$, a function $Z_A^\varphi(r) = \langle A(u_r),K_{T,r}\varphi \rangle$. From Proposition~\ref{forward_backward1} and Lemma~\ref{compos4} it follows inductively that for every $k \in \N_0$ there exists a constant $C_k$ depending on $T$ and $L,N,F_i$ such that for all $A \in \cA_k$ we have:
 \begin{equ}[zcompos]
 	|(Z^\varphi_A,(Z^\varphi_A)')|_{X,2\gamma,[s,T]} \leq C_k R_{s}(u_0)^2.
 \end{equ}
The above holds true if in the definition of $R_s$ we take $\alpha$ big enough (depending on $k$) so that the assumptions on spatial regularities of Proposition~\ref{forward_backward1} would be satisfied.
\begin{rem} \label{rem 8}
 Note that we only impose high spatial regularity on the solution $u_r$ and not on $K_{T,s}$. This will give us the advantage of being able to use the fact that $K_{T,T} $ is the identity.
\end{rem}
The following two results are almost exact analogues of the finite-dimensional statements from \cite{hairer2013regularity}. Proposition~\ref{adj_jac2} allows us to carry out the same techniques.
\begin{lem} \label{first_iteration}
	Let $T>0$ then for all $0< s < T$, there exist $q,r > 0$ and $M$ independent of $\XX,\varphi,u_0$ 	 
	such that for all $A \in \cA_0$ the following bound holds:
\begin{equ}[iter1]
	\|Z^\varphi_A\|_{\infty,[s,T]} \leq M R^{\,q}_{s}(u_0) \langle \cM_T\varphi,\varphi\rangle^{r}.
\end{equ}
for all $\varphi \in \HH$ such that $\|\varphi\| =1$ and all initial conditions $u_0\in\HH$.
\end{lem}
\begin{proof}
	Note that we have 
	$$\langle \cM_T\varphi,\varphi\rangle = \sum_{i=1}^d\int_0^T \langle F_i(u_r),K_{T,r}\varphi\rangle^2dr =  \sum_{i=1}^d \|Z^\varphi_{F_i}\|^2_{L^2[0,T]}.$$ 
	Consider an interpolation inequality (see \cite{hairer2011ergodicity} Lemma A.3):
	$$\|f\|_\infty \leq 2\max\{T^{-1/2}\|f\|_{L^2},\|f\|^{2\gamma/(2\gamma+1)}_{L^2}|f|^{1/(2\gamma+1)}_{\CC^\gamma}\}.$$
	Since the final time is fixed, the $L^2$ norm is controlled by the $\gamma$-H\"{o}lder norm, so 
	$$\|Z^\varphi_{F_i}\|_{\infty,[s,T]}^{2\gamma+1} \lesssim_s \|Z^\varphi_{F_i}\|^{2\gamma}_{L^2[s,T]}|Z^\varphi_{F_i}|_{\CC^\gamma[s,T]}.$$
	The first term is clearly controlled by $\langle \cM_T\varphi,\varphi\rangle ^{2\gamma}$ since $[s,T] \subset [0,T]$ and the second term is controlled by $C_0R_{s}(u_0)^2$ by~\eqref{zcompos}.
\end{proof} 

Now we show that the same holds for any vector field in $\cA_k$.
\begin{lem} \label{k_iteration}
	Let $T>0$ and $(X,\X) \in \cC_g^\gamma([0,T],\R^d)$ be $\theta$-H\"{o}lder rough for $\theta < 2\gamma$. Then for all $0< s < T$ and every $k \in \N_0$ there exist $q_k,r_k > 0$ and $M_k$ independent of $\XX,\varphi,u_0$ such that for all $A \in \cA_k$ the following bound holds:
	\begin{equ}[iterk]
	\|Z^\varphi_A\|_{\infty,[s,T]} \leq M_k R^{\,q_k}_{s}(u_0) \langle \cM_T\varphi,\varphi\rangle^{r_k},
	\end{equ}
	for all $\varphi \in \HH$ such that $\|\varphi\| =1$ and all initial conditions $u_0 \in \HH$.
\end{lem}
\begin{proof}
	Define first for $A \in \cA_k$ a quantity
	$$R_A = 1 + L_\theta(X)^{-1} + \varrho_\gamma(X) + |(Z^\varphi_{[F,A]},(Z^\varphi_{[F,A]})')|_{X,2\gamma} + |Z^\varphi_{[L+N,A]}|_{\CC^\gamma},$$
	with all the norms taken on the interval $[s,T]$.
	Note that from~\eqref{zcompos} we have $R_A \leq C_k R^2_{s}(u_0)$ uniformly over all $\|\varphi\| = 1$. Assume now that the statement is true for $k$. Let $A \in \cA_k$, denote $F_0 = L+N$ then by Proposition~\ref{adj_jac2} we have the representation for all $s\leq r\leq T$:
	$$Z^\varphi_A(r) = Z^\varphi_A(s) + \int^r_sZ^\varphi_{[F_0,A]}(v)dv + \int^r_sZ^\varphi_{[F,A]}(v)dX_v.$$
	Therefore we can apply the ``rough Norris lemma'', Theorem~\ref{norris}, to deduce:
	\begin{equs}
		\sup_{i=0\ldots d}\|Z^\varphi_{[F_i,A]}\|_{\infty,[s,T]} &\leq M R^q_A \|Z^\varphi_A\|^r_{\infty,[s,T]} \\
		&\leq M C^q_k M^r_k R^{2q}_{s}(u_0) R^{\,r\times q_k}_{s}(u_0) \langle \cM_T\varphi,\varphi\rangle^{r\times r_k}.
	\end{equs}
	Since all $B \in \cA_{k+1}$ are of the form $[F_i,A]$ for $i = 0,\ldots, d$ and $A \in \cA_k$ we conclude by induction.
\end{proof}
\begin{rem} \label{no_ass}
	Note that both of the above lemmas are purely deterministic. Moreover we did not make any assumption on the solution or its Jacobian. In fact if the solution explodes before time $T$ or if it does not have a Jacobian on the interval $[s,T]$ one should simply read the inequalities~\eqref{iter1} and~\eqref{iterk} as trivial statements ``$\infty \leq \infty$''.
\end{rem}
We will now present the precise assumptions on the noise and solution which will enable us to prove our H\"ormander's theorem. We start with 
an assumption on the driving noise which guarantees that it gives rise to a well-behaved rough path, but is also sufficiently
irregular to kick the solution around in a very non-degenerate way.
\begin{ass}\label{ass3}
For some $\gamma \in ({1\over 3},{1\over 2})$, the random rough path $\XX=(X,\X) \in \cC_g^\gamma([0,T],\R^d)$ is the canonical lift of a $d$-dimensional, continuous Gaussian process 
$X$ with independent components defined on some underlying probability space $(\Omega,\FF,\Prob)$. We also assume that there exist $M<\infty$ and $p \in [1,1/2\gamma)$ such that for $i \in \{1,\ldots,d\}$ and $[s,t] \subseteq [0,T]$, the covariances of $X^i$ satisfy
	$$\|R_{X^i}\|_{p,[s,t]^2} \leq M |t-s|^{1/p}.$$
We also assume that $X$ is almost surely $\theta$-H\"{o}lder rough for some $\theta < 2\gamma$ and moreover that all inverse moments of 
its modulus of $\theta$-H\"{o}lder roughness are bounded, i.e.\ $\E[L^{-q}_\theta(X)] < \infty$ for all $q \ge 1$.
\end{ass}
With the driving noise $\XX$ at hand, we assume the global existence of the solution:
\begin{ass} \label{ass1}
	Let $\XX$ be as in Assumption~\ref{ass3} and let $\{F_i\}_1^d \subset C^\infty_{-2\gamma,0}(\HH)$, $N \in  \Poly^{\infty,n}_{0,-\delta}(\HH)$ for $\delta < 1-\gamma$. We assume that for every initial condition $u_0$,~(\ref{e:rpde2}) has a global solution $(u,F(u)) \in \DD^{2\gamma}_X(\R_+,\HH)$ for almost every realisation of $\XX$. We also assume that the Jacobian $J_{t,s}$ and its adjoint $K_{t,s}$ exist for all times and satisfy the corresponding mild equations~\eqref{mal 5} and~\eqref{e:adj_jac1} respectively.
\end{ass}

In the parabolic case, we cannot expect to have a bound on $\inf_{\|\varphi\| = 1} \langle \cM_T\varphi,\varphi\rangle $ 
since this would imply the invertibility of the Malliavin matrix, contradicting the fact that $\cM_T$ is a compact operator. 
Instead, we consider an orthogonal projection $\Pi :\HH \to \HH$ with finite-dimensional range and, 
for $a \in (0,1)$, we define $\mathscr{S}_a \subset \HH$ to be
$$\mathscr{S}_a = \{\varphi \in \HH : \|\varphi\| = 1,\, \|\Pi\varphi\| \geq a \}.$$
For $k \in \N_0$ define the positive symmetric quadratic form-valued function $Q_{k}$ such that for all $u \in \HH_{\infty}$
$$\langle\varphi,Q_{k}(u)\varphi\rangle = \sum_{A \in \cA_k}\langle\varphi, A(u) \rangle^2.$$
With this notation we assume that the following non-degeneracy condition holds on the Lie brackets in $\cA_k$:
\begin{ass} \label{ass2}
	Assume that $\{F_i\}_1^d \subset C^\infty_{-2\gamma,0}(\HH)$, $N \in  \Poly^{\infty,n}_{0,-\delta}(\HH)$, $\delta < 1-\gamma$. 
	Moreover assume that for some orthogonal projection $\Pi :\HH \to \HH$ and for every $1>a > 0$, there exists 
	$k \in \N_0$ as well as a continuous function $\Lambda_a : \HH \to (0,\infty)$ such that 
	$$\inf_{\varphi \in \mathscr{S}_a} \langle \varphi,Q_{k}(u)\varphi\rangle \geq \Lambda_a(u),$$
	for every $u \in \HH_{\infty}$.
\end{ass}

Finally, we assume that we have good enough control on the solution to be able to ``fight'' the loss of control
generated by regions where the function $\Lambda_a$ from Assumption~\ref{ass2} is small.

\begin{ass} \label{ass4}
	We assume that Assumptions~\ref{ass3},~\ref{ass1},~\ref{ass2} hold and that there exist two functions $\Phi_1,\Phi_2: \HH\times [0,\infty)  \to [0,\infty)$ such that the following two growth assumptions are true:
\begin{itemize}
	\item[\textbf{(1)}] For all $p\geq 1$ and all $a \in (0,1)$ there exist $K_1$ such that for the solution $u$ to~(\ref{e:rpde2}), the
	 inverse moment bound
	$$\E[\Lambda^{-p}_a(u_T)] \leq K_1 \Phi_1^p(u_0,T),$$
	holds for every initial condition $u_0\in\HH$.
	
	\item[\textbf{(2)}] For $\alpha := (k-2\gamma+\delta)\vee 0$ (where $\delta$ is from $\Poly^{\infty,n}_{0,-\delta}$ and $k$ is from assumption~\ref{ass2}), some $1/3 < \eta \leq \gamma < 1/2$ and for all $p\geq 1$ there exist $K_2$ such that
	$$\E\Big[\|(u,F(u))\|^p_{\cD^{2\eta,2\gamma,0}_{S,X}(I_T,\HH_{\alpha})}+\|(K_{T,\cdot}\,,K'_{T,\cdot})\|^p_{\cD^{2\eta,2\gamma,0}_{S,X,\back}(I_T,\LL^{-2\gamma})}\Big] \leq K_2 \Phi_2^p(u_0,T),$$
	for every initial condition $u_0\in\HH$, where $I_T = [T/2,T]$.
\end{itemize}
\end{ass}

The requirement for $\alpha$ is coming from the fact that for all $A \in \cA_k$ the inner product $\langle A(u_s), K_{T,s}\varphi\rangle$ should satisfy the assumptions of Proposition~\ref{forward_backward1} on spatial regularities and the fact that the assumption of Lemma~\ref{compos4} for $N$ should be satisfied ($N$ has to be a smooth function on the level of rough path regularity, which is $\alpha$ in this case).

Note that among all the assumptions we do not a priori assume that the solution is Malliavin differentiable, but this does follow from assumptions~\ref{ass3} and~\ref{ass1}, combined with Theorem~\ref{mal_diff}. 
Moreover, Definition~\ref{Mal_matr} of $\cM_T$ does \textit{not} coincide with the definition of the Malliavin matrix in general
(see \cite{Nual:06}), but it always makes sense whenever Jacobian is well-defined. In the particular case where $\XX$ is a Brownian motion,
we will see  in the proof of Theorem~\ref{densities} below that two definitions agree and 
our version of H\"ormander's theorem provides a statement for the usual Malliavin matrix.
With all these assumptions at hand we are ready to present the main result of this article.
\begin{thm} \label{main}
Let $T>0$ and let the noise $\XX \in \cC_g^\gamma$ satisfy Assumption~\ref{ass3}. Let $0\leq \delta < 1-\gamma$ and assume that $N$ and $F_i$ 
satisfy Assumption~\ref{ass2} for some orthogonal projection $\Pi :\HH \to \HH$, $a \in (0,1)$, $k \in \N_0$ and continuous function $\Lambda_a : \HH \to (0,\infty)$. Assume that $(u,F(u))$  solving~(\ref{e:rpde2}) satisfies Assumptions~\ref{ass1} and~\ref{ass4}. Then there exists a function $\Phi_T : \HH \to [0,\infty)$ such that for every $p\geq 1$ there exist a constant $C_p$ such that the operator 
$\cM_T$ defined in~\eqref{mal 7} satisfies the bound
\begin{equ} 
	\Prob(\inf_{\varphi \in \mathscr{S}_a}\langle\cM_T\varphi, \varphi \rangle \leq \varepsilon) \leq C_p \Phi_T^p(u_0) \varepsilon^p,
\end{equ}
for every initial condition $u_0$. Here $C_p$ is independent of the initial condition.
\end{thm}
\begin{proof}
	Fix $\varphi \in \mathscr{S}_a$, an initial condition $u_0 \in \HH$, and let $\cA_k$ and $Q_{k}$ be as in Assumption~\ref{ass2}. 
	Since $K_{T,T}\varphi = \varphi$ and, by Proposition~\ref{smoothing}, $u_T \in \HH_\infty$ almost surely, we have:
	\begin{equs}
		\Lambda_a(u_T) &\leq \langle \varphi, Q_{k}(u_T)\varphi\rangle \lesssim \max_{A \in \cA_k} \langle \varphi, A(u_T)\rangle^2 =  \max_{A \in \cA_k} \langle K_{T,T}\varphi, A(u_T)\rangle^2  \\ 
		&=\max_{A \in \cA_k} |Z_A^\varphi(T)|^2
		\leq \max_{A \in \cA_k} \|Z^\varphi_A\|^2_{\infty,I_T} \leq M^2_k R^{\,2q_k}_{T/2}(u_0) \langle \cM_T\varphi,\varphi\rangle^{2r_k},
	\end{equs}
	where $R_{T/2}(u_0)$ is defined by~\eqref{dominater} with $\alpha = (k-2\gamma+\delta)\vee 0$. This shows the existence of some $q,r,M >0$ independent of noise and initial condition such that for all $\varphi \in \mathscr{S}_a$
	$$\langle \cM_T\varphi,\varphi\rangle \geq M \Lambda^{r}_a(u_T) R^{-q}_{T/2}(u_0).$$
	Therefore we can use Markov's inequality to deduce for every $p\geq 1$:
	\begin{equs}
		\Prob&(\inf_{\varphi \in \mathscr{S}_a}\langle\cM_T\varphi, \varphi \rangle \leq \varepsilon) \leq \Prob(M \Lambda^{r}_a(u_T) R^{-q}_{T/2}(u_0) \leq \varepsilon)\\
		&\leq M^{-p}\varepsilon^{p} \E[\Lambda^{-pr}_a(u_T) R^{pq}_{T/2}(u_0)] \leq M^{-p}\varepsilon^p \E[\Lambda^{-2pr}_a(u_T)]^{1/2}\E[R^{2pq}_{T/2}(u_0)]^{1/2} \\
		&\lesssim \varepsilon^p \Phi_1^{pr}(u_0,T)\big(1 +\E[L^{-2pq}_\theta(X)+\varrho^{2pq}_\gamma(X)]+\Phi_2^{2pq}(u_0,T)\big)^{1/2}.
	\end{equs}
	Here  we used Cauchy-Schwarz in the third inequality and, in the last inequality, we used~\ref{ass4}. Finally from~\ref{ass3} the expectation $\E[L^{-2pq}_\theta(X)+\varrho^{2pq}_\gamma(X)]$ is always bounded by some constant $C_p$ and thus taking for instance 
	\begin{equ}
		\Phi_T(u) = \Phi^r_1(u,T)(1+\Phi^{2q}_2(u,T))\, ,
	\end{equ}
	gives the result.
\end{proof} 
We now give an extra condition on the solution $u$ which will guarantee the smoothness of the densities of finite-dimensional projections:
\begin{ass}\label{ass5}
In addition to previous assumptions assume that for all $T>0$ there exists a function $\Psi_T: \HH \to [0,\infty)$ such that for all $p\geq 1$ there exist $K_3$ with:
\begin{equ}
	\E\Big[\Phi^p_T(u_T)\Big] \leq K_3 \Psi_T^p(u_0),
\end{equ}
for every initial condition $u_0$ and function $\Phi$ from the Theorem~\ref{main}. 
\end{ass}

\begin{thm}\label{densities}
	Assume that the rough path $\XX = (B,\B^\Strat)$. Let $(u,u')$ be a global solution to~(\ref{e:rpde2}) like in Theorem~\ref{main}, and assume that it also satisfies~\ref{ass5}. Assume also that the image of the orthogonal projection $\Pi$ (from assumption~\ref{ass2}) is finite-dimensional. Then for all $t>0$ the law of $\Pi u_t$ has a smooth density with respect to Lebesgue measure on $\Pi(\HH)$.
\end{thm}
\begin{proof}
Since~\eqref{e:rpde2} is driven by the Stratonovich lift of Brownian motion to the space of rough paths, it follows from Proposition~\ref{spde=rpde}
that $u$ coincides almost surely with the solution to the corresponding It\^o SPDE. It follows from the
growth assumption~\ref{ass4} that we can derive an SPDE for the Malliavin derivative of any order and the Jacobian of any order. 
Moreover, using Duhamel's formula similar to \eqref{mal 6} for higher order Malliavin derivatives and moment assumptions from~\ref{ass4} 
one concludes that for every  $t > 0$, $u_t$ belongs to the space $\mathbb{D}^{\infty}$ of Malliavin smooth random variables whose 
Malliavin derivatives of all orders have moments of all orders. (For more details in the additive case see 
\cite[Thm~8.1]{hypoelipticity}.) Since $u_t$ is Malliavin smooth and $\Pi$ is a bounded linear map (hence smooth), we 
deduce that $\Pi u_t$ is also Malliavin smooth. By \cite[Chap.~2]{Nual:06}  it remains to show that the Malliavin matrix of $\Pi u_t$ 
denoted by $\cM^\Pi_t$ is almost surely invertible and has moments of all orders. Just for notational convenience we will consider 
$\Pi u_{2t}$ instead of $\Pi u_t$. First we can view the element $u_{2t}$ as an element of the probability space with 
Gaussian structure induced by the increments of $B$ over the interval $[t,2t]$ and, as in~\cite[Chap.~1]{Nual:06}, we view increments of $B$ 
over $[0,t]$ as irrelevant randomness. This shows that almost surely 
\begin{equ}[e:boundMall]
\cM^\Pi_{2t} = \Pi \cM_{t}(u_{0})\Pi + \Pi \cM_{t,2t}(u_{t})\Pi \ge \Pi \cM_{t,2t}(u_{t})\Pi\;,
\end{equ}
where $\cM_{t,2t}(u_{t})$ is defined like in~\eqref{mal 7} but over the interval $[t,2t]$, and we treat $u_{t}$ as 
an ``initial condition'' at time $t$. 
Recall that $\FF_s$ is the natural filtration of the underlying Brownian motion, denote $\KK = \Pi(\HH)$ then we have:
	\begin{equs}
		\Prob(\inf_{\varphi \in \KK;\, \|\varphi\| = 1}\langle\Pi\cM_{t,2t}&(u_{t})\Pi\varphi, \varphi \rangle \leq \varepsilon\vert \FF_{t}) = \Prob(\inf_{\varphi \in \KK;\, \|\varphi\| = 1}\langle\cM_{t,2t}(u_{t})\varphi, \varphi \rangle \leq \varepsilon\vert \FF_{t})\\
		&\leq\Prob(\inf_{\varphi \in  \mathscr{S}_a}\langle\Pi\cM_{t,2t}(u_{t})\Pi\varphi, \varphi \rangle \leq \varepsilon\vert \FF_{t}) \leq C_p \Phi_t^p(u_t) \varepsilon^p.
	\end{equs} 
	We have used above that $\Pi \varphi =\varphi$ and therefore for every $a \in (0,1)$
	$$\{\varphi \in \KK;\, \|\varphi\| = 1\} \subset \{\varphi \in \HH;\, \|\varphi\| = 1, \|\Pi \varphi\| \geq a\} =  \mathscr{S}_a.$$
	Moreover in the last inequality we have used the Markov property of the solution. Taking the expectation and using the
	bound \eqref{e:boundMall}, we see that 
	\begin{equs}
		\Prob(\inf_{\varphi \in \KK;\, \|\varphi\| = 1}\langle\cM^\Pi_{2t}\varphi, \varphi \rangle &\leq \varepsilon) \leq C_p\varepsilon^p\E\Big[\Phi^p_t(u_t)\Big] \leq C_pK_3 \Psi_t^p(u_0)\varepsilon^p.
	\end{equs}
	This guarantees the invertibility of Malliavin matrix $\cM^\Pi_{2t}$ on $\KK$ and that $(\cM^\Pi_{2t})^{-1}$ has moments of all orders, thus finishing the proof.
\end{proof}

\section{Examples} \label{examples}

As mentioned before, we will restrict ourselves to the Brownian rough path in all examples since 
Assumption~\ref{ass1} is a priori not known to hold for more general Gaussian rough paths. 
We will focus on equations driven by Brownian motion that have 
global solutions as well as a Jacobian, which will imply Assumption~\ref{ass1} from 
Section~\ref{SPDEs}. We will then show examples of noises for which H\"{o}rmander's condition, 
Assumption~\ref{ass2} is satisfied. The moment bounds for the rough path norms of solution and Jacobian 
for Assumption~\ref{ass4} part (2) might not be easy to obtain in general and require a closer look as a separate 
problem on its own. We decide to postpone the study of such moments but refer the reader to  
\cite{friz2013integrability} where this question was answered for the rough SDE case.

We want to point out that the present work is indeed a generalisation of the additive case from \cite{hypoelipticity} since 
Assumption~\ref{ass2} is a slight modification of Assumption~C.2 from that article. 
Consider an equation which has both an additive and a multiplicative noise:
\begin{equ} 
	du_t = Lu_tdt + N(u_t)dt + \sum_{i=1}^kg_i dB^i_t \;+ \sum_{i=k+1}^dF_i(u_t)dB^i_t, \label{ex1}
\end{equ}
for $g_i \in \HH_\infty$. Note that as a Fr\'echet derivative $Dg_i = 0$ thus the Lie brackets $[g_i,g_j]=0$ and there is no contribution from Lie brackets of this additive part. Only an interplay of $[L+N,g_i]$, $[L+N,F_i]$, $[F_j,g_i]$, $[F_j,F_i]$ and of higher order Lie brackets contributes to Assumption~\ref{ass2}. In particular,
if \eqref{ex1} satisfies Assumption~\ref{ass2} for $F_i = 0$, then it also satisfies it for $F_i \neq 0$.

We now give a simple criteria for when the Assumption~\ref{ass2} is satisfied and moreover the function $\Lambda_a$ can be taken constant.
\begin{prop} \label{prop ex}
 Let $L,N,F_i$ be as in Theorem~\ref{main}. Define the set $\cA \subset \HH$ of all possible constant directions created by the Lie 
 brackets of the above vector fields, namely 
 $$\cA = \bigcup_{k\geq 0} \{ A \in \cA_k:\;\forall u\, \in \HH_k\, ,\, A(u)=A(0) \in \HH\}.$$
If the linear span of $\cA$ is dense in $\HH$, then for every finite rank orthogonal projection $\Pi : \HH \to \HH$ and every $a \in (0,1)$ H\"{o}rmander's condition~\ref{ass2} is satisfied for some $k$. Moreover, the function $\Lambda_a$ can be chosen as a constant 
depending on $\Pi$ and $a$. As a consequence, part $(1)$ of the condition~\ref{ass4} is trivially satisfied too. 
\end{prop}
A proof of this statement can be found in \cite[Lem.~8.3]{hypoelipticity}. 
This criterion is 
the one that we are going to use in our next examples.
\begin{rem}
If in addition \eqref{e:rpde2} is driven by Brownian motion and solutions $u_t$ satisfy Assumption~\ref{ass5}, then the above proposition and Theorem~\ref{densities} guarantee that $\Pi u_t$ has a smooth density with respect to Lebesgue measure for every surjective linear map $\Pi\colon \HH \to \R^n$.
\end{rem}

\subsection{Stochastic Navier-Stokes equation in 2-d}

The Navier-Stokes equation describes the time evolution of incompressible fluid and is given by
$$\partial_t u_t(x)  + (u_t(x)\cdot \nabla) u_t(x) = \Delta u_t(x) -\nabla p_t(x) + \xi(t,x,u_t(x)),\quad  \nabla \cdot u_t = 0.$$
For $u_t(x) \in \R^2$ is a velocity field, $p(t,x)$ is a pressure, and $\xi$ is a noise term describing an external force acting on the fluid. We decide to work with the equation with spatial variable lying on the two dimensional sphere $x \in \mathbb{S}^2$. Moreover because of divergence free assumption we can work with the vorticity formulation of this equation which can be written as: 
\begin{equ} 
	dw_t = \Delta w_t dt + N(w_t)dt + \sum_{i=1}^k f_idB^i_t\, + \sum_{j=1}^n \big(f_{j+k} + w_t g_{j})\, dB^{k+j}_t,\quad w_0 \in L^2(\mathbb{S}^2,\R).	\label{ex2}
\end{equ}
Here $(B^i_t)_{1\leq i \leq n+k}$ are mutually independent Brownian motions on $\R$, our noise is a mixture of the additive an linear multiplicative noise. $\Delta = - \nabla^{*}\nabla$ is the negative Bochner Laplacian on $\mathbb{S}^2$. The non-linearity $N(w)$ is given by $N(w) = B(w,w)$ for the symmetric operator
$$B(v,w) = \frac{1}{2}(\nabla \cdot (v K w) +\nabla \cdot (w K v) ).$$
Operator $K$ is an operator that reconstructs velocity field from the vorticity:
$$u = K w = - \curl \Delta^{-1} w$$
(See more details on the derivation of these equation and their further study in \cite{temam1993inertial}.) Here the Hilbert space $\HH = L^2(\mathbb{S}^2,\R)$ and the interpolation spaces generated by the Laplacian will be $\HH_\alpha = H^{2\alpha}(\mathbb{S}^2,\R)$, 
the usual Sobolev spaces on the sphere. We assume that all the functions $g_i,f_i \in H^\infty(\mathbb{S}^2,\R)$. Note that later for the reaction-diffusion equation we will consider also a polynomial noise, but here if we take polynomials of higher order than just linear 
it is expected that the blow up created by the diffusion part will not be compensated by the drift part and so it is hopeless to get any global bounds on the solution. In \cite{hypoelipticity}, the authors show that indeed such $N$ is a smooth function 
$\HH_\alpha \to \HH_{\alpha-\delta}$ for any $\delta > 1/2$ and $\alpha \geq 0$. Since Brownian motion can be lifted to a 
Stratonovich rough path almost surely in $\cC^\gamma$ for every $\gamma < 1/2$ we get in particular that $1-\gamma > 1/2$ and so we 
can indeed take $1/2 < \delta < 1- \gamma$ so that non-linearity $N$ falls into our framework. Also for any $g,f \in \HH_\infty$ an affine function $F(u) = u g +f$ is trivially smooth as a function $\HH_\alpha \to \HH_\alpha$ for every $\alpha \in \R$ and sends bounded subsets of $\HH_\alpha$ to bounded subsets of $\HH_\alpha$.

Note that for affine functions of the type $F_i(u) = u g_i +f_i$ we have 
$$[F_i,F_j](u) = (ug_i+f_i)g_j - (ug_j+f_j)g_i = f_ig_j - f_jg_i.$$ 
These Lie brackets of affine functions between each other will therefore produce additional constant directions for the spread of the noise. We define the set of functions recursively:
\begin{equs}
	A_0 &= \{f_i: 1\leq i \leq k\}\cup \{f_{i+k}g_j - f_{j+k}g_i: 1\leq i,j \leq n\},\\  
	A_{k+1} &= \{B(g,h),\, g_j h : g,h\in A_k, 1 \leq j \leq n\}.
\end{equs}
Here the terms $B(g,h)$ will arise from second order Lie brackets $[[\Delta+N,g],h] = B(g,h)$ for any $g,h \in \HH_\infty$. Terms $h g_j$ will arise from Lie brackets of constant part with the affine part of the noise: $[h,F_j](u) = hg_j$ for $F_j(u) = ug_j+f_{j+k}$. Clearly, the linear span of $\bigcup_{k\geq 0} A_k$ is contained in the linear span of $\cA$ from Proposition~\ref{prop ex}. In particular if linear span of $\bigcup_{k\geq 0} A_k$ is dense in $\HH$ then condition~\ref{ass2} for~\eqref{ex2} is satisfied for every finite rank orthogonal projection $\Pi$.

\subsection{Reaction-Diffusion equations and Ginzburg-Landau equation}

The reaction diffusion equation in $m$ dimensions are equations of the form
\begin{equ} 
	du_t = \Delta u\,dt + f \circ u_t \,dt + \sum_{i=1}^n p_i(u_t) \,dB^i_t, \label{ex3}
\end{equ}
with $u_t :D \to \R^l$ where $D$ is either an $m$-dimensional torus or a nice $m$-dimensional domain (compact smooth $m$-dimensional Riemannian manifold or bounded open subset of $\R^m$ with smooth boundary), then $\Delta$ is a Laplace or Laplace-Beltrami operator on $D$. 
Here, the non-linearity is given by the Nemitsky operator of composition with a polynomial $f$. 
To guarantee that this operator is smooth, the Hilbert space $\HH$ must be an algebra, which can be satisfied by taking $\HH$ to be a 
Sobolev space of high enough order: $H^k(D,\R^l)$ is an algebra if $k > m/2$. Now for $p_i$ we can take also polynomials with coefficients being functions in $\HH_\infty$: i.e. can have 
$$p_i(u) = \sum_{j=1}^{k_i}g^i_j f_j^i\circ u.$$
For polynomials $f_j^i : \R^l \to \R^l$ and constants $g_j^i \in \HH_\infty$. Recall that we must require $p_i$ to be smooth functions 
$\HH_\alpha \to \HH_\alpha$ for all $\alpha \geq -2\gamma$. If the degree of any of these polynomials $f_j^i$ is greater than one, then for the $p_i$ to be of this kind we must require $\HH_{-2\gamma}$ to be an algebra, where $\gamma$ is a regularity of Brownian rough path. 
Thus taking $\HH = H^k(D,\R^l)$ we have that $\HH_{-2\gamma} = H^{k-4\gamma}$ and therefore we have the requirement on $k$ to satisfy 
$k > m/2 +4\gamma$. This means that it will be more beneficial to take as low as possible rough path regularity $\gamma$. For Brownian 
motion, we can take $\gamma$ arbitrarily close to $1/3$ so that all the above theory would be still true. For instance 
for $m = 1$, one can take $k = 2$ but for both $m = 2$ and $m=3$ one can take $k=3$. The higher the degree of differentiability $k$ 
is taken, the more difficult it might be to obtain the a priori bounds~\ref{ass4}. Note that if $\HH_{\alpha}$ is an algebra for $\alpha \geq -2\gamma$ then clearly any polynomial sends bounded subsets of $\HH_\alpha$ bounded subsets of $\HH_\alpha$.

An assumption on $f$ sufficient to avoid explosion is  that, if we write $f  = \sum_{k=0}^d f_k$ where 
$f_k$ is a k-linear map $\R^l$ into itself,  then
$$\langle f_d(u,\ldots ,u,v),\,v\rangle_{\R^l} < 0,$$
for every $u,v \in \R^l\backslash \{0\}$. 
This will guarantee global existence at least in the additive case (see \cite[Sec.~8.3]{hypoelipticity}).

If we stick to the case when $l = 1$ then polynomials will produce the following Lie brackets: for $p,q \in \N_0$ such that $p+q \geq 1$ and $g,h \in \HH_\infty$
$$[g u^p, h u^q](u) = (q-p) gh u^{p+q-1}.$$
Note that new constant directions will arise from the Lie bracket of a constant and a linear term, but for instance second 
iterated Lie bracket with a square also creates a new constant direction:
$$[f,[g,hu^2]] = 2 f g h.$$
In general, the $k$th iterated Lie bracket of polynomials of total degree $k$ produces a constant direction. 
Potential appearance of the higher order polynomials can eventually create enough constant directions so that their linear span 
is dense in $\HH$. We do not present a general interplay of all such Lie brackets for general polynomials, since it might be quite 
cumbersome and in the actual example it might be simpler to compute all the Lie brackets by hand. Nevertheless we explicitly 
provide an example below.

\subsubsection{Ginzburg-Landau equation}
This is the equation given by
\begin{equs} 
	du_t(x) = \Delta u_t dt + (u_t(x) &- u^3_t(x))dt + \sum_{i=1}^k f_i(x)dB^i_t \\
	&+ \sum_{j=1}^n \big(f_{j+k}(x) + u_t(x) g_{j}(x) \big)dB^{k+j}_t,\quad u_0 \in H^1(\mathbb{T},\R).
\end{equs}
Here $f_i,g_i \in C^\infty(\mathbb{T},\R)$, we can take $\HH = H^1$ since the noise is linear and thus we only require $\HH$ to be an algebra which $H^1$ is in one dimension. Similarly to the Navier-Stokes example, define:
\begin{equs}
A_0 &= \{f_i: 1\leq i \leq k\}\cup \{f_{i+k}g_j - f_{j+k}g_i: 1\leq i,j \leq n\},\\  
A_{k+1} &= \{g_j h : h\in A_k, 1 \leq j \leq n\} \cup \{h_1h_2h_3 : h_i \in A_k\}.
\end{equs}
If $\bigcup_{k} A_k$ is dense in $\HH$, then condition~\ref{ass2} for~\eqref{ex2} is satisfied. In particular only two instances of noise is enough and the following equation satisfies the assumption~\ref{ass2}:
\begin{equs}
	du_t(x) = \Delta u_t dt + (u_t(x) - u^3_t(x))dt &+ \big(\sin(x)+u_t(x)\cos(x)\big)dB^1_t\\
	&+ \big(\cos(x)-u_t(x)\sin(x)\big)dB^2_t.
\end{equs}
Since if we call $F_1(u) = \sin(x)+u(x)\cos(x)$ and $F_2(u) = \cos(x)-u(x)\sin(x)$
then:
$$[F_1,F_2](u) = \cos^2(x)+\sin^2(x) = 1,\;$$ 
$$[[F_1,F_2],F_1] = \cos(x),\quad [[F_1,F_2],F_2] = -\sin(x),$$
$$[[[F_1,F_2],F_1],F_1] = \cos^2(x),\quad [[[F_1,F_2],F_2],F_1] = -\sin(x)\cos(x)\ldots $$
Proceeding similarly we see that we can produce any term of the form $\sin^k(x) \cos^\ell(x)$, which creates a basis 
for $H^1(\mathbb{T},\R)$ and thus H\"{o}rmander's condition is trivially satisfied.
\bibliographystyle{Martin}
\bibliography{bib_list.bib}

\end{document}